\documentclass{amsart}
\usepackage{amsmath,amsthm,indentfirst}
\usepackage{amssymb}
\usepackage{amsfonts,dsfont}
\usepackage{a4}
\usepackage{amscd}
\usepackage{tikz}
\usetikzlibrary{calc}
\usepackage{pgf}
\usetikzlibrary{patterns}
\usepackage{xifthen}
\usepackage[final]{showlabels}
\usepackage[latin1]{inputenc}
\usepackage{url}
\theoremstyle{plain}
\newtheorem{thm}{Theorem}
\newtheorem{cor}[thm]{Corollary}
\newtheorem{lem}[thm]{Lemma}
\newtheorem{prop}[thm]{Proposition}

\newtheorem{question}[thm]{Question}

\newtheorem{conjecture}[thm]{Conjecture}

\newtheorem{definition}[thm]{Definition}

\theoremstyle{remark}
\newtheorem{remark}[thm]{Remark}

\usepackage{enumitem}
\setenumerate[0]{label=(\alph*)}
\def\bbz{\mathbb{Z}}
\def\bbq{\mathbb{Q}}

\def\bbr{\mathbb{R}}

\def\bbc{\mathbb{C}}

\def\bbe{\mathbb{E}}

\def\bbp{\mathbb{P}}

\def\ocal{\mathcal{O}}
\def\ocal{\mathcal{O}}

\def\pcal{\mathcal{P}}

\def\ncal{\mathcal{N}}
\def\afr{\mathfrak{a}}
\def\Bfr{\mathfrak{B}}

\def\pfr{\mathfrak{p}}
\def\Afr{\mathfrak{A}}

\def\f{\mathfrak{f}}

\def\xbf{\mathbf{x}}

\def\abar{\overline{A}}
\def\bbar{\overline{B}}
\def\sbar{\overline{S}}
\DeclareMathOperator\SL{SL}

\DeclareMathOperator\gr{gr}
\newcommand{\wt}[1]{\widetilde{#1}}
\newcommand{\wh}[1]{\widehat{#1}}
\def\h{\hspace{1mm}}

\def\vare{\varepsilon}
\def\be{\begin{equation}}
\def\ee{\end{equation}}
\def\one{\mathds{1}}

\def\abar{\overline{A}}
\def\bbar{\overline{B}}
\newcommand{\gen}[2]{\langle #1 \rangle_{#2}}

\begin{document}
\title{Sum-product phenomena: $\pfr$-adic case.}
\author{Alireza Salehi Golsefidy}
\address{Mathematics Dept, University of California, San Diego, CA 92093-0112}
\email{golsefidy@ucsd.edu}
\thanks{A. S-G. was partially supported by the NSF grant DMS-1303121, the A. P. Sloan Research Fellowship. Parts of this work was done when  I was visiting Isaac Newton Institute and the MSRI, and I would like to thank both of these institutes for their hospitality.}
\date{\today}
\begin{abstract}
The sum-product phenomena over a finite extension $K$ of $\bbq_p$ is explored. The main feature of the results is the fact that the implied constants are independent $p$.
\end{abstract}
\maketitle
\section{Introduction}
\subsection{Bounded generation phenomena and the general approach towards proving them.}\label{ss:introduction}
Let $R$ be a unital commutative ring. Given two subsets $A$ and $B$ of $R$, we define the {\em sum set} 
\[
A+B:=\{a+b|\h a\in A, b\in B\},
\]
the {\em difference set}
\[
A-B:=\{a-b|\h a\in A, b\in B\},
\]
and the {\em product set}
\[
A\cdot B:=\{ab|\h a\in A, b\in B\}.
\]
We also define the $l$-fold sum set and the $l$-fold product set of a subset $A$ of $R$:
\be\label{e:sum-product-set}
\textstyle
\sum_l A:=\{a_1+\cdots+a_l|\h a_i\in A\}\h\h{\rm and }\h\h
\prod_l A:=\{a_1\cdots a_l|\h a_i\in A\}.
\ee
Starting with a subset $A$ of the ring $R$, the subring generated by $A$ is denoted by $\langle A\rangle$; that means $\langle A\rangle$ is the smallest subset of $R$ which contains $A$ and $\langle A\rangle \cdot \langle A\rangle\subseteq \langle A\rangle$ and $\langle A\rangle-\langle A\rangle=\langle A\rangle$. In order to have a measurement on how fast $A$ generates the subring $\langle A\rangle$, we define 
\be\label{e:ring-generating-set}
\textstyle
\langle A\rangle_l:=\sum_l\prod_l A-\sum_l\prod_l A;
\ee
and so, if $0,1\in A$, then $\langle A\rangle=\bigcup_{l=1}^{\infty}\langle A\rangle_l$.

The remarkable {\em sum-product phenomenon} in a finite field, proved by Bourgain, Katz, and Tao~\cite{BKT} (see Lemma~\ref{l:BoundedGenerationFiniteField}), implies that: for any $\vare>0$ there is a positive integer $C$ such that for a subset $A$ of a finite field $\f$, if $\log |A|\ge \vare \log |\f|$ and $0,1\in A$, then $\langle A\rangle_C=\langle A\rangle$. Based on this result, Helfgott~\cite{Hel1} proved a {\em product} theorem in $\SL_2(\f_p)$ where $\f_p$ is the finite field of prime order $p$; this result implies that for any $\vare>0$ there is a positive integer $C$ such that, for any symmetric generating set $A$ of $\SL_2(\f_p)$ of cardinality at least $|\SL_2(\f_p)|^{\vare}$, we have $\prod_C A=\SL_2(\f_p)$. One can view these results as examples of {\em bounded
generation phenomena}. Later the product theorem had been extended to all the finite simple groups of Lie type (see~\cite{Hel2} for $\SL_3(\f_q)$ case and either \cite{BGT} or \cite{PS} for the general case); and this product theorem implies a bounded generation result for such groups. Within the proof of the mentioned sum-product result for finite fields, the {\em vector space} structure of such fields had been used and it was proved that for any $\vare>0$ there is a positive integer $C$ such that, for any subset $A$ of a finite field $\f$, if $\log |A|\ge \vare \log |\f|$, then there are $\alpha_1,\ldots,\alpha_C\in \f$ such that $\alpha_1 A
+\cdots+\alpha_CA=\f$; this result can be viewed as yet another example of bounded generation phenomena. Based on these examples, one can philosophize and vaguely formulate a {\em na\"ive} bounded generation phenomenon that says: if a finite algebraic structure $\Afr$ is {\em rich enough}, then any {\em generic} subset $A$ of $\Afr$ of cardinality at least $\vare\log|\Afr|$ generates $\Afr$ in $C$ steps, where $C$ only depends on $\vare$. There is, however, one algebraic obstruction to the proposed bounded generation phenomenon: having a bounded generation for $\Afr$ passes to all of its factors; that means if $\pi:\Afr\rightarrow \pi(\Afr)$ is a surjective homomorphism of $\Afr$ and $A$ generates $\Afr$ in $C$ steps, then $\pi(\Afr)$ can be generated by $\pi(A)$ in $C$ steps as well. So one would need to have $\log|\pi(A)|\ge \vare' \log|\pi(\Afr)|$ for some $\vare'>0$ independent of $\pi$. Notice that this obstruction does not show up in a finite field $\f$ or $\SL_2(\f_p)$ as they do not have a lot of factors. One can see the subtlety of this issue already in the ring $\bbz/p^n\bbz$. This ring has $n$ factors: for any $1\le i\le n$, let $\pi_{p^i}:\bbz/p^n\bbz\rightarrow \bbz/p^i\bbz$ be the natural quotient map. And so for a subset $A$ of $\bbz/p^n\bbz$ that contains $0,1$ and satisfies $\log |A|\ge \vare \log |\bbz/p^n\bbz|$, one cannot expect to get $\langle A\rangle_C=\bbz/p^n\bbz$ for a constant $C$ which depends only on $\vare$ unless $\log |\pi_{p^i}(A)|\ge \vare' \log |\bbz/p^i\bbz|$ for any $1\le i\le n$ and some $\vare':=\vare'(\vare)>0$. By fixing $p$ and varying $n$, we get the following $p$-adic interpretation of the above mentioned case: for a subset $A$ of the ring $\bbz_p$ of $p$-adic integers and a number $0<\delta<1$, let $\ncal_{\delta}(A)$ be the smallest positive integer $n$ such that there are balls $B_1, \ldots, B_n$ of radius $\delta$ (with respect to the standard metric on $\bbz_p$) such that $A\subseteq B_1\cup \cdots \cup B_n$. Notice that $\ncal_{p^{-i}}(A)=|\pi_{p^i}(A)|$ as any ball of radius $p^{-i}$ is an additive coset of $p^{i}\bbz_p$ in $\bbz_p$; and so $\log |\pi_{p^i}(A)|\ge \vare' \log|\pi_{p^i}(\bbz_p)|$ is equivalent to $\log \ncal_{p^{-i}}(A)/\log (p^i)\ge \vare'$. On the other hand, let us recall that the lower box dimension of a subset $A$ of a metric space is defined to be $\liminf_{\delta \rightarrow 0^+} \log \ncal_{\delta}(A)/\log(1/\delta)$; since for us the analysis in a given scale $\delta$ is important, we call $\log \ncal_{\delta}(A)/\log(1/\delta)$ {\em the box dimension of $A$ at the scale $\delta$}. Therefore, for a subset $A$ of $\bbz_p$ with lower box dimension $\vare>0$, we have that, if $n$ is a large enough integer depending on $\vare$, then the box dimension $\log \ncal_{p^n}(A)/\log(p^n)$ of $A$ at the scale $p^{-n}$ is at least $\vare/2$. Now using a {\em regularization} argument (see~\cite[Section 4]{Bor} or Lemma~\ref{l:RegularSubset}), one can find a subset $A'$ of $(A-A)\cap p^{n_0}\bbz_p$ such that $\log \ncal_{p^{-n}}(A')/\log(p^{n-n_0})\ge \vare/4$ for any $n\ge n_0(\vare)$. So after rescaling $A'$ one can apply the proposed bounded generation and get that $\langle A\rangle_C$ contains a $\bbz_p$-segment of size $L$ where both $C$ and $L$ depend only on the lower box dimension $\vare$. This type of {\em bounded generation} seems to be the {\em right} property to look for in various cases; that means if $A$ is a subset with lower box dimension $\vare$ of an algebraic structure $\Afr$ which is {\em rich enough}, then $A$ generates a {\em large substructure} $\Bfr$ of $\Afr$ in $C$ steps, where both the {\em largeness} of $\Bfr$ and the positive integer $C$ are supposed to depend only on $\vare$. For instance Bourgain's proof (see \cite{Bou-RingConjecture}) of the Katz-Tao discretized ring conjecture (see~\cite{KaTa-Discretized}) implies this phenomenon for the ring $\bbr$.  

To prove a bounded generation result, using in part tools from additive combinatorics (for instance see the influential article~\cite{Gow-SzemerediTheorem} and the nice book on this subject~\cite{TV-AdditiveCombinatorics}; the method of the proof of the main theorem of \cite{EdMi-ErdosVolkmann} has been used in the subsequence articles on this subject, too), one often proves the weighted version; that means one starts with the probability counting measure $\pcal_A$ on the set $A$ and then consider the push-forward of $\pcal_A\times \cdots \times \pcal_A$ under the map $f_C:\prod_{i=1}^{\Theta(C^{\Theta(1)})} \Afr \rightarrow \Afr$, where $f_C$ is given by the $C$ step generation in $\Afr$ by its algebraic operations. For instance, when $\Afr=G$ is a group, $f_C:\prod_{i=1}^{2C} G\rightarrow G$, $f_C(g_1,\ldots,g_{2C}):=g_1g_2^{-1}\cdots g_{2C-1}g_{2C}^{-1}$ and we start with a set that contains the identity element; when $\Afr=R$ is a ring, $f_C:\prod_{i=1}^{C}\prod_{j=1}^{C} R\times R \rightarrow R$, $f_C(a_{ij},b_{ij}):=\sum_{i=1}^C\prod_{j=1}^Ca_{ij}-\sum_{i=1}^C\prod_{j=1}^Cb_{ij}$ and we start with a set that contains the zero and the identity elements. In this setting, the {\em na\"{i}ve} bounded generation  implies that the support of $\mu_{A}^{[C]_{\Afr}}:=f_C(\pcal_A\times\cdots\times \pcal_A)$ is the entire $\Afr$. In the {\em weighted} version, one would like to show that $\mu_A^{[C]_{\Afr}}$ is close to the equidistribution on $\Afr$; that means $\mu_A^{[C]_{\Afr}}$ is close to the probability counting measure on $\Afr$ for some positive integer $C$ which only depends on the {\em box dimension} of $A$. To get such a result, one often imposes additional assumptions on the set $A$ and proves that $\mu_A^{[O(1)]}$ is {\em substantially more distributed} compared to $\pcal_A$ for some positive integer $O(1)$ which only depends on $\Afr$. 

There are many ways to quantify how well a probability measure $\mu$ is distributed. One way is to use {\em Fourier analysis}; for instance in the abelian setting, an upper bound on the values of the Fourier transform $\wh{\mu}$ of $\mu$ gives us a way to say how well $\mu$ is distributed. This means one needs to get some cancellations in certain {\em exponential sums} (for instance see ~\cite{Cha-PolynomialFrieman,Bor,Bou-MordellRevisited}, the appendix of \cite{BG}). In a metric (not necessarily abelian) setting, this can be interpreted as saying that one needs to get an upper bound on $\mu\ast f$ where $f$ is a function which lives in a {\em scale $\delta$} (this roughly means $f$ is almost constant in balls of radius $\delta^2$ and almost orthogonal to the characteristic functions of balls of radius $\delta^{1/2}$) (for instance see~\cite{BouGam-SU(d),BouIoaGol-LocalSpectralGap}). Another way of measuring how well a measure $\mu$ is distributed is using its {\em entropy} $H(\mu)$ (for instance see~\cite{Rud-x2x3,Joh-CircleInvariant,LinMeiPer-ConvolutionCircle}). In this note following a work of Lindenstrauss and Varj\'{u}, we use the entropy approach to get the desired bounded generation result, which will be explained in the next section. 

It should be pointed out that bounded generation results in rings, such as finite fields $\f_q$, $\bbr$, or $\bbz_p$, have been playing an indispensable role in proving bounded generation results in groups (for instance see~\cite{Hel1, Hel2, BGT, PS} for the case of finite simple groups of Lie type,~\cite{BouGam-SU(2),BouGam-SU(d), dS-Product} for the case of compact simple Lie groups, and 
~\cite{BG} for the case of $\SL_n(\bbz_p)$). Using the results of this note, first a bounded generation result for semisimple $p$-adic analytic groups is proved in \cite{SG1} and then in \cite{SG2} this result is extended to the case of perfect $p$-adic analytic groups with abelian unipotent radical. Furthermore these results are {\em uniform} on the prime $p$; in the sense that the implied constants do not depend on $p$. Using such bounded generation results, the $p$-adic case of {\em super-approximation} property is proved (we refer the interested reader to the mentioned articles for the precise formulation of these results).           

\subsection{Main results.}
One of the main results of this note is the following {\em bounded generation} result for a characteristic zero non-Archimedean local field $K$; the importance of this result is on the fact that the implied constants are {\em independent of the characteristic of the residue field of $K$}. 

\begin{thm}\label{t:ThickSegment}
Suppose $0<\vare\ll 1$, $d$ is a positive integer, and $N\gg_{d,\vare} 1$ is a positive integer. Then there are $0<\delta:=\delta(\vare,d)$, and positive integer $C:=C(\vare,d)$, such that for any finite extension $K$ of $\bbq_p$ with degree $[K:\bbq_p]\le d$ the following holds: let $\ocal$ be the ring of integers of $K$, and $\pfr$ be a uniformizing element of $K$. Suppose $A\subseteq \ocal$ such that
\[
|\pi_{\pfr^N}(A)|\ge |\pi_{\pfr}(\ocal)|^{N\vare},
\]
where $\pi_{\pfr^N}:\ocal\rightarrow \ocal/\pfr^N\ocal$ is the canonical quotient map. Then there are positive integers $N_1$ and $N_2$, $a\in \ocal$, and a subfield $K_0$ of $K$ with ring of integers $\ocal_0$ such that  
\begin{align}
\label{eq:Thickness}
\lfloor N\delta\rfloor +N_1
& \le N_2\le NC,
& 
\text{ (Scale and thickness) }
\\
\notag
\pi_{\pfr^{N_2}}(\ocal_0 a)
&
\subseteq \pi_{\pfr^{N_2}}(\gen{A}{C}), 
\hspace{1cm}
v_{\pfr}(a)=N_1,
&
\text{ (Bounded generation) }
\\
\notag
|\pi_p(\ocal_0)|
&
\ge |\pi_p(\ocal)|^{\vare/4}.
&
\text{ (Box dimension control) }
\end{align}
where $\gen{A}{C}$ is defined as in (\ref{e:ring-generating-set}).
\end{thm}
As there are many parameters in Theorem~\ref{t:ThickSegment}, the reader might find the following rough description of the parameters useful. 

Think about $\vare$ as a lower bound for the box dimension of $A$ at scale $|\pfr^N|$: recall that the smallest number of balls of radius $\delta$ which cover $A$ is denoted by $\ncal_{\delta}(A)$; and so $\ncal_{|\pfr^{N}|}=|\pi_{\pfr^N}(A)|$ and we have
\[
\log \ncal_{|\pfr^{N}|}(A)/\log (1/|\pfr^{N}|)\ge \vare.
\]
Then Theorem~\ref{t:ThickSegment} provides us a {\em large subring of integers} $\ocal_0$ and a lower bound $C$ for the {\em number of steps} needed in order to get an $\ocal_0$ segment of {\em length} $|\pfr^{N_1}|$ at the {\em scale} $|\pfr^{N_2}|$; of course the significance of this statement is on the fact that the length $|\pfr^{N_1}|$ is {\em much larger} than the scale $|\pfr^{N_2}|$. For instance the inequality in \eqref{eq:Thickness} implies that the box dimension of $\pfr^{N_1}\ocal_0$ at scale $|\pfr^{N_2}|$ is at least $\vare\delta/4C$ as one can see in the following equation:
\be\label{e:HowMuchThick}
\frac{\log \ncal_{|\pfr^{N_2}|}(\pfr^{N_1}\ocal_0)}{\log (1/|\pfr^{N_2}|)}=
\frac{\log_p |\pi_{\pfr^{N_2-N_1}}(\ocal_0)|}{N_2[\f:\f_p]}
\geq \frac{(N_2-N_1)\log_p |\pi_p(\ocal_0)|}{e(\pfr/p) N_2f(\pfr/p)}
=\frac{(N_2-N_1)\log_p |\pi_p(\ocal_0)|}{N_2\log_p(\pi_p(\ocal))}
\ge\frac{\vare\delta}{4C},
\ee 
where $e(\pfr/p)$ is the ramification index and $f(\pfr/p)$ is the residue degree of $\pfr$ in the extension $K/\bbq_p$ (see~\cite[Page 14]{Ser-LocalFields}). So $\delta$ measures the {\em thickness} of the attained $\ocal_0$-segment. 

It should be pointed out that in \cite[Proposition 3.3]{BG}, a weaker form of Theorem~\ref{t:ThickSegment} is proved where it is assumed that the extension $K/\bbq_p$ is {\em not widely ramified} and the {\em characteristic of its residue field} is a fixed prime; that means the implied constants depend on the characteristic of the residue field as well.

As in the proof of  \cite[Proposition 3.1]{BG}, using induction on the rank, one can extend Theorem~\ref{t:ThickSegment} from the rank 1 case to the rank $d_0$, where $d_0$ is a {\em fixed} positive integer. 
\begin{cor}\label{c:ThickSegment}
For any $0<\vare\ll 1$ and positive integers $d_0$ and $d$, there are $0<\delta:=\delta(\vare,d_0,d)$, and positive integer $C:=C(\vare,d_0,d)$, such that for any finite extension $K$ of $\bbq_p$ with degree $[K:\bbq_p]\le d$ the following holds: let $\ocal$ be the ring of integers of $K$, and $\pfr$ be a uniformizing element of $K$. Suppose $A\subseteq \ocal^{d_0}:=\ocal\times \cdots \times \ocal$ such that
\[
|\pi_{\pfr^N}(A)|\ge |\f|^{N\vare}.
\]
Then 
\[
\pi_{\pfr^{N_2}}(\pfr^{N_1}\bbz \xbf)\subseteq \pi_{\pfr^{N_2}}(\gen{A}{C}),
\]
for some $\xbf\in \ocal^{d_0}\setminus \pfr\ocal^{d_0}$, and integers $N_1$ and $N_2$ such that
\[
\lfloor N\delta\rfloor +N_1\le N_2\le NC. 
\]
\end{cor}
Another important corollary of Theorem~\ref{t:ThickSegment} is its {\rm global} version; this is a generalization of \cite[Corollary, Part I.1]{Bor} where the case of $k=\bbq$ (and $d_0=1$) is proved.
\begin{cor}\label{c:NumberField}
For any $0<\vare\ll 1$ and positive integers $d$ and $d_0$, there are $0<\delta:=\delta(\vare,d,d_0)$, and positive integer $C:=C(\vare,d,d_0)$, such that for any finite extension $k$ of $\bbq$ of degree at most $d$ the following holds:

Let $\ocal_k$ be the ring of integers of $k$, and $\pfr$ be a non-zero prime ideal of $\ocal_k$. Suppose $A\subseteq \ocal_k^{d_0}:=\ocal_k\times \cdots \times \ocal_k$ such that
\[
|\pi_{\pfr^N}(A)|\ge |\ocal_k/\pfr|^{N\vare}.
\]
Then 
\[
\pi_{\pfr^{N_2}}(\{i\xbf|\h i\in \bbz\cap \pfr^{N_1}\})\subseteq \pi_{\pfr^{N_2}}(\gen{A}{C}),
\]
for some $\xbf\in \ocal_k^{d_0}\setminus \pfr^d\ocal_k^{d_0}$, and integers $N_1$ and $N_2$ such that
\[
\lfloor N\delta\rfloor +N_1\le N_2\le NC. 
\]
\end{cor}
\begin{proof}
For any $\pfr$, let $k_{\pfr}$ be the completion of $k$ with respect to the $\pfr$-adic topology. Let $\ocal_{\pfr}$ be the ring of integers of $k_{\pfr}$, and $\wt{\pfr}$ be a uniformizing element of $\ocal_{\pfr}$. Then it is well-known that  $\pfr\ocal_{\pfr}=\langle \wt{\pfr} \rangle$ and the embedding of $\ocal_k$ into $\ocal_{\pfr}$ induces an isomorphism between $\pi_{\pfr^N}(\ocal_k)$ and $\pi_{\wt{\pfr}^N}(\ocal_{\pfr})$. Now we get the desired result by Corollary~\ref{c:ThickSegment}.
\end{proof}

It is not clear to the author whether the implied constants in Theorem~\ref{t:ThickSegment} should depend on the degree or not. In this note the degree is used in a crucial way to analyze subrings of $\ocal$. But the implied constants in many results proved here are independent of $K$ as it will be explained in Section~\ref{ss:outline}. Here is one such result.

\begin{thm}\label{thm-multiscaleBKT}
For any positive integer $t$, positive numbers $0<\vare_1\ll\vare_2\ll_t 1$, $0<\delta\ll_{\vare_1} 1$, any positive integer $C\gg_{\vare_1} 1$, and any finite extension $K$ of $\bbq_p$ with large, depending on $\vare_1$, residue field $\f$ the following holds: 	let $\ocal$ be the ring of integers of $K$, and $\pfr$ be a uniformizing element of $K$. Suppose $A\subseteq \pi_{\pfr^N}(\ocal)$ such that
\begin{enumerate}
\item $|\pi_{\pfr^i}(A)|\ge |\f|^{i\vare_1}$ for any $N\delta\le i\le N$.
\item $0,1\in A$ and there are $a_{1}, a_{2}\in A$ such that $a_{1}-a_{2}\in \pi_{\pfr^{N}}(\pfr\ocal\setminus \pfr^{2}\ocal)$.
\end{enumerate}
Then either 
\[
\pi_{\pfr^N}(\pfr^{\lceil\vare_2 N\rceil}\ocal)\subseteq \gen{A}{C},
\]
or 
\[
\pi_{\pfr^{\lfloor t\vare' N\rfloor}}(\gen{A}{C}\cap \pi_{\pfr^N}(\pfr^{\lceil\vare'N\rceil}\ocal))
\text{ is a ring, and } \gen{A}{C}\cap \pfr^{\lceil\vare'N\rceil}\ocal\setminus \pfr^{\lceil\vare'N\rceil+1}\ocal\neq \varnothing
\]
for some $\vare'$ in $[\vare_2^{m(\vare_1)},\vare_2]$.
\end{thm}
The subtlety in Theorem~\ref{t:ThickSegment} is that $A$ might be in a smaller field. Or even if the field generated by $A$ is the entire $K$, still in certain scales $A$ might be seen as a subring of a smaller field. Condition (b) in Theorem~\ref{thm-multiscaleBKT} guarantees that at least the ramification index of the field generated by $A$ is the same as $K$ and it can be detected even in the large scale $|\pfr|$.   
 
\begin{question}~\label{qu-field-conditions}
	Does Theorem~\ref{t:ThickSegment} hold with no restriction on $K$? If not, what is the least information needed on $K$? 
\end{question}
\subsection{A bounded generation conjecture for the quotients of ring of integers of a number field.}
As it was pointed out in Section~\ref{ss:introduction}, Theorem~\ref{t:ThickSegment} was used in \cite{SG1} to deduce a bounded generation statement for semisimple $p$-adic analytic groups, and then it was extended to perfect groups with abelian unipotent groups in \cite{SG2}; and ultimately these results were utilized to prove the $p$-adic case of super-approximation. If the following bounded generation for number fields holds, then one might be able to prove super-approximation for semisimple groups (or at least absolutely almost simple groups). 
\begin{conjecture}\label{conj-adele}
Suppose $0<\vare\ll 1$, $d$ is a positive integer, and $N_0\gg_{d,\vare} 1$ is a positive integer. Then there are $0<\delta:=\delta(\vare,d)$, and positive integer $C:=C(\vare,d)$ such that for any number field $k$ of degree at most $d$ the following holds: let $\ocal$ be the ring of integers of $k$. Suppose $\afr$ is an ideal of $\ocal$ such that $N(\afr):=|\ocal/\afr|\ge N_0$; and suppose $A\subseteq \ocal$ such that 
\[
|\pi_{\afr}(A)|\ge |\pi_{\afr}(\ocal)|^{\vare}.
\]
Then there are an ideal $\afr_2$ of $\ocal$ and $a_1\in\ocal$ such that 
\begin{align*}
\afr^C\subseteq \afr_2,& 
\h\h 
N(\afr_2:a_1)\ge N(\afr)^{\delta} 
\text{ where } (\afr_2:a_1)=\{x\in\ocal|\h xa_1\in \afr_2\}, \text{ and }
\\
\pi_{\afr_2}(\bbz a_1)& \subseteq \pi_{\afr_2}(\gen{A}{C}).  
\end{align*}	
\end{conjecture}  
    
\subsection{A detailed outline of proofs of Theorems~\ref{t:ThickSegment} and~\ref{thm-multiscaleBKT}; and some of the auxiliary results.}\label{ss:outline} 
In this section a detailed outline of the arguments is given. Here are some of my reasons to include this admittedly long overview (1) many of the statements are fairly technical; but lots of ideas in their proofs can be useful  for other problems as well. Having an overview which includes the main ideas of the auxiliary results can help readers to focus on the parts of the note that they find suitable for their purposes; (2) this can help to highlight the needed new ideas introduced here and put them in the perspective of some of the previous related works; (3) this should help the reader to get a more coherent picture of otherwise {\em locally disconnected} note.      

Before we get to the main goal of this section, we start with recalling the setting and some of the basic properties of the ring $\ocal$ of integers of a finite extension $K$ of $\bbq_p$.  

We let $\pfr$ be a uniformizing element of $\ocal$, and $\f$ be the residue field (that means $\f:=\ocal/\pfr\ocal$). For $x\in \ocal$, we let $v(x)$ be its $\pfr$-adic valuation; that means $v(x)$ is a non-negative integer such that $x\in \pfr^{v(x)}\ocal\setminus \pfr^{v(x)+1}\ocal$. For $x\in \ocal$, we let $|x|:=(1/|\f|)^{v(x)}$. It is well-known that $d(x,y):=|x-y|$ defines a metric on $\ocal$ and the ball of radius $|\pfr^n|$ centered at $0$ is $\pfr^n\ocal$. On the algebraic side, $\{\pfr^i\ocal\}_{i=0}^{\infty}$ is a filtration of $\ocal$; that means it is a family of ideals of $\ocal$ and $\pfr^i\ocal\cdot\pfr^j\ocal\subseteq \pfr^{i+i}\ocal$. A common technique to study an algebra with a filtration is making use of the associated graded algebra. In our setting this means, we define $\gr_{i,\pfr}(\ocal):=\pfr^i\ocal/\pfr^{i+1}\ocal$ for any non-negative integer $i$, and let $\gr_{\pfr}(\ocal):=\bigoplus_{i=0}^{\infty}\gr_{i,\pfr}(\ocal)$. As $\gr_{i,\pfr}(\ocal)$ are abelian groups so is $\gr_{\pfr}(\ocal)$. It is well-known that $\gr_{\pfr}(\ocal)$ is a graded algebra with respect to the following multiplication:  $(x_i+\pfr^{i+1}\ocal)(x_j+\pfr^{j+1}\ocal):=x_ix_j+\pfr^{i+j}\ocal$, for any $x_i+\pfr^{i+1}\ocal\in \gr_{i,\pfr}(\ocal)$ and $x_j+\pfr^{j+1}\ocal\in \gr_{j,\pfr}(\ocal)$. In fact, it is well-known that $\gr_{\pfr}(\ocal)$ is isomorphic to the ring of polynomials $\f[t]$ with coefficients in the residue field $\f$. 
Based on this fact and the completeness of $\ocal$, we get the following description of its elements: suppose $\Omega$ is a subset of $\ocal$ such that the quotient map $\pi_{\pfr}:\ocal\rightarrow \f$ induces a bijection from $\Omega$ to the set $\f^{\times}$ of non-zero elements of $\f$. Then for any $X\in \ocal$ there are unique $X_i\in \Omega\cup\{0\}$ such that 
\be\label{eq-p-base}
X=X_0+\pfr X_1+ \pfr^2 X_2 + \cdots . 
\ee
We call $X_i$ {\em the $i$-th $\pfr$-adic digit with respect to $\Omega$}, and sometimes denote it by $D_{i,\Omega}(X)$ or simply $D_i(X)$ (these digits {\em depend} on the choice of $\Omega$, but $\Omega$ will be fixed at the beginning of any proof). For instance when $\ocal=\bbz_p$ and $\Omega:=\{1,\ldots,p-1\}$, Equation \eqref{eq-p-base} gives us the usual $p$-base description of the $p$-adic integers. (At some point we will be working with a different set of digits, but for now a reader can think about $\ocal=\bbz_p$ and $\Omega=\{1,\dots,p-1\}$ in order to get a more concrete understanding of the setting). Let us notice that for any non-negative integer $n$ and $X,Y\in \ocal$ we have that $\pi_{\pfr^n}(X)=\pi_{\pfr^n}(Y)$ if and only if $D_i(X)=D_i(Y)$ for any $0\le i\le n-1$. So for $0\le i\le n-1$, we can and will talk about {\em the $i$-th $\pfr$-adic digit $D_i(\pi_{p^n}(X))$ with respect to $\Omega$} of an element $\pi_{\pfr^n}(X)$ of $\ocal/\pfr^{n}\ocal$. Hence for any $x\in \ocal/\pfr^n\ocal$ we have
\[
x=D_0(x)+D_1(x)\pfr+\cdots+D_{n-1}(x)\pfr^{n-1}+\pfr^{n}\ocal.
\]

{\bf Step 1.} {\em Describing subrings of the ring of integers $\ocal$.} 

Starting with a subset $A$ of $\ocal$, Theorems \ref{t:ThickSegment} and~\ref{thm-multiscaleBKT} are claiming certain bounded generation phenomena within the ring $R$ generated by $A$ in certain scales. So it is only reasonable to start with a description of subrings of $\ocal$ in a given scale. 

\begin{thm}\label{t:SubringOfO}
	Let $\ocal$ be the ring of integers of a finite extension $K$ of $\bbq_p$. Let $\pfr$ be a uniformizing element of $\ocal$, and $\f:=\ocal/\pfr\ocal$ be the residue field of $K$. Suppose $R$ is a closed subring of $\ocal$ which contains $1$. Let $C$ be an integer which is at least $3$. Suppose $F$ is an integer and $F\gg_C [K:\bbq_p]$. Then there are integers $a$ and $b$, and a subfield $K_0$ of $K$ such that 
	\be\label{eq-thickness-ring-of-integers-case}
	b-a\gg_{C,[k:\bbq_p]} F, \h\text{ and }\h F\ge b\ge Ca,
	\ee 
	\be\label{eq-subring-structure-ring-of-integers-case}
	\pi_{\pfr^b}(\ocal_0\cap \pfr^a\ocal)=\pi_{\pfr^b}(R\cap \pfr^a\ocal), 
	\text{ and } 
	\pi_{\pfr^b}(R)\subseteq \pi_{\pfr^b}(\ocal_0)
	\ee	
	where $\ocal_0$ is the ring of integers of $K_0$ and $\pi_{y}:\ocal\rightarrow \ocal/y\ocal$ is the natural quotient map for any $y\in \ocal\setminus \{0\}$.
\end{thm}
Theorem~\ref{t:SubringOfO} essentially says that, if we can only compute the first $F$-digits of the elements of the subring $R$, then we can find a {\em large} segment (proportional with $F$) of digits where $R$ is the same as ring of integers of a closed subfield. 

To prove Theorem~\ref{t:SubringOfO}, first we use the above mentioned philosophy, and prove the graded version (see Proposition~\ref{p:SubringOfPolynomialRing}). Proof of Proposition~\ref{p:SubringOfPolynomialRing} is combinatorial in nature. Along the way a result for numerical semigroups is proved that might be of independent interest (see Proposition~\ref{p:NumercialSemigroups}). 

In general going to a graded structure we might lose a lot of information about the original ring. For instance, starting with a wildly ramified Galois extension $K/\bbq_p$ there are non-trivial elements $\sigma\in {\rm Gal}(K/\bbq_p)$ such that $\sigma(a)\equiv a \pmod{\pfr}$ for any $a\in \ocal$. This means $\sigma$ induces the trivial automorphism of ${\rm gr}_{\pfr}(\ocal)$. So we can get subrings of $\ocal$ that give us the same graded subrings of ${\rm gr}_{\pfr}(\ocal)$. 

Proof of Theorem~\ref{t:SubringOfO} is a bit delicate which relies on rather well-known techniques from algebraic number theory; for instance a generalization of Hensel's lemma, Krasner's lemma, and basic facts about local fields. Along the way we get that if a closed subring $R$ of $\ocal$ have the same graded ring as ring of integers $\ocal_0$ of a closed subfield, then $R$ is the ring of integers of a closed subfield (see Proposition~\ref{prop:detecting-ring-of-integers-via-grading-structure} and Step 1 of proof of Proposition~\ref{prop-bounded-generation-for-equal-grades}).

A reader who is interested in the new techniques related to sum-product results can skip the proof of these statements. These results are used only towards the end of the note in the proof of Theorem~\ref{t:ThickSegment}. But readers should familiarize themselves with the notation and basic properties introduced in Lemma~\ref{lem-attached-graded-rings} and Corollary~\ref{cor-grade-shift}.     

{\bf Step 2.} {\em Using conditional entropy to get a Scalar-Sum expansion.}
Starting with two subsets $A$ and $B$ of $\pi_{\pfr^{N}}(\ocal)$, we would like to get a lower bound on $|A+B|$. As it was explained earlier, one often proves {\em a weighted version}: let $\pcal_A$ and $\pcal_B$ be the probability counting measures on $A$ and $B$, respectively. Then $A+B$ is the support of the additive convolution $\pcal_A\ast \pcal_B$ of $\pcal_A$ and $\pcal_B$. So if we show this new measure is {\em more distributed} than the initial measures, we should get a desired expansion on their supports. In this note, following~\cite{LV}, we use entropy to quantify how well  a measure is distributed. Let $X$ and $Y$ be random variables with respect to the distribution laws $\pcal_A$ and $\pcal_B$, respectively. Then it is well-known that 
\[
\log |A+B|\ge H(X+Y)
\]
 where $H(\bullet)$ is the (Shanon) entropy of the given random-variable (see~Definition \ref{d:Entropy} and Lemma \ref{l:PropertiesEntropy} for the definition and some of the basic properties of entropy). As it was explained above, a random variable $Z$ with values in $\pi_{\pfr^N}(\ocal)$ can be given in terms of its $\pfr$-adic digits with respect to $\Omega$. So we get random variables $D_i(Z)$ with values in $\Omega\cup\{0\}$ for any $0\le i \le N-1$; and we have 
 \be\label{e:GoingToDigits}
 H(Z)=H(D_0(Z),\ldots,D_{N-1}(Z)).
 \ee
By \eqref{e:GoingToDigits} and a basic property of conditional entropy (see Lemma~\ref{l:PropertiesEntropy}), we get
\be\label{e:GoingToConditionalMeasures}
H(Z)=H(D_0(Z))+H(D_1(Z)|D_0(Z))+\cdots+H(D_{N-1}(Z)|D_0(Z),\ldots,D_{N-2}(Z)).
\ee
Now we observe that the  {\em carry over} method for addition works in $\ocal$ as well; this means 	for $X,Y\in \ocal$, to determine the $m$-th $\pfr$-adic digit $D_m(X+Y)$ of $X+Y$, we should add the $m$-th digits $D_m(X)$ and $D_m(Y)$ of $X$ and $Y$, and add the {\em carry over} $f_{\Omega}(D_0(X),\ldots,D_{m-1}(X),D_0(Y),\ldots,D_{m-1}(Y))$ from the addition of the first $m-1$ digits. Moreover, since $\pi_{\pfr}$ induces a bijection between $\Omega\cup\{0\}$ and $\f$, to find $D_m(X+Y)$ it is necessary and sufficient to find 
\be\label{eq-m-th-digit}
\pi_{\pfr}(D_m(X))+\pi_{\pfr}(D_m(Y))+\pi_{\pfr}(f_{\Omega}(D_0(X),\ldots,D_{m-1}(X),D_0(Y),\ldots,D_{m-1}(Y))).
\ee 
In particular, the first $m$-th $\pfr$-adic digits with respect to $\Omega$ of $X+Y$ are uniquely determined by the first $m$ $\pfr$-adic digits with respect to $\Omega$ of $X$ and $Y$. Therefore for any $0\le m\le N-1$ we have
\be\label{eq-GoingToResidueField-Entropy}
H(D_m(X+Y)|D_0(X+Y),\ldots,D_{m-1}(X+Y)) \hspace{8cm}
\ee
\begin{align*}
\ge & H(D_m(X+Y)|D_0(X),\ldots,D_{m-1}(X),D_0(Y),\ldots,D_{m-1}(Y))
\\
=& H(\pi_{\pfr}(D_m(X))+\pi_{\pfr}(D_m(Y))+\pi_{\pfr}(f_{\Omega}(D_0(X),\ldots,D_{m-1}(X),D_0(Y),\ldots,D_{m-1}(Y)))
\\
&
\hspace{7.5cm}|D_0(X),\ldots,D_{m-1}(X),D_0(Y),\ldots,D_{m-1}(Y))
\\
=
&
H(\pi_{\pfr}(D_m(X))+\pi_{\pfr}(D_m(Y))|D_0(X),\ldots,D_{m-1}(X),D_0(Y),\ldots,D_{m-1}(Y)).
\end{align*}
Based on \eqref{e:GoingToConditionalMeasures} and \eqref{eq-GoingToResidueField-Entropy}, in order to get a lower bound on $H(X+Y)$, one needs to get a lower bound on $H(\overline{X}+\overline{Y})$ where $\overline{X}$ and $\overline{Y}$ are two random-variables with values in the residue field $\f$. This is the line of thought in \cite{LV} where they deal with the case of $\bbz/2^N\bbz$; and so $\f=\bbz/2\bbz$. In that case, any distribution on $\f=\bbz/2\bbz$ can be characterized by one value, say the probability of hitting 1. Based on this and using calculus of single variable functions a desired lower bound for $H(\overline{X}+\overline{Y})$ is attained in \cite{LV}. 

When the order of the residue field can be arbitrarily large, our method should have some implications for finite fields as well. In~\cite[Lemma 2.1]{BKT} in order to prove a sum-product result for finite fields, first a scalar-sum expansion is proved; to be precise it is showed that in average the size of $|\overline{A}+\overline{\alpha}\overline{B}|$ is at least $\min\{|\overline{A}||\overline{B}|/2,|\f|/10\}$, where $\overline{\alpha}$ is a random-variable with respect to the counting probability measure on $\f^{\times}$. So it seems the following question to be the right property to seek.
\begin{question}\label{qu-entropy-lower-bound}
	Let $\f$ be a finite field. Suppose $\overline{X}$, $\overline{Y}$,  $\overline{\alpha}$, and $\overline{Z}_{\f}$ are random variables with values in $\f$; $\overline{\alpha}$ is distributed with respect to the probability counting measure on the set $\f^{\times}$ of non-zero elements of $\f$, and $\overline{Z}_{\f}$ is distributed with respect to the counting probability measure on $\f$. Is there a (fixed universal) positive number $c$ such that
	\[
	H(\overline{X}+\overline{\alpha}\overline{Y}|\overline{\alpha})\ge \min\{H(\overline{X})+H(\overline{Y}),H(\overline{Z}_{\f})\}-c?
	\]
\end{question}
It is worth pointing out that we know $H(\overline{Z}_{\f})$ is $\log |\f|$ and $\overline{Z}_{\f}$ has the maximum entropy among all the random variables with values in $\f$; in particular $H(\overline{X}+\overline{\alpha}\overline{Y}|\overline{\alpha})\le H(\overline{Z}_{\f})$ (in the setting of Question~\ref{qu-entropy-lower-bound}). For now, we do not know the answer to Question~\ref{qu-entropy-lower-bound} for arbitrary random variables $\overline{X}$ and $\overline{Y}$; but Lemma~\ref{l:AverageL2norm} implies an affirmative answer to this question when $\overline{X}$ and $\overline{Y}$ are distributed according to the probability counting measures $\pcal_{\overline{A}}$ and $\pcal_{\overline{B}}$, respectively. More precisely, Proposition~\ref{p:cond-entropy-linear-comb} states 
\be\label{eq-restricted-entropy-bound-residue-field}
H(\overline{X}+\overline{\alpha}\overline{Y}|\overline{\alpha})\ge -\log\left(\frac{1}{|\overline{A}||\overline{B}|}+\frac{1}{|\f|}\right)\ge \min\{\log |\overline{A}|+\log |\overline{B}|,\log|\f|\}-\log 2
\ee
where $\overline{A}$ and $\overline{B}$ are subsets of $\f$, and $\overline{X}$ and $\overline{Y}$ are distributed according to the probability counting measures $\pcal_{\overline A}$ and $\pcal_{\overline B}$, respectively.  

In order to be able to use \eqref{eq-restricted-entropy-bound-residue-field} in the $\pfr$-adic setting via \eqref{e:GoingToConditionalMeasures} and \eqref{eq-GoingToResidueField-Entropy}, we need to start with {\em regular} subsets $A$ and $B$ of $\pi_{\pfr^{N}}(\ocal)$ (see Definition~\ref{d:RegularSubset}). Basically a subset $A$ of $\pi_{\pfr^N}(\ocal)$ is an $(m_0,\ldots,m_{N-1})$-regular subset, if for a random-variable $X$ according to the probability counting measure on $A$ and a given first $i$ $\pfr$-adic digits $w_0,\ldots,w_{i-1}$ of an element of $A$, the conditional distribution
\[
P(D_i(X)|D_0(X)=w_0,\ldots,D_{i-1}(X)=w_{i-1})
\]
is the probability counting measure on a set of cardinality $m_i$. So for an $(m_0,\ldots,m_{N-1})$-regular subset $A$, an $(l_0,\ldots,l_{N-1})$-regular subset $B$, a random-variable $X$ according to the probability counting measure $\pcal_A$, and a random-variable $Y$ according to the probability counting measure $\pcal_B$, we have
\begin{align*}
H(X+\alpha Y|\alpha)\ge& \sum_{i=0}^{N-1} H(\pi_{\pfr}(D_i(X))+\pi_{\pfr}(\alpha)\pi_{\pfr}(D_i(Y))|\alpha, D_j(X), D_j(Y) \text{ for } 0\le j\le i-1)
\\
\ge & \sum_{i=0}^{N-1} -\log\left(\frac{1}{m_il_i}+\frac{1}{|\f|}\right),
\end{align*}
where $\alpha$ is a random-variable according to the probability counting measure on $\pi_{\pfr^N}(\Omega)$. And this implies 
\be\label{eq-scalar-sum-expansion}
\max_{w\in \Omega} |A+\pi_{\pfr^{N}}(w)B|\ge \prod_{i=0}^{N-1}\left(\frac{1}{m_il_i}+\frac{1}{|\f|}\right)^{-1},
\ee
which is our desired Scalar-Sum expansion (see Proposition~\ref{p:ScalarSumRegular}). Roughly this inequality says that, if we do not get a meaningful Scalar-Sum expansion, the reason is that at any level $i$ either both $\log m_i$ and $\log l_i$ are close to $\log |\f|$ or both $\log m_i$ and $\log l_i$ are close to $0$.  

{\bf Step 3.} {\em Following Lindenstrauss-Varj\'{u}'s treatment to get a Scalar-Sum-Product expansion for a regular set.} In this step, we prove that for some $a\in A-A$ and $w\in \Omega$ the set $A+\pi_{\pfr^N}(w)aA$ is significantly larger than $A$ where $A$ is  a {\em regular} subset of $\pi_{\pfr^N}(\ocal)$ with three other conditions. The key observation behind this step is the fact that, if $A$ is an $(m_0,\ldots,m_{N-1})$-regular subset of $\pi_{\pfr^{N}}(\ocal)$, then, for any $x\in \ocal$, $\pi_{\pfr^{N}}(x)A$ is a $(1,\ldots,1,m_0,\ldots,m_{N-1-v(x)})$-regular subset where $v(x)$ is the $\pfr$-adic valuation of $x$. This implies that, for any $i\in B:=\{j\in [0,N-1]|\h m_j>1\}$, there is $a\in A-A$ such that $aA$ is a $(1,\ldots,1,m_0,\ldots,m_{N-1-i})$-regular subset. Let $T:=\{j\in [0,N-1]| \log m_j/\log|\f| \ge 1/2\}$. Now applying the Scalar-Sum expansion proved in the first step (see the inequality in \eqref{eq-scalar-sum-expansion}) for the regular sets $A$ and $aA$ for any $a\in A-A$, we get that either  $|A+\pi_{\pfr^N}(w)aA|$ is significantly larger than $|A|$ for some $w\in \Omega$ and $a\in A-A$, or $T$ is almost invariant under the shifts by elements of $B$. Then assuming that $A$ has at least box dimension $\vare$ for any scale smaller than $(1/|\f|)^{O(\vare^4 N)}$ and $0,1\in B$, we deduce that a shift of $B$ has Schnirlmann density (see Definition~\ref{d:Density}) at least $\vare$; and then by a theorem of Mann (see Theorem~\ref{t:Mann}) we reach to a contradiction.   

	Let us emphasis that the key point of the argument is where we say {\em the set $T$ of indexes where $m_i$ is at least $\sqrt{|\f|}$ is almost invariant under shifts by integers $j$ such that $m_j>1$} (see Lemma~\ref{l:BT}); and we deduce this claim using the inequality in \eqref{eq-scalar-sum-expansion} for the regular sets $A$ and $aA$ for suitable $a\in A-A$. 
	
	Another remark is that the crucial condition $m_1>1$ is why we get a result with no dependence on $K$.
		
{\bf Step 4.} {\em Proving a Scalar-Sum-Product expansion: removing the regularity assumption.} Finally at this step we get a satisfactory expansion result:
\begin{thm}[Scalar-Sum-Product expansion]\label{t:ScalarSumExpansion}
	For any $\vare>0$, $0<\delta\ll \vare^5$, and any finite extension $K$ of $\bbq_p$ with large, depending on $\vare$, residue field $\f$ the following holds:
	
	Let $\ocal$ be the ring of integers of $K$, and $\pfr$ be a uniformizing element of $K$. Let $\Omega\subseteq \ocal$, and suppose $\pi_{\pfr}$ induces a bijection between $\Omega\subseteq \ocal$ and $\f^{\times}$. Suppose $A\subseteq \pi_{\pfr^N}(\ocal)$ such that
	\begin{enumerate}
		\item $|A|\le |\f|^{N(1-\vare)}$,
		\item $|\pi_{\pfr^i}(A)|\ge |\f|^{i\vare}$ for any $N\delta\le i\le N$.
		\item there are $a_{01}, a_{02}, a_{11}, a_{12}\in A$ such that $a_{01}-a_{02}\in \pi_{\pfr^{N}}(\ocal\setminus \pfr\ocal), a_{11}-a_{12}\in \pi_{\pfr^{N}}(\pfr\ocal\setminus \pfr^{2}\ocal)$.
	\end{enumerate}
	Then
	\[
	\max_{\omega\in \Omega}|\gen{A}{6} +\pi_{\pfr^N}(\omega)\gen{A}{6}|\ge |A| |\f|^{N\delta}.
	\]
\end{thm}  
Condition (a) allows us to have enough space to expand. Condition (b) says that the box dimension of the lift $\pi_{\pfr^N}^{-1}(A)\subseteq \ocal$ of $A$ at any scale smaller than $(1/|\f|)^{N\delta}$ is at least $\vare$. This is a rather (needed) technical assumption which will be eventually removed; but removing this condition results to having a weaker conclusion, which has a meaning only in $\ocal$ and not in $\pi_{\pfr^N}(\ocal)$. Condition (c) tells us something about the valuation of elements of $A-A$: it is equivalent to say that there are $a_0,a_1\in A-A$ such that $v(a_0)=0$ and $v(a_1)=1$. As before this crucial condition helps us get a result that works with no dependency on the field $K$. 

The key idea is a {\em regularization process} that has been used in most of the previous works related to either a sum-product or a product result in a {\em multi-scaled space} (for instance see \cite{BG, Bou-RingConjecture, Bor} or \cite[Section 2.2]{SG1}). In this process, we construct a rooted regular tree with $N$ levels; the vertices in the $i$-th row are elements of 
	$\pi_{\pfr^i}(\ocal)$, and the parent of $\pi_{\pfr^i}(x)$ is $\pi_{\pfr^{i-1}}(x)$ for $1\le i\le N$. We view $A$ as a subset of the vertices at the $N$-th level, and consider the rooted sub-tree induced by $A$. Through this process, each time we choose a subset $A_i$ of $A$ such that first $|A_i|\ge |A|/(\log |\f|)^i$ and second the last $i$-th levels of the rooted sub-tree induced by $A_i$ are {\em regular}; that means there are equal number of paths from any vertex at the $(N-i)$-th level to the $N$-th level. After obtaining this regular {\em large} subset and changing it a little bit, we apply the Scalar-Sum-Product expansion for regular sets and deduce the desired result. 

{\bf Step 5.} {\em Proving a bounded generation result.} At this step we prove:
\begin{thm}\label{t:BoundedGenerationScalarSum}
	For any $0<\vare_1\ll\vare_2\ll 1$, a positive integer $m$, $0<\delta\ll_{m,\vare_1} 1$, positive integers $1 \ll_{m,\vare_1} C$ (number of needed sum-product) and $1 \ll_{\vare_1} k$ (number of needed scalars) and any finite extension $K$ of $\bbq_p$ with large, depending on $\vare_1$, residue field $\f$ the following holds:
	
	Let $\ocal$ be the ring of integers of $K$, and $\pfr$ be a uniformizing element of $K$. Let $\Omega\subseteq \ocal$, and suppose $\pi_{\pfr}$ induces a bijection between $\Omega\subseteq \ocal$ and $\f^{\times}$. Suppose $A\subseteq \pi_{\pfr^N}(\ocal)$ such that
	\begin{enumerate}
		\item $|\pi_{\pfr^i}(A)|\ge |\f|^{i\vare_1}$ for any $N\delta\le i\le N$.
		\item there are $a_{01}, a_{02}, a_{11}, a_{12}\in A$ such that $a_{i1}-a_{i2}\in \pi_{\pfr^{N}}(\pfr^i\ocal\setminus \pfr^{i+1}\ocal)$.
	\end{enumerate}
	Then
	\be\label{e:BGScalarSumProduct}
	\pi_{\pfr^N}(\pfr^{\lceil\vare_2^m N\rceil}\ocal)\subseteq \gen{A}{C}+\pi_{\pfr^N}(\omega_1)\gen{A}{C}+\cdots+\pi_{\pfr^N}(\omega_k)\gen{A}{C},
	\ee
	for some $\omega_i\in \prod_k(\Omega\cup\{1\})$.
\end{thm}
The conditions (a) and (b) are similar to the technical conditions (b) and (c) of Theorem~\ref{t:ScalarSumExpansion}. As here we are seeking a {\em bounded generation} result and not an {\em expansion} result, no upper bound on $|A|$ is needed (see condition (a) of Theorem~\ref{t:ScalarSumExpansion}). 

A quick explanation of the parameters involved in Theorem~\ref{t:BoundedGenerationScalarSum}: $\vare_1$ gives us a lower bound for the box dimension of $\pi_{\pfr^N}^{-1}(A)$ at scales smaller than $(1/|\f|)^{N\delta}$; so a smaller $\delta$ imposes more conditions on $A$; the {\em thickness} of the $\ocal$-segment generated in $C$ steps sum-product and $k$ steps scalar-sum is roughly $(1/|\f|)^{\vare_2^mN}$; so for a smaller $\vare_2$ and a larger $m$ we get a thicker $\ocal$-segment. An important point to raise is that the number $k$ of the needed scalar-sum steps is {\em independent} of $m$. This is crucial when we want to get a bounded generation result using only sum and product. 

To prove Theorem~\ref{t:BoundedGenerationScalarSum}, first we show the case of $m=1$ (see Proposition~\ref{p:BoundedGenerationScalarSum}). To show this case, we use the Scalar-Sum-Product expansion result, Theorem~\ref{t:ScalarSumExpansion}, repeatedly to get a subset of $\pi_{\pfr}(\ocal)$ with arbitrarily large {\em box dimension}; that means to get a subset $A'$ such that $|A'|\ge |\pi_{\pfr}(\ocal)|^{1-\delta}$ for a small fixed positive number $\delta$. Then we get the desired bounded generation result using Fourier analysis. This is a common feature of most of the proofs on this type of results; for instance this part of a bounded generation result in groups is usually done by proving a kind of {\em mixing property} (see Sarnak-Xue~\cite{SaXu} and Gower's notion of quasi-randomness~\cite{Gow-quasirandom}). 

In the second step, we appeal to the associated graded algebra ${\rm gr}_{\pfr}(\ocal):=\bigoplus_{i=0}^{\infty} \pfr^i\ocal/\pfr^{i+1}\ocal$ (as we have pointed out earlier this algebra is isomorphic to $\f[t]$) in order to gain more information on $\gen{A}{C}$ for some integer $C:=C(\vare)$. To be more precise, to any subset $\wt{A}$ of $\ocal$ and any non-negative integer $i$, we associate {\em the $i$-th grade} ${\rm gr}_{i,\pfr}(\wt{A};\ocal):=((\wt{A}\cap \pfr^i\ocal)+\pfr^{i+1}\ocal)/\pfr^{i+1}\ocal$ of $\wt{A}$ (and of course this can be done for any subset $A$ of $\pi_{\pfr^N}(\ocal)$ as well). And for $A\subseteq \pi_{\pfr^N}(\ocal)$ we let $J(A):=\{i\in [0,N-1]|\h {\rm gr}_i(A)\neq 0\}$.
	Notice that 
	\[
	J(A)=\{v(x)|\h \pi_{\pfr^N}(x)\in A\}\cap [0,N-1]
	\]
	 where $v$ 
	the is $\pfr$-adic valuation, and condition (b) of Theorem~\ref{t:BoundedGenerationScalarSum} is equivalent to saying $0,1$ are in $J(A-A)$. Let us also observe that the graded algebra structure of ${\rm gr}_{\pfr}(\ocal)$ implies $J(A_1A_2)\supseteq J(A_1)+J(A_2)$ for any two subsets $A_1$ and $A_2$ of $\ocal$. Having these in mind, using Lemma~\ref{l:LargeInterval} (which is based on the Mann theorem on sets with positive Schnirlmann density) we deduce that $J(\gen{A}{3\lceil 1/\vare_1 \rceil}) \supseteq (\lceil 1/\vare_1 \rceil\delta N, N)\cap \bbz$ if $A$ satisfies properties (a) and (b) of Theorem~\ref{t:BoundedGenerationScalarSum}.  

Finally to prove Theorem~\ref{t:BoundedGenerationScalarSum}, we use the $m=1$ case for the parameters 
$\vare_1^{\rm new}:=\vare_1$, 
$\vare_2^{\rm new}:=\vare_2$, 
$N^{\rm new}:=\lfloor \vare_2^{m-1}N\rfloor$, and 
$A^{\rm new}:= \pi_{\pfr^{N^{\rm new}}}(A)$, in order to get a scalar-sum-product set $\wh{A}$ such that $\pi_{\pfr^{N^{\rm new}}}(\wh{A})$ contains 
$\pi_{\pfr^{\lfloor \vare_2^{m-1}N\rfloor}}(\pfr^{\lceil \vare_2^{m}N\rceil}\ocal)$. Now using the sum-product set 
$\gen{A}{3\lceil 1/\vare_1 \rceil}$ we shift $\wh{A}$ at most $1/\vare_2$ many times and add them in order to fill out the entire $\ocal$-segment 
$\pi_{\pfr^N}(\pfr^{\lceil\vare_2^m\rceil}\ocal)$ {\em without introducing a new scaling parameter}. This is crucial as later we need to get rid of the used scalars; and each time we reduce the number of used scalars, it comes with a {\em cost} on the thickness of $\ocal$-segment. So we need to start with a {\em thick} enough $\ocal$-segment at the beginning of the process. 

{\bf Step 6.} {\em A multi-scaled version of the Bourgain-Katz-Tao argument and proof of Theorem~\ref{thm-multiscaleBKT}}  (see~\cite[Proofs of Lemma 4.2 and Theorem 4.3]{BKT}). Roughly the following steps were employed in \cite{BKT} to prove a bounded generation result in a finite field $\f$ (see \cite{EdMi-ErdosVolkmann} where a similar approach is used to prove Erd\"{o}s-Volkmann's ring conjecture): 
	\begin{enumerate}
	\item (Scalar-Sum bounded generation) There is a linear function $l:\f^k\rightarrow \f, l(x_1,\ldots,x_k):=\sum_{i=1}^k \alpha_i x_i$ such that $l(A^k)=\f$ (here the number $k$ of the needed scalars depends on $\log |A|/\log |\f|$);
	\item (Reducing the number of involved scalars) We have either {\bf (injectivity)} $l$ is injective on $A\times \cdots \times A$ or {\bf (reduction)} there is a linear function $l':\f^{k-1}\rightarrow \f$ with one less variable such that $l'(\gen{A}{2}^{k-1})=\f$;
	\item (Analyzing the injectivity case) If a linear function $l:\f^k\rightarrow \f$ is injective on $A^k$ and $l(A^k)=\f$, then $A$ is a subfield of $\f$.
	\end{enumerate}
(Here $A$ is a subset of $\f$ with cardinality at least $|\f|^{\vare}$ and $0,1\in A$. See Lemma~\ref{l:BoundedGenerationFiniteField} for the details and the precise statement.) For the purposes of this note, we need a {\em multi-scaled} version of these steps. That means we will be needing the injectivity of a linear map on a {\em large neighborhood} of a given set in order to be able to deduce existence of some algebraic structure on it. In the algebraic language the difficulty arises as we have lots of nilpotent elements in $\pi_{\pfr^N}(\ocal)$, but in a field $\f$ any non-zero element is invertible. Here we briefly explain how we overcome this difficulty. 

{\em Choice of the set of digits $\Omega\cup\{0\}$.} So far we have been working with an arbitrary set $\Omega$ of representatives in $\ocal$ of the non-zero elements of the residue field; and the needed scalars for the scalar-sum-product bounded generation were picked from $\prod_C \Omega\cup\{1\}$. At this step in order to have a slightly neater version of the process, we assume that $\Omega$ is a subgroup of the group of units of $\ocal$; using the Hensel lemma we know that there is a subgroup $\Omega$ of the group of units of $\ocal$ such that $\pi_{\pfr}:\Omega\rightarrow \f^{\times}$ is a group isomorphism. So we can and will assume $\prod_C \Omega\cup\{1\}=\Omega$ for any positive integer $C$, and more importantly the inverse of an element of $\Omega$ is again in $\Omega$. 

{\em The main dichotomy.} To explain this part, we introduce the symbol ${\rm BG}(A;\vare, k,C)$; for a subset $A$ of $\pi_{\pfr^N}(\ocal)$, a positive number $\vare$, and positive integers $k, C$, we say ${\rm BG}(A;\vare, k,C)$ holds if  
\be\label{eq-bounded-generation}
\pi_{\pfr^N}(\pfr^{\lceil\vare N\rceil}\ocal)\subseteq \gen{A}{C}+\pi_{\pfr^N}(\alpha_1)\gen{A}{C}+\cdots+\pi_{\pfr^N}(\alpha_k)\gen{A}{C}
\ee
for some $\alpha_1,\ldots,\alpha_k\in \Omega$. So roughly ${\rm BG}(A;\vare, k,C)$ is a compact way of saying that an $\ocal$-segment, whose box dimension at scale $|\pfr^N|$ is at least $1-\vare$, can be generated by $A$ in $C$ steps sum-product, and $k+1$ steps scalar-sum with scalars in $\Omega$. Let us observe that for $0<\vare_1\ll \vare_2\ll 1$ and a positive integer $m$ (under certain conditions on the set $A$), by Theorem~\ref{t:BoundedGenerationScalarSum}, 
${\rm BG}(A;\vare_2^m,k(\vare_1),C(\vare_1,m))$ holds. This will be serving us as the {\em initial seed} of a process similar to the explained Bourgain-Katz-Tao argument.
 
Assuming ${\rm BG}(A;\vare,k,C)$ holds, for any $\delta_0>0$, we prove (see Claim 1 in the proof of Lemma~\ref{lem-multiscale-BKT}) that either
\begin{enumerate}
\item {\bf ($\delta_0$-injectivity)} for any $\xbf,\xbf'\in \gen{A}{2C}^{k+1}$,  $l(\xbf)=l(\xbf_0)$ implies $\xbf-\xbf'\in \pfr_{\pfr^N}(\pfr^{\lfloor \delta_0 N\rfloor})$, where $l(x_0,\ldots,x_k)=x_0+\pi_{\pfr^N}(\alpha_1) x_1+\cdots+\pi_{\pfr^N}(\alpha_k) x_k$ and $\alpha_i$'s satisfy (\ref{eq-bounded-generation}), or
\item {\bf (Reduction)} 	${\rm BG}(A;\vare+\delta_0,k-1,8C)$ holds.
\end{enumerate}
	
{\em Analyzing the case where the reduction fails.} Suppose $\delta_0>\vare$, ${\rm BG}(A;\vare,k,C)$ holds for the sequence $\alpha_1,\ldots,\alpha_k$ of scales in $\Omega$, and we have the $\delta_0$-injectivity for $\gen{A}{2C}$ and $l(x_0,\ldots,x_k):=x_0+\pi_{\pfr^N}(\alpha_1)x_1+\cdots+\pi_{\pfr^N}(\alpha_k)x_k$; then we prove (see Lemma~\ref{lem-multiscale-BKT}) that
\be\label{eq-closed-under-addition}
\pi_{\pfr^{\lfloor \delta_0N\rfloor}}\left(\gen{A}{C}\cap \pi_{\pfr^N}(\pfr^{\lceil\vare N\rceil}\ocal) \right)
\ee
is a subring of $\pi_{\pfr^{\lfloor \delta_0N\rfloor}}(\ocal)$.

{\em Gaining an algebraic structure without using scalars.} We use the conclusion of Theorem~\ref{t:BoundedGenerationScalarSum} as the {\em initial seed} for using the main dichotomy. That means ${\rm BG}(A;\vare_2^m,k(\vare_1),C(\vare_1,m))$ holds for $0<\vare_1\ll \vare_2\ll 1$ and a positive integer $m$ (under certain assumptions on $A$).  Then using the main dichotomy we reduce the number of needed scalars. But each time the main dichotomy is used, assuming the reduction occurs, the number of needed scalars is reduced in the cost of getting a {\em smaller} $\ocal$-segment. So the next time we should pick a larger scale $\delta_0$ for using the main dichotomy. This shows how important it is to know that the number $k$ of needed scalars only depends on $\vare_1$ and it is {\em independent} of $m$.  Now by choosing $m$ large enough depending only on $\vare_1$ and choosing the scalars carefully, we get (see the proof of Lemma~\ref{lem-multiscale-BKT}) that 
\[
\pi_{\pfr^{\lfloor t\vare' N\rfloor}}\left(\gen{A}{C'}\cap \pi_{\pfr^N}(\pfr^{\lceil\vare' N\rceil}\ocal) \right)
\]
is a ring for some $\vare_2^m\le \vare'\le \vare_2$ and $C':=C'(\vare_1)$ (where $t$ is a given fixed integer).

{\em Controlling the gap of indexes of non-zero grades, and finishing the proof of Theorem~\ref{thm-multiscaleBKT}.} Another application of the associated graded algebra and Mann's theorem as in Step 4 helps us finish the proof of Theorem~\ref{thm-multiscaleBKT}.

{\bf Step 7.} {\em Proof of Theorem~\ref{t:ThickSegment}.} The main shortcoming of Theorem~\ref{thm-multiscaleBKT} is on  the assumption that there is an element $a\in A-A$ whose $\pfr$-adic valuation is 1. In fact, the set $A$ might be in a smaller field, or the ring $R$ generated by $A$ might behave as the ring of integers of different subfields in different scales. 

Having a description of closed subrings of $\ocal$ in hand, using a corollary of Theorem~\ref{t:BoundedGenerationScalarSum}, we follow scheme of Bourgain's proof in~\cite[Section A.3]{BG}. It is worth repeating that Bourgain had assumed $p$ is a fixed prime and $K/\bbq_p$ is not widely ramified; and both of these assumptions were utilized for understanding structure of certain subrings of $\ocal$.  

{\em Assuming the ring $R$ generated by $A$ is $\ocal$.} Our starting point is where there is no complication on the ring $R$. The following is an immediate corollary of Theorem~\ref{t:BoundedGenerationScalarSum}, which is a uniform version (in the sense that the implied constants do not depend on the field $K$) of \cite[Corollary A.1]{BG}.
\begin{cor}\label{c:BoundedGenerationFullModP}
	For any $0<\vare_1\ll \vare_2\ll 1$, $0<\delta\ll_{\vare_1}1$, and positive integer $1\ll_{\vare_1} C$, and any finite extension $K$ of $\bbq_p$ with large, depending on $\vare_1$, residue field $\f$ the following holds:
	
	Let $\ocal$ be the ring of integers of $K$, and $\pfr$ be a uniformizing element of $K$. Suppose $A\subseteq \pi_{\pfr^N}(\ocal)$ such that
	\begin{enumerate}
		\item $|\pi_{\pfr^i}(A)|\ge |\f|^{i\vare_1}$ for any $N\delta\le i\le N$.
		\item $\pi_{\pfr^{e'}}(A)=\pi_{\pfr^{e'}}(\ocal)$, where $e'=1$ if $K$ is an unramified extension, and $e'=2$ otherwise. 
	\end{enumerate}
	Then 
	\[
	\textstyle
	\pi_{\pfr^N}(\pfr^{\lceil\vare_2 N\rceil}\ocal)\subseteq \sum_{C_1}\prod_{C_2}A-\sum_{C_1}\prod_{C_2}A,
	\]
	for some integers $C_1,C_2\le C$. 
\end{cor}

{\em Assuming the first $N$ grades of $R$ with respect to powers of $p$ have equal cardinality.} In Proposition~\ref{prop-bounded-generation-for-equal-grades}, we consider the case where $|{\rm gr}_{0,p}(R;\ocal)|=\cdots=|{\rm gr}_{N-1,p}(R;\ocal)|$. Using the description of closed subrings of $\ocal$, we deduce that $\pi_{p^{N-4}}(R)=\pi_{p^{N-4}}(\ocal_0)$ where $\ocal_0$ is the ring of integers of a closed subfield if $N$ is large compared to {\em the degree} of the field extension $K/\bbq_p$. Then we use Corollary~\ref{c:BoundedGenerationFullModP} to get the desired bounded generation. For small $N$, Bourgain-Katz-Tao's sum-product result for finite fields is used. 

{\em Bourgain's technique of detecting mutations.} In Lemma~\ref{lem-bounded-generation-large-mod-all-powers-till-N}, under the assumption that the box dimension of $A$ is at least $\vare$ at the scale $|\pfr^i|$ for any integer $i$ in $[0,N-1]$, we prove a bounded generation result. It is clear that the cardinality of the grades ${\rm gr}_{i,p}(R;\ocal)$ of the ring $R$ generated by $A$ with respect to powers of $p$ is a non-decreasing sequence; and it has at most $[K:\bbq_p]$ many jumps. So at least one of the ranges where equality of grades occurs is {\em large}. And one would hope to zoom in this portion and use the previous step. However there are two important issues: (1) we need to rescale a subset of $A$ in order to zoom in to the equal grade portion; and this changes the ring; (2) it is not clear why we can get any element at the particular grade where the mutation occurred only in a {\em bounded} number of steps. For these reasons, we recursively define a sequence of quadruples $(A_i,R_i,n_i,C_i)$ where $R_i$ is the ring generated by $A_i$; the first $n_i$ grades of $R_i$ have the same cardinality; in $C_i$ steps $A_i$ generates a {\em large} ideal of $R_i$ at the scale $|p^{n_i}|$; $\pi_{\pfr}(A_i)=\pi_{\pfr}(R_i)$; and $|\pi_p(R_i)|$ is getting larger (see Step 0 of the proof of Lemma~\ref{lem-bounded-generation-large-mod-all-powers-till-N}). Having this sequence, we get the needed bounded generation result.  

{\em Finishing proof of Theorem~\ref{t:ThickSegment}.} For small residue fields, we use Bourgain's result. We notice that Bourgain had assumed that $p$ is a fixed {\em large} prime; but in the proof the largeness of $p$ is only used to ensure the extension $K/\bbq_p$ is not widely ramified. This, in turn, is only used to understand closed subrings of $\ocal$. So results of Section~\ref{s:subrings} of this note automatically extends Bourgain's result to any fixed prime $p$. For large residue fields, another application of the regularization technique gives us a set where we can use the previous step and deduce the desired result.

\subsection{Notation.}
In this note, $K$ is a finite extension of $\bbq_p$, $\ocal$ is its ring of integers, $\pfr|p$ is a uniformizing element, $\f:=\ocal/\langle \pfr \rangle$ is its residue field, and $e$ is the ramification index of $K$ over $\bbq_p$, i.e. $\langle p\rangle=\langle \pfr^e\rangle$. For any ring $R$ and $a\in R$, $\pi_a:R\rightarrow R/\langle a \rangle$ is the canonical quotient map $\pi_a(x):=x+\langle a \rangle$. 

We use the usual Vinogradov notation: $x\gg y$ means that there is a universal positive constant $c$ such that $x\ge cy$, and $x\gg_{z_1,z_2} y$ means that there is a positive function $c(z_1,z_2)$ of $z$ such that $x\ge c(z_1,z_2) y$. 

For a subset $A$ of a ring $R$ and a positive integer $C$, let $\sum_C A:=\{\sum_{i=1}^C a_i|\h a_i\in A\}$, $\prod_CA :=\{\prod_{i=1}^C a_i|\h a_i\in A\}$, and $\gen{A}{C}:=\{\sum_{i=1}^C\prod_{j=1}^C a_{ij}-\sum_{i=1}^C\prod_{j=1}^C a_{ij}'|\h a_{ij},a_{ij}'\in A\}$; and $\langle A\rangle$ is the subring generated by $A$. 
 
\section*{Acknowledgements} I am in debt of P. Varj\'{u} for explaining to me his joint work in progress with E. Lindenstrauss, where they give a new proof of a sum-product result for $\bbz/2^N\bbz$. I would like to thank E. Lindenstrauss for the fruitful conversations that helped me to strengthen my initial results. I would like to again thank E. Lindenstrauss and P. Varj\'{u} for allowing me to include slight variation of some of their arguments before the completion of their work. I am thankful to A. Mohammadi and K. Kedlaya for helpful conversations. And finally I would like to express my gratitude to J. Bourgain for our enlightening communications regarding this problem.

\section{Structure of subrings}\label{s:subrings}
In this section we describe structure of subrings of the ring $\ocal$ of integers of a finite extension $K$ of $\bbq_p$ (see Theorem~\ref{t:SubringOfO}). To do so first graded subrings of the ring of polynomials over a finite field are described (see Proposition~\ref{p:SubringOfPolynomialRing}). Along the way a result for semigroups of non-negative integers is proved that is of independent interest (see Proposition~\ref{p:NumercialSemigroups}). It should be said that Theorem~\ref{t:SubringOfO} will be used only towards the end of the article in the proof of Theorem~\ref{t:ThickSegment}. But the technique of assigning a graded set to a subset of $\ocal$ in order to understand the shape of this set in various scales is well illustrated in this section. One can take Theorem~\ref{t:SubringOfO} as a blackbox and skip this section; but the basic properties proved and notation introduced in Lemma~\ref{lem-attached-graded-rings} and Corollary~\ref{cor-grade-shift} should be reviewed. 
\subsection{Structure of graded subrings of the ring of polynomials over a finite field.}

The main goal of this section is to prove the following proposition.
\begin{prop}\label{p:SubringOfPolynomialRing}
Let $\f$ be a finite field, and $S:=\bigoplus_{i=0}^{\infty} S_i t^i$ be a graded subring of the ring of polynomials $\f[t]$ over the field $\f$; that means $S_i$ is an additive subgroup of $\f$ and $S_iS_j\subseteq S_{i+j}$, for any $i,j$; in particular $S_0$ is a subfield of $\f$. Suppose 
\be\label{eq-subring-parameters}
[\f:S_0]:=d\hspace{1cm}\text{{\rm  and}}\hspace{1cm}
S_m\neq 0.
\ee  
For any integer $C\ge 3$ and any integer $F\gg_{C,m,d} 1$, there are integers $r$, $a$, and $b$, an element $\lambda\in \f$, and a subfield $\f_0$ of $\f$ such that $r|m$, $S_0\subseteq \f_0$,
\begin{align}
	\label{eq-thickness} b-a\gg_{C,d,m} F ,&\h{\rm and }\h b\ge Ca,\\
	\label{eq-ring-structure} \pi_{t^{bm}}(S\cap t^{am}\f[t])&=\pi_{t^{bm}}(t^{am}\f_0[\lambda t^r])
\end{align}
where $\pi_{t^k}:\f[t]\rightarrow \f[t]/t^k\f[t]$ is the canonical quotient map and $\f_0[\lambda t^r]:=\bigoplus_{i=0}^{\infty} \f_0(\lambda t^r)^i$.
\end{prop}
   Proposition \ref{p:SubringOfPolynomialRing} has an immediate implication for numerical semigroups that is of independent interest.
\begin{minipage}{.84\textwidth}
\begin{prop}\label{p:NumercialSemigroups}
	Let $J$ be a subsemigroup of non-negative integers.  Suppose the greatest common divisor of elements of $J$ is one (such a semigroup is called a numerical semigroup).
 Let $m:=\min J\setminus \{0\}$ (it is called the multiplicity of $J$), and let $F:=\max \bbz\setminus J$ (it is called the Frobenius number of $J$). Suppose $C\ge 3$ and $F\gg_{C,m} 1$. Then there are integers $a,b$, and $d$, such that $d|m$, $d>1$, 
\begin{align}
	\label{eq-thickness-semigroup} b-a\gg_{C,m} F &,\h \text{ and }\h\h b\ge Ca,\\
	\label{eq-semigroup-arithmetic-progression} 	J\cap [a,b)&=d\bbz\cap [a,b).
\end{align}
	\end{prop} 
\begin{proof}
Let $S:=\bigoplus_{j\in J}\f (t^j)$. Since $J$ is a subsemigroup of non-negative integers, $S$ is a graded subring of $\f[t]$. Since $S_0=\f$, by Proposition~\ref{p:SubringOfPolynomialRing}, there are integers $r$, $a'$, and $b'$ such that $d|m$ 
\be\label{eq-ring-to-semigroup}
b'-a'\gg_{C,m} F,  b'\ge Ca',{\rm and }\h
\pi_{t^{b'm}}(S\cap t^{a'm}\f[t])=\pi_{t^{b'm}}(t^{a'm}\f[t^d]).
\ee  	
Since $\pi_{t^{b'm}}(S\cap t^{a'm}\f[t])=\pi_{t^{b'm}}(\oplus_{j\in J\cap [a'm,b'm)}\f(t^j))$, by \eqref{eq-ring-to-semigroup} we deduce that 
\[
J\cap [a'm,b'm)=d\bbz \cap [a'm,b'm);
\]
and the claim follows.
\end{proof}
The following picture shows us a numerical semigroup $J$ with the multiplicity $m(J)=6$ and the Frobenius number $F(J)=133$. We can see that the set of non-negative integers can be covered with at most $m(J)$ windows with different {\em patterns}. And one of the {\em large} windows is going to have a {\em regular strip} pattern.
\end{minipage}
\begin{minipage}{.16\textwidth} 
\begin{tikzpicture}
\def\m{6}
\def\l{24}
\def\t{.35}
\pgfmathparse{\t*\m+.25};
\xdef\xfinal{\pgfmathresult}; 
\pgfmathparse{\t*\l+.25};
\xdef\yfinal{\pgfmathresult};         
\clip (-.25,-1) rectangle (\xfinal,\yfinal);

\foreach \i in {0,107,32,123,16,139}
	{
	\edef\q{0};
	\pgfmathparse{floor(\i/\m)};
	\xdef\q{\pgfmathresult};
	\pgfmathparse{(\l-\q)*\t};
	\xdef\y{\pgfmathresult};
	\draw [thick, draw=red] (-\t,\y)--(\m,\y);
	\edef\upperbound{0};
	\pgfmathparse{floor((\l*\m-\i-1)/\m)};
	\xdef\upperbound{\pgfmathresult};
	\foreach \counter in {0,...,\upperbound}
		{
		\edef\j{0};
		\pgfmathparse{\counter*\m+\i};
		\xdef\j{\pgfmathresult};
		\edef\r{0};
		\edef\q{0};
		\pgfmathparse{floor(\j/\m)};
		\xdef\q{\pgfmathresult};
		\pgfmathparse{\j-\q*\m};
		\xdef\r{\pgfmathresult};
		\pgfmathparse{(\l-\q)*\t};
		\xdef\y{\pgfmathresult};
		\pgfmathparse{\r*\t};
		\xdef\x{\pgfmathresult};
		\draw [ultra thin, draw=black, fill=blue, fill opacity=0.2]
		(\x,\y-\t) rectangle +(\t,\t);
		};
	};
\draw[step=\t,gray,very thin](0,0) grid (\m*\t,\l*\t);	
\pgfmathparse{int(\m-1)};
\xdef\m{\pgfmathresult};
\pgfmathparse{int(\l-1)};
\xdef\l{\pgfmathresult};
\edef\v{int(0)};
\pgfmathparse{int(\l*(\m+1)-1)};
\xdef\v{\pgfmathresult};
\foreach \y in {0,...,\l}
	{
	\foreach \x in {0,...,\m} 
		{
		\pgfmathparse{int(\v+1)};
        \xdef\v{\pgfmathresult};
        \pgfmathparse{(\x+.5)*\t};
        \xdef\c{\pgfmathresult};
        \pgfmathparse{(\y+.5)*\t};
        \xdef\d{\pgfmathresult};        
		\node at (\c,\d) {\tiny $\v$};		
		};
	\pgfmathparse{\v-2*\m-2};
    \xdef\v{\pgfmathresult};
	};
\pgfmathparse{7*\t};
\xdef\ylowercorner{\pgfmathresult};
\pgfmathparse{(\m+2)*\t};
\xdef\xuppercorner{\pgfmathresult}; 
\pgfmathparse{(\l-4)*\t};
\xdef\yuppercorner{\pgfmathresult}; 
\draw[pattern=north east lines, pattern color=orange, fill opacity=0.4] (-\t,\ylowercorner) rectangle (\xuppercorner,\yuppercorner);
\pgfmathparse{((\m+1)/2)*\t};
\xdef\center{\pgfmathresult}; 
\draw (\center,-.25) node {\tiny{Large window where}}; 
\draw (\center,-.5) node {\tiny{exactly multiples of}}; 
\draw (\center,-.75) node {\tiny{$d=2$ can be seen.}};
\end{tikzpicture}
\end{minipage}  
To prove Proposition~\ref{p:SubringOfPolynomialRing}, we start with the following combinatorial lemma.
\begin{lem}\label{l:combinatorial-lemma-deg}
	Suppose $\{d_i\}_{i=0}^{\infty}$ is a sequence of non-negative integers with the following properties.
	\begin{enumerate}
	\item[(A1)] For a positive integer $d$ we have $0\le d_i\le d$ for any $i$.
	\item[(A2)] If $d_i\neq 0$, then for any non-negative integer $j$ we have $d_j\le d_{i+j}$.
	\item[(A3)] For some positive integer $m$, $d_m$ is not zero; and $d_0$ is not zero.  
	\end{enumerate}
	Then for any integers $C\ge 2$ and $F\ge C^{md+1}$ there are integers $a$, $b$, and $x_0,\ldots,x_{m-1}$ such that 
	\begin{enumerate}
		\item[(C1)]($m$-periodic) For any $i\in [am,bm)$, $d_i=x_{r_i}$ where $r_i$ is the remainder of $i$ divided by $m$.
		\item[(C2)](Length of periodicity) $b\ge C a$, and $b\ge C^{-md-1} F$. 
	\end{enumerate}
\end{lem}
\begin{remark}
	It is worth mentioning that, if $d=2$, then the assumptions (A1) and (A2) are equivalent to saying that the map $i\mapsto d_i$ is the characteristic function of a subsemigroup of non-negative integers. 
\end{remark}
\begin{proof}[Proof of Lemma~\ref{l:combinatorial-lemma-deg}]
Since $d_m\neq 0$ (see (A3)), we have that $i\mapsto d_{mi+j}$ is increasing for any non-negative integer $j$. And so for non-negative integers $j$ and $l$ the set $\{i\in \bbz^{\ge 0}|\h d_{im+j}=l\}$ is a segment of integers. Since $0\le d_{mi+j}\le d$, for any $j$ there is a partition $P_j:=\{I_{0,j},\ldots,I_{d,j}\}$ of the set of non-negative integers such that for any $i\in I_{l,j}$ we have $d_{mi+j}=l$ (some of $I_{l,j}$'s might be empty). Collecting all the end points of the intervals $I_{l,j}\cap [0,F]$ (for $0\le l\le d$ and $0\le j\le m-1$) we get integers $0=:g_0\le g_1\le \cdots \le g_{md}\le g_{md+1}:=F$ such that for any index $j$ and integer $i\in [g_jm,g_{j+1}m)$ the number $d_i$ only depends on the remainder of $i$ divided by $m$.  

{\bf Claim 1.} There is an integer $j\in [0,md]$ such that $(g(j+1)+1)\ge C (g(j)+1)$ . 

{\em Proof of Claim 1.} Suppose to the contrary that for any $j\in [1,md]$ we have $(g(j+1)+1)<C(g(j)+1)$; then inductively we get that $(g(md+1)+1)<C^{j}(g(md+1-j)+1)$ for any integer $j\in [1,md+1]$. And for $j=md+1$, we get that $F+1<C^{md+1}$; and this gives us a contradiction.

{\bf Claim 2.} There is an integer $j\in [1,md]$ such that $(g(j+1)+1)\ge C(g(j)+1)$ and $g(j+1)\ge C^{-md-1} F$.

{\em Proof of Claim 2.} Let $j_0$ be the largest integer in $[0,md]$ such that $(g(j_0+1)+1)\ge C(g(j_0)+1)$ (using Claim 1 we know that such an integer exists). Since $j_0$ is the largest integer with this property, we have 
\[
(g(i+1)+1)< C (g(i)+1),\h \text{ for any integer }i\in [j_0+1,md].
\] 
Hence $g(md+1)+1<C^{md-j_0} (g(j_0+1)+1)<C^{md}(g(j_0+1)+1)$; and this implies 
\[
C^{-md-1}F< C^{-md} F+C^{-md}-1< g(j_0+1).
\]

Let $b:=g(j_0+1)$ and $a:=g(j_0)$. Claim 1 and Claim 2 imply that $a$ and $b$ satisfy the conclusion (C2) in the statement of Lemma~\ref{l:combinatorial-lemma-deg}. For an integer $j\in [0,m)$, let $x_j:=d_{am+j}$. Since for any integer $i\in [am,bm)$ the number $d_i$ only depends on the remainder $r_i$ of $i$ divided by $m$, we get that $d_i=x_{r_i}$ for any integer $i\in [am,bm)$; and the claim follows.  
\end{proof}
\begin{proof}[Proof of Proposition~\ref{p:SubringOfPolynomialRing}]
Since $S=\bigoplus_{i=0}^{\infty} S_it^i$ is a graded subring of the ring of polynomials $\f[t]$ over a finite field, we have that $S_0$ is a subfield of $\f$ and $S_i$ is a $S_0$-vector space for any non-negative integer $i$	. Let $d_i:=\dim_{S_0} S_i$. Next we check that the assumptions (A1), (A2), and (A3) of Lemma~\ref{l:combinatorial-lemma-deg} hold for the sequence $\{d_i\}_{i=0}^{\infty}$. As $S_i$'s are $S_0$-subspaces of $\f$, we get that $\dim_{S_0} S_i\le \dim_{S_0}\f$; and so the the assumption (A1) of Lemma~\ref{l:combinatorial-lemma-deg} holds for the sequence $\{d_i\}_{i=0}^{\infty}$. Since for non-negative integers $i,j$ we have $S_iS_j\subseteq S_{i+j}$, we get that the assumption (A2) of Lemma~\ref{l:combinatorial-lemma-deg} holds for the sequence $\{d_i\}_{i=0}^{\infty}$. And because of \eqref{eq-subring-parameters}, the assumption (A3) of Lemma ~\ref{l:combinatorial-lemma-deg} holds for the sequence $\{d_i\}_{i=0}^{\infty}$. So by Lemma~\ref{l:combinatorial-lemma-deg}, there are integers $a$, $b$, $x_0,\ldots, x_{m-1}$ that satisfy the conclusions (C1) and (C2) of Lemma~\ref{l:combinatorial-lemma-deg}. 

Let $\lambda\in S_m\setminus\{0\}$ and $\sbar_i:=\lambda^{-\lfloor i/m\rfloor} S_i$. Then $1\in \sbar_m$; and so $\sbar_i\subseteq \sbar_m\sbar_i$. On the other hand, $S_mS_i\subseteq S_{i+m}$ implies that $\sbar_m\sbar_i\subseteq \sbar_{i+m}$. For any integer $i\in [ma,m(b-1))$, we have that $\dim_{S_0} \sbar_{i+m}=d_{i+m}=d_i=\dim_{S_0}\sbar_i$. Altogether we get that 
\be\label{eq-equality-grades}
\sbar_i=\sbar_{i+m}, \hspace{.5cm} \text{ for any integer i}\in [am, (b-1)m).
\ee
Let $B_{i}:=\sbar_{am+i}$ for any integer $i\in [0,m)$. And so by \eqref{eq-equality-grades} we have that for any integer $i\in [am,bm)$, $\sbar_i=B_{r_i}$ where $r_i$ is the remainder of $i$ divided by $m$. We will identify the set of integers in the interval $[0,m)$ with $\bbz/m\bbz$, and in the rest of the proof $\pi_m(i):=i+m\bbz$ will be used instead of $r_i$; that means we have $B_{\pi_m(i)}=\sbar_i$ for any integer $i\in [am,bm)$. 

{\bf Claim.} Let $I:=\{\bar{i}\in \bbz/m\bbz|\h B_{\bar i}\neq 0\}$. Then $I$ is a subgroup of $\bbz/m\bbz$.

{\em Proof of Claim.} Suppose $i,j$ are two integers in $[0,m)$ and $\pi_m(i),\pi_m(j)\in I$. Then $B_{\pi_m(i)}$ and $B_{\pi_m(j)}$ are non-zero $S_0$-vector spaces. Since $(2a+1)m<Cam\le bm$, we have 
\be\label{eq-grading}
S_{am+i}=\lambda^{\lfloor i/m\rfloor} B_{\pi_m(i)}, \h
S_{am+j}=\lambda^{\lfloor j/m\rfloor} B_{\pi_m(j)}, \h \text{and }
S_{2am+i+j}=\lambda^{\lfloor i+j/m\rfloor} B_{\pi_m(i+j)}; 
\ee
and so the grading equation $S_{am+i}S_{am+j}\subseteq S_{2am+i+j}$ implies that   
\be\label{eq-translation-of-grading}
0\neq B_{\pi_m(i)}B_{\pi_m(j)}\subseteq \lambda^{(\lfloor i+j/m\rfloor-\lfloor i/m\rfloor-\lfloor j/m\rfloor)}B_{\pi_m(i+j)}=\lambda^{\lfloor i+j/m\rfloor}B_{\pi_m(i+j)};
\ee
in particular $\pi_m(i+j)=\pi_m(i)+\pi_m(j)\in I$. And the claim follows.

Suppose $r$ is the positive integer in $[0,m)$ such that $I=r\bbz/m\bbz$; in particular $r|m$. Let $s:=m/r$. Suppose $\eta\in B_{r}\setminus\{0\}$, and let $\bbar_i:= \eta^{-i} B_{\pi_m(ri)}$ for any integer $i$ in $[0,s)$. So \eqref{eq-translation-of-grading} implies that 
$\bbar_0 \cdot \bbar_i \eta^i \subseteq \bbar_i \eta^i$ for any integer $i$ in $[0,s)$. Therefore we get 
\be\label{eq-grade0}
\bbar_i=\bbar_0\cdot \bbar_i,
\ee
for any integer $i$ in $[0,s)$; in particular $\bbar_0$ is a subfield of $\f$, and $\bbar_i$'s are $\bbar_0$-vector spaces. Another application of \eqref{eq-translation-of-grading} implies that $\bbar_1\eta \cdot \bbar_i \eta^i \subseteq \bbar_{i+1} \eta^{i+1}$ for any integer $i$ in $[0,s-1)$. Therefore we get 
\be\label{eq-grade1}
\bbar_1\cdot \bbar_i \subseteq \bbar_{i+1},
\ee
for any integer $i$ in $[0,s-1)$. By \eqref{eq-grade0}, \eqref{eq-grade1}, and the fact that $1\in \bbar_1$, we deduce that
\be\label{eq-grades-are-increasing}
\bbar_0\subseteq \bbar_1 \subseteq \cdots \subseteq \bbar_{s-1}.
\ee 
Next we again use \eqref{eq-translation-of-grading} to deduce that $\bbar_1\eta \cdot \bbar_{s-1}\eta^{s-1}\subseteq \lambda \bbar_0$. Therefore we have 
\be\label{eq-first-jump-in-the-degree}
\lambda^{-1}\eta^{s} \bbar_1\bbar_{s-1}\subseteq \bbar_0.
\ee 
By \eqref{eq-grades-are-increasing}, \eqref{eq-first-jump-in-the-degree}, the fact that $1\in \bbar_1$, and comparing the dimensions, we get that there is a subfield $\f_0$ of $\f$ such that 
\be\label{eq-grades-are-the-same-field}
\bbar_0=\bbar_1=\cdots=\bbar_{s-1}=\f_0,\h\text{ and }\h \lambda \f_0^{\times}=\eta^s\f_0^{\times}.
\ee 
Since $\bbar_0=\lambda^{-m}S_{am}\supseteq S_0$, we get that $\f_0$ is a field extension of $S_0$. By \eqref{eq-grading}, $\bbar_i:=\eta^{-i}B_{\pi_{m}(ri)}$, \eqref{eq-grades-are-the-same-field}, $\bbar_i:=\eta^{-i}B_{\pi_{m}(ri)}$, and $m=rs$, we have
\begin{align}\label{eq-part-of-smaller-polynomial-ring}
\bigoplus_{i=am}^{bm-1} S_i t^i&=\bigoplus_{i=a}^{b-1}\left(\bbar_0\oplus \bbar_1\eta t^r \oplus \cdots \oplus \bbar_{s-1}\eta^{s-1}t^{r(s-1)}\right) (\lambda t^m)^i \\
\notag &=\bigoplus_{i=as}^{bs-1} \f_0 (\eta t^r)^i.
\end{align}
By \eqref{eq-part-of-smaller-polynomial-ring}, we get \eqref{eq-ring-structure}; and since $a$ and $b$ satisfy the conclusions (C1) and (C2) of Lemma~\ref{l:combinatorial-lemma-deg}, we get \eqref{eq-thickness}. And the claim follows.
\end{proof}
 
\subsection{Structure of subrings of the ring $\ocal$ of integers of a finite extension $K$ of $\bbq_p$.} The main goal of this section is to prove Theorem~\ref{t:SubringOfO}; but before we get to that, let us define certain graded algebras which are crucial throughout this note.  For any $\eta\in \pfr\ocal$, we get a filtration \(\{\eta^i\ocal\}_{i=0}^{\infty}\) of \(\ocal\); and then we can define a corresponding graded algebra: let 
${\rm gr}_{\eta}(\ocal):=\bigoplus_{i=0}^{\infty} {\rm gr}_{i,\eta}(\ocal)$ where ${\rm gr}_{i,\eta}:=\pi_{\eta^{i+1}}(\eta^i\ocal)$. For any $x\in \ocal$, there is a unique non-negative integer $i$ such that $x\in \eta^i\ocal\setminus \eta^{i+1}\ocal$; we denote such non-negative integer by $\deg_{\eta}(x)$, and it is clear that $\deg_{\eta}(x)=\lfloor v_{\pfr}(x)/v_{\pfr}(\eta)\rfloor$. For any $x\in \ocal$ and $\eta\in \pfr\ocal$, we let $l_{\eta}(x):=\pi_{\pfr^{\deg_{\eta}+1}}(x)$ and $\overline{l}_{\eta}(x):=\pi_{\eta}(\eta^{-\deg_{\eta}(x)}x)\in \pi_{\eta}(\ocal)$ and call $\overline{l}_{\eta}(x)$ the $\eta$-leading term of $x$. The following is a useful lemma that justifies the use of the above terminology for $\overline{l}_{\eta}(x)$. 
\begin{lem}\label{lem-attached-graded-rings}
	In the above setting, let $\phi_{\eta}:{\rm gr}_{\eta}(\ocal)\rightarrow \pi_{\eta}(\ocal)[t]$ be the graded map induced by 
	\[
	\phi_{\eta}(l_{\eta}(x)):=\overline{l}_{\eta}(x)t^{\deg_{\eta} x},
	\]
	 where $\pi_{\eta}(\ocal)[t]$ is the ring of polynomials over $\pi_{\eta}(\ocal)$. Then $\phi_{\eta}$ is a graded ring isomorphism.  
\end{lem}
\begin{proof}
	Suppose $\overline{x}\in {\rm gr}_{i,\eta}(\ocal)$; then $\overline{x}=\pi_{\pfr^{i+1}}(x)$ for some $x\in \eta^i\ocal$. And we have $\phi_{\eta}(\overline{x})=\pi_{\pfr}(\eta^{-i}x)t^i$. Therefore for $\overline{x},\overline{x'}\in {\rm gr}_{i,\eta}(\ocal)$, there are $x,x'\in \eta^i\ocal$ such that $\pi_{\pfr^i}(x)=\overline{x}$ and $\pi_{\pfr^i}(x')=\overline{x}'$; and 
	\[
	\phi_{\eta}(\overline{x}+\overline{x}')=\phi_{\eta}(\pi_{\pfr^{i+1}}(x+x'))=\pi_{\eta}(\eta^{-i}(x+x'))t^i=\pi_{\eta}(\eta^{-i}x)t^i+\pi_{\eta}(\eta^{-i}x')t^i=\phi_{\eta}(\overline{x})+\phi_{\eta}(\overline{x}').
	\]
	Since $\ocal$ is an integral domain, $\pi_{\eta^{i+1}}(x)\mapsto \pi_{\eta}(\eta^{-i}x)$ is a bijection from $\pi_{\eta^{i+1}}(\eta^i\ocal)$ to $\pi_{\eta}(\ocal)$. Hence $\phi_{\eta}$ is an additive group isomorphism from ${\rm gr}_{\eta}(\ocal)$ to $\pi_{\eta}(\ocal)[t]$.
	
	For $\overline{x}\in {\rm gr}_{i,\eta}(\ocal)$ and $\overline{x}'\in {\rm gr}_{j,\eta}(\ocal)$, we have that $\overline{x}=\pi_{\eta^{i+1}}(x)$ and $\overline{x}'=\pi_{\eta^{j+1}}(x')$ for some $x\in \eta^i\ocal$ and $x'\in \eta^j\ocal$; and $\phi_{\eta}(\overline{x})=\pi_{\eta}(\eta^{-i}x)t^i$ and $\phi_{\eta}(\overline{x}')=\pi_{\eta}(\eta^{-j}x')t^j$. 	
	Then based on the graded structure of ${\rm gr}_{\eta}(\ocal)$, we have $\overline{x}\overline{x}'=\pi_{\eta^{i+j+1}}(xx')$; and so $\phi_{\eta}(\overline{x}\overline{x}')=\pi_{\eta}(xx')t^{i+j}=\phi_{\eta}(\overline{x})\phi_{\eta}(\overline{x}')$. And the claim follows.
\end{proof}
\begin{cor}\label{cor-grade-shift}
Suppose $x\in\ocal\setminus\pfr\ocal$; then for any non-negative integers $i$ and $k$ multiplication by $x\eta^i$ induces a bijection $x_k+\eta^{k+1}\ocal\mapsto x_kx\eta^i+\eta^{k+i+1}\ocal$ from ${\rm gr}_{k,\eta}(\ocal)$ to ${\rm gr}_{k+i,\eta}(\ocal)$.   
\end{cor}
\begin{proof}
We notice that $x$ is a unit in $\ocal$; and so $\pi_{\eta}(x)$ is a unit in $\pi_{\eta}(\ocal)$; and so multiplication by $\pi_{\eta}(x)$ induces a bijection from $\pi_{\eta}(\ocal)$ to itself. And by Lemma~\ref{lem-attached-graded-rings} the claim follows. 
\end{proof}

To any subset $X$ of $\ocal$ we associate a graded subset ${\rm gr}_{\eta}(X;\ocal)$ of ${\rm gr}_{\eta}(\ocal)$. For a non-negative integer $i$, we let ${\rm gr}_{i,\eta}(X;\ocal):=\pi_{\eta^{i+1}}(X\cap \eta^i\ocal)$; and we define ${\rm gr}_{\eta}(X;\ocal):=\bigoplus_{i=0}^{\infty} {\rm gr}_{i,\eta}(X;\ocal)$. 

\begin{lem}\label{lem-graded-subrings}
	In the above setting, suppose $R$ is a subring of $\ocal$; then ${\rm gr}_{\eta}(R;\ocal)$ is a subring of ${\rm gr}_{\eta}(\ocal)$.
\end{lem}
\begin{proof}
It is clear.	
\end{proof}
\begin{proof}[Proof of Theorem~\ref{t:SubringOfO}]
By Lemma~\ref{lem-graded-subrings}, we have that ${\rm gr}_{\pfr}(R;\ocal)$ is a subring of ${\rm gr}_{\pfr}(\ocal)$; and by Lemma~\ref{lem-attached-graded-rings} there is a graded algebra isomorphism $\phi:{\rm gr}_{\pfr}(\ocal)\rightarrow \f[t]$ where $\f$ is the residue field of $\ocal$. Since $1$ is in $R$, we have that ${\rm gr}_{0,\pfr}(R;\ocal):=\pi_{\pfr}(R)$ is a subfield of $\f$ and $p\in R$. The latter implies that ${\rm gr}_{e,\pfr}(R;\ocal)\neq 0$ where $e:=v_{\pfr}(p)$ is the ramification index of $K$ over $\bbq_p$. It is well-known that the degree $[K:\bbq_p]$ of the field extension $K$ of $\bbq_p$ is equal to the product of its ramification index $e$ and its residue degree $[\f:\f_p]$.    Hence by Proposition~\ref{p:SubringOfPolynomialRing} there are a subfield $\f_0$ of $\f$ which is an extension of ${\rm gr}_{0,\pfr}(R;\ocal)$, $\lambda\in \f$, and positive integers $r,a,b$ such that $r|e$ and 
\be\label{eq-at-the-level-of-grading}
b-a\gg_{C,[K:\bbq_p]} F, \hspace{2mm} b\ge Ca, \text{ and } \bigoplus_{ae\le i<be} \phi({\rm gr}_{i,\pfr}(R;\ocal))=\bigoplus_{ae/r\le j< be/r} \f_0 (\lambda t^r)^j.
\ee
An immediate consequence of \eqref{eq-at-the-level-of-grading} is the existence of certain elements in $R$ which help us pass to smaller scales {\em within} $R$. By \eqref{eq-at-the-level-of-grading}, for any integer $j$ in $[1,e/r)$, there is $\wh{\xi}_j\in R$ such that $v_{\pfr}(\wh{\xi}_j)=ea+rj$. By letting $\wh{\xi}_0:=p^a$, we get the same claim for $j=0$ as well. For any integer $i$ in $[a,b)$ and integer $j$ in $[0,e/r)$, let $\wh{\xi}_{i,j}:=p^{i-a}\wh{\xi}_j$.  Hence for any such integers $i$ and $j$, we have 
\be\label{eq-shifting-elements-in-R}
\wh{\xi}_{i,j}\in R  \text{ and } v_{\pfr}(\wh{\xi}_{i,j})=ei+rj.
\ee
For any integer $k$ in $[0,(b-a)e/r)$, let $\xi_{k}:=p^{-a}\wh{\xi}_{\overline{q},\overline{r}}$ where $\overline{q}$ is the quotient of $rk+ea$ divided by $e$ and $r\overline{r}$ is the remainder of $rk+ea$ divided by $e$. Therefore for any such $k$, $\overline{q}$, and $\overline{r}$ by \eqref{eq-shifting-elements-in-R} we have
\be\label{eq-scaling-parameters}
p^a\xi_k\in R \text{ and }
v_{\pfr}(\xi_k)=-ea+v_{\pfr}(\wh{\xi}_{\overline{q},\overline{r}})=-ea+(e\overline{q}+r\overline{r})=rk.
\ee

Another consequences of \eqref{eq-at-the-level-of-grading} is the fact that ${\rm gr}_{i,\pfr}(R;\ocal)$ is zero if $r\nmid i$; and so for any non-negative integer $i$ we have 
\be\label{eq-intersection-goes-to-a-smaller-ball}
R\cap \pfr^i \ocal=R\cap \pfr^{r\lceil i/r\rceil} \ocal.
\ee

And the last immediate consequence of \eqref{eq-at-the-level-of-grading} (see Corollary~\ref{cor-grade-shift}) that we mention here is that for any integer $i$ in $[ae,(b-1)e)$ 
\be\label{eq-bijection-multiplying-by-p}
x+\pfr^{i+1}\ocal\mapsto px+\pfr^{e+i+1}\ocal \text{ is a bijection from } {\rm gr}_{i,\pfr}(R;\ocal) \text{ to }{\rm gr}_{e+i,\pfr}(R;\ocal).
\ee
By \eqref{eq-at-the-level-of-grading} we also deduce that there is a function $\wt{s}:\f_0\rightarrow R\cap p^a\ocal$ such that $\pi_{\pfr}(p^{-a}\wt{s}(x))=x$ for any $x\in \f_0$.

{\bf Claim 1.} Let $s:\f_0\rightarrow \ocal, s(x):=p^{-a}\wt{s}(x)$ and $r$ be as in \eqref{eq-at-the-level-of-grading}. Then $\pi_{\pfr^r}\circ s:\f_0\rightarrow \pi_{\pfr^r}(\ocal)$ is a ring embedding.

{\em Proof of Claim 1.} For any $x\in \f_0$, we have $\pi_{\pfr}(s(x))=x$. Hence for any $x_1,x_2\in \f_0$, we have $s(x_1+x_2)-s(x_1)-s(x_2)\in \pfr \ocal$; this implies that
\be\label{eq-initial-cancellation-addition}
\wt{s}(x_1+x_2)-\wt{s}(x_1)-\wt{s}(x_2)\in R\cap \pfr^{ea+1}\ocal.
\ee
Hence by \eqref{eq-intersection-goes-to-a-smaller-ball} we have 
\be\label{eq-r-level-cancellation-addition}
\wt{s}(x_1+x_2)-\wt{s}(x_1)-\wt{s}(x_2)\in R\cap \pfr^{ea+r}\ocal.
\ee
And so $\pi_{\pfr^r}(s(x_1+x_2))=\pi_{\pfr^r}(s(x_1))+\pi_{\pfr^r}(s(x_2))$.

We also have $\wt{s}(x_1)\wt{s}(x_2)\in R\cap p^{2a}\ocal$. By \eqref{eq-bijection-multiplying-by-p}, there is $x'\in R\cap p^{a}\ocal$ such that 
\be\label{eq-level-reduction}
p^ax'-\wt{s}(x_1)\wt{s}(x_2)\in R\cap \pfr^{2ae+1}\ocal.
\ee
We also have $\pi_{\pfr}(p^{-a}\wt{s}(x_1x_2))=x_1x_2=\pi_{\pfr}(p^{-a}\wt{s}(x_1))\pi_{\pfr}(p^{-a}\wt{s}(x_2))=\pi_{\pfr}(p^{-2a}\wt{s}(x_1)\wt{s}(x_2))$. And so by \eqref{eq-level-reduction} we get 
$\pi_{\pfr}(p^{-a}x')=\pi_{\pfr}(p^{-a}\wt{s}(x_1x_2))$. Hence we get 
\be\label{eq-level-reduction-for-multiplication}
p^a\wt{s}(x_1x_2)-\wt{s}(x_1)\wt{s}(x_2)\in R\cap \pfr^{2ea+1}\ocal.
\ee
By \eqref{eq-level-reduction-for-multiplication} and \eqref{eq-intersection-goes-to-a-smaller-ball} we have
\be\label{eq-multiplication-initial-reduction}
p^a \wt{s}(x_1x_2)-\wt{s}(x_1)\wt{s}(x_2)\in R\cap \pfr^{2ae+r}\ocal.
\ee
Hence $p^{-a}\wt{s}(x_1x_2)-(p^{-a}\wt{s}(x_1))(p^{-a}\wt{s}(x_2))\in \pfr^r\ocal$; and this implies that
\be\label{eq-r-level-cancellation-multiplication} 
\pi_{\pfr^r}(s(x_1x_2))=\pi_{\pfr^r}(s(x_1))\pi_{\pfr^r}(s(x_2));
\ee
and the claim follows. QED.

By Hensel's lemma we know that there is a group embedding $\theta:\f^{\times}\rightarrow \ocal^{\times}$ such that for any $x\in \f$ we have $\pi_{\pfr}(\theta(x))$. Next we show that a good approximation of a multiple of $\theta(\f_0^{\times})$ can be found in $R$. This will help us to get the {\em unramified part} of $K_0$.

{\bf Claim 2.} Suppose $a,b$ are given as in \eqref{eq-at-the-level-of-grading} for a constant $C$ which is more than $3[K:\bbq_p]$. Then there is $\wh{s}:\f_0^{\times}\rightarrow R\cap p^{2a} \ocal$ such that $\wh{s}(x)-p^{2a}\theta(x)\in p^{b-2a([\f_0:\f_p]-1)}\ocal$ for any $x\in \f_0^{\times}$.

{\em Proof of Claim 2.} Let us fix $x_0\in \f_0^{\times}$. We notice that $\theta(x_0)$ is a zero of $t^{|\f_0|-1}-1$; and so its (monic) minimal polynomial $\Phi(t)$ over $\bbq_p$ is actually in $\bbz_p[t]$ and $\pi_{\pfr}(\Phi(t))$ does not have multiple zeros in $\f$. And so by the Hensel's lemma, if $\Phi(y)\in \pfr^l\ocal$ for some $y\in \ocal$ and positive integer $l$ and $\pi_{\pfr}(y)=x_0$, then $y-\theta(x_0)\in \pfr^l\ocal$. Moreover we notice that $d:=\deg \Phi$ is at most the residue index $d_0:=[\f:\f_p]$ of $K$ over $\bbq_p$.  Next for integers $k$ in $[0,(b-2ad_0)e/r)$ we inductively find $y_k\in \ocal$ such that 
\be\label{eq-approximating-zero-of-cyclotomic}
\pi_{\pfr}(y_k)=x_0,\h\h y_k-y_{k-1}\in \pfr^{rk} \ocal,\h\h p^{2a} y_k \in R,
\text{ and } \Phi(y_k)\in \pfr^{r(k+1)}\ocal,
\ee
where $r$ is as in \eqref{eq-at-the-level-of-grading}. The argument is similar to the proof of Hensel's lemma but we have to be careful that every step of estimation is done within the ring $R$ and  only in certain scales we have a control on $R$. We write the Taylor expansion of $\Phi(t)$ at $y_k$,
\be\label{eq-HenselTaylorExpansion}
\Phi(y_k+t)=\Phi(y_k)+\sum_{i=1}^{d} \frac{\Phi^{(i)}(y_k)}{i!} t^i;
\ee
and we point out that, since the coefficients of $\Phi$ are in $\bbz_p$, we have that $\frac{\Phi^{(i)}(t)}{i!}$ is again in $\bbz_p[t]$. Therefore, for any integer $i$ in $[1,d]$, $p^{2ai} \frac{\Phi^{(i)}(y_k)}{i!}$  
is in the $\bbz_p$-algebra generated by $p^{2a}y_k$; and so by the induction hypothesis
\be\label{eq-still-in-R}
p^{2a(d-i)} \frac{\Phi^{(i)}(y_k)}{i!}\in R. \text{ And in addition } y_k\in \ocal \text{ implies that } \frac{\Phi^{(i)}(y_k)}{i!}\in \ocal.
\ee
As we mentioned earlier, $\Phi(t)$ is a divisor of $t^{|\f_0|-1}-1$; since $t^{|\f_0|-1}-1$ does not have multiple zeros in $\f$, $\Phi(t)$ does not have multiple zeros in $\f$. Hence $\pi_{\pfr}(\Phi(y_k))=0$ implies that \be\label{eq-derivative-has-norm-1}
\pi_{\pfr}(\Phi'(y_k))\neq 0.
\ee

By \eqref{eq-derivative-has-norm-1} and the induction hypothesis (see the last condition in \eqref{eq-approximating-zero-of-cyclotomic}), we deduce that there is $x_{k+1}\in \f_0$ such that 
\be\label{eq-step-k-refinement-hensel}
\Phi(y_k)+\Phi'(y_k)(\xi_{k+1} s(x_{k+1}))\in \pfr^{r(k+1)+1}\ocal,
\ee
where $\xi_{k+1}$ is as in \eqref{eq-scaling-parameters}. We also notice that by \eqref{eq-scaling-parameters} and the induction hypothesis (see the third condition in \eqref{eq-approximating-zero-of-cyclotomic}) we have
\be\label{eq-still-being-in-R}
p^{2a}(y_k+\xi_{k+1}s(x_{k+1}))=p^{2a}y_k+(p^a\xi_{k+1})(p^as(x_{k+1}))\in R.
\ee

By \eqref{eq-HenselTaylorExpansion}, \eqref{eq-still-in-R}, \eqref{eq-step-k-refinement-hensel}, and \eqref{eq-still-being-in-R} we have
\be\label{eq-almost-the-next-step-hensel}
\Phi(y_k+\xi_{k+1}s(x_{k+1}))=\Phi(y_k)+\Phi'(y_k)(\xi_{k+1} s(x_{k+1}))+\sum_{i=2}^{d}  \frac{\Phi^{(i)}(y_k)}{i!} (\xi_{k+1} s(x_{k+1}))^i\in \pfr^{r(k+1)+1}\ocal,
\ee 
and 
\be\label{eq-bing-in-R}
p^{2ad}\Phi(y_k+\xi_{k+1}s(x_{k+1}))=p^{2ad}\Phi(y_k)+\sum_{i=1}^{d} \left(p^{2a(d-i)}\frac{\Phi^{(i)}(y_k)}{i!}\right) (p^a\xi_{k+1})^i(p^as(x_{k+1}))^i\in R.
\ee
By \eqref{eq-almost-the-next-step-hensel} and \eqref{eq-bing-in-R}, we have 
$p^{2ad}\Phi(y_k+\xi_{k+1}s(x_{k+1}))\in R\cap \pfr^{2ead+r(k+1)+1}\ocal$; and so by \eqref{eq-intersection-goes-to-a-smaller-ball} and having $k+1<(b-2ad_0)e/r$ we deduce that $p^{2ad}\Phi_{d}(y_k+\xi_{k+1}s(x_{k+1}))\in R\cap \pfr^{2ead+r(k+2)}\ocal$. Hence we get 
\be\label{eq-the-next-step-hensel}
 \Phi(y_k+\xi_{k+1}s(x_{k+1}))\in \pfr^{r(k+2)}.
\ee
Therefore by \eqref{eq-scaling-parameters}, \eqref{eq-still-being-in-R}, and \eqref{eq-the-next-step-hensel} we get that $y_{k+1}:=y_k+\xi_{k+1}s(x_{k+1})$ satisfies all the conditions mentioned in \eqref{eq-approximating-zero-of-cyclotomic}.

We notice that $\pi_{\pfr}(y_k)=x_0$ and $\Phi(y_k)\in \pfr^{r(k+1)}\ocal$ imply that 
\be\label{eq-close-to-root-of-unity}
y_k-\theta(x_0)\in \pfr^{r(k+1)}\ocal.
\ee

Let $k_0:=(b-2ad_0)e/r-1$ and $\wh{s}(x_0):=p^{2a}y_{k_0}$. Then by \eqref{eq-approximating-zero-of-cyclotomic} and \eqref{eq-close-to-root-of-unity} we have
\[
\wh{s}(x_0)\in R, \text{ and }\h \wh{s}(x_0)-p^{2a}\theta(x_0)\in \pfr^{2ae+r(k_0+1)}\ocal=p^{2a+(b-2ad_0)}\ocal;
\]
and the claim follows. QED.

Claim 2 implies that $\pi_{p^{b-2a(d_0-1)}}(p^{2a}\theta(\f_0^{\times}))\subseteq \pi_{p^{b-2a(d_0-1)}}(R)$. On the other hand, as $\theta(\f_0^{\times})$ is a subgroup of the group $\ocal^{\times}$ of units of $\ocal$, we have that 
\[
\textstyle
\bbz_p[\theta(\f_0^{\times})]=\{\sum_{x\in \f_0^{\times}}c_x \theta(x)|\h c_x\in \bbz_p\};
\]
and so we get 
\be\label{eq-unramified-part}
\pi_{p^{b-2a(d_0-1)}}(p^{2a}\bbz_p[\theta(\f_0^{\times})])\subseteq \pi_{p^{b-2a(d_0-1)}}(R).
\ee  

{\bf Claim 3.} Suppose $a,b,r$ are given as in \eqref{eq-at-the-level-of-grading} for a constant $C$ which is more than $3[K:\bbq_p]$. And as before let $\wh{\xi}_1 \in R\cap \pfr^{ae+r} \ocal\setminus \pfr^{ae+r+1}\ocal$. Let $b':=b-2a([\f_0:\f_p]-1)$. Then 
\[
\pi_{p^{b'}}(R\cap p^{4ar'}\ocal)=
\pi_{p^{b'}}(p^{ar'}(\ocal_{0,u}+\ocal_{0,u}\wh{\xi}_1+\cdots+\ocal_{0,u}\wh{\xi}_1^{r'-1})\cap p^{4ar'}\ocal),
\]
where $\ocal_{0,u}:=\bbz_p[\theta(\f_0^{\times})]$ and $r':=e/r$.

{\em Proof of Claim 3.} By \eqref{eq-unramified-part} we have $\pi_{p^{b'}}(p^{2a}\ocal_{0,u})\subseteq \pi_{p^{b'}}(R)$. As $\pi_{p^{b'}}(R)$ is a ring and $\pi_{p^{b'}}(\wh{\xi}_1)\in \pi_{p^{b'}}(R)$, we have
\be\label{eq-one-side-inclusion-of-rings}
\pi_{p^{b'}}(p^{2a}(\ocal_{0,u}+\ocal_{0,u}\wh{\xi}_1+\cdots+\ocal_{0,u}\wh{\xi}_1^{r'-1}))\subseteq \pi_{p^{b'}}(R).
\ee
Let $\wt{\xi}_{i,j}:=p^{j+a(r'-i)}\wh{\xi}_1^i$ for integers $i$ in $[0,r')$ and $j$ in $[0,b'-ar')$; then $v_{\pfr}(\wt{\xi}_{i,j})=ej+ea(r'-i)+eai+ri=ear'+ej+ri$. Hence for any integer $k$ in $[ar'^2,b'r')$ there is $\wt{\xi}_k\in \bigcup_{j=0}^{r'-1}\bbz_p\wh{\xi}_1^j$ such that $v_{\pfr}(\wt{\xi}_k)=rk$.

Hence by Corollary~\ref{cor-grade-shift} and \eqref{eq-at-the-level-of-grading}  we have that 
\be\label{eq-again-shifting-grades}
x+\pfr^{l+1}\ocal\mapsto \wt{\xi}_{k}x+\pfr^{l+rk+1}\ocal 
\text{ is a bijection from }
{\rm gr}_{l,\pfr}(R;\ocal) \text{ to } {\rm gr}_{l+rk,\pfr}(R;\ocal)
\ee
if $k$ is an integer in $[ar'^2,b'r')$ and $l$ is an integer in $[ae,be-rk)$.

Let us fix $y_0\in R\cap p^{4ar'}\ocal$. Inductively we prove that for any integer $m$ in $[4ar'^2,b'r']$ we have 
\[
\pi_{\pfr^{rm}}(y_0)\in \pi_{\pfr^{rm}}(p^{ar'}(\ocal_{0,u}+\ocal_{0,u}\wh{\xi}_1+\cdots+\ocal_{0,u}\wh{\xi}_1^{r'-1})).
\]

The case of $m=4ar'^2$ is clear as $\pfr^{r(4ar'^2)}\ocal=p^{4ar'}\ocal$. 

By the induction hypothesis, we have that there is 
\be\label{eq-induction-hypothesis}
v_m\in p^{ar'}(\ocal_{0,u}+\ocal_{0,u}\wh{\xi}_1+\cdots+\ocal_{0,u}\wh{\xi}_1^{r'-1})
\text{ such that } y_0-v_m\in \pfr^{rm}\ocal.
\ee
By \eqref{eq-one-side-inclusion-of-rings},  there are $z_m\in R$ and $e_m\in \ocal$ such that
$
 v_m=z_m+p^{b'}e_m.  
$ 
 And so by \eqref{eq-induction-hypothesis}
 \be\label{eq-level-k-approximation}
y_0-z_m\in R\cap \pfr^{rm}\ocal.
 \ee
On the other hand applying \eqref{eq-again-shifting-grades} for the parameters $l=2ae$ and $k=m-2ar'\ge 4ar'^2-2ar'\ge ar'^2$ we have that
\[
y_0-z_m-\wt{\xi}_{m-2ar'}\wh{s}(x_m)\in \pfr^{rm+1}\ocal\cap R;
\]
and so by \eqref{eq-intersection-goes-to-a-smaller-ball}
\be\label{eq-next-digit}
y_0-z_m-\wt{\xi}_{m-2ar'}\wh{s}(x_m)\in \pfr^{r(m+1)}\ocal\cap R.
\ee 
On the other hand, by Claim 2 and \eqref{eq-induction-hypothesis},  
\begin{align}\label{eq-why-it-is-in-the-RHS}
\pi_{p^{b'}}(z_m+\wt{\xi}_{m-2ar'}\wh{s}(x_m)) & 
=\pi_{p^{b'}}(v_m)+\pi_{p^{b'}}(\wt{\xi}_{m-2ar'})\pi_{p^{b'}}(\wh{s}(x_m))\\
\notag
& =\pi_{p^{b'}}(v_m)+\pi_{p^{b'}}(\wt{\xi}_{m-2ar'})\pi_{p^{b'}}(p^{2a}\theta(x_m))\\
\notag
& =\pi_{p^{b'}}(v_m)+\pi_{p^{b'}}(\wt{\xi}_{m})\pi_{p^{b'}}(\theta(x_m))\\
\notag
& \in 
\pi_{p^{b'}}(p^{ar'}(\ocal_{0,u}+\ocal_{0,u}\wh{\xi}_1+\cdots+\ocal_{0,u}\wh{\xi}_1^{r'-1}));
\end{align}
in the last assertion we are using 
\begin{align*}
v_{\pfr}(\wt{\xi}_{m-2ar'})-v_{\pfr}(\wh{\xi}_1^{r'-1})&\ge (4ar'^2-2ar')r-(r'-1)(ae+r)
\\
&> e((3r'-1)a-1)\ge ear'=v_{\pfr}(p^{ar'}).
\end{align*}
By \eqref{eq-why-it-is-in-the-RHS} there is 
\be\label{eq-next-step}
v_{m+1}\in p^{ar'}(\ocal_{0,u}+\ocal_{0,u}\wh{\xi}_1+\cdots+\ocal_{0,u}\wh{\xi}_1^{r'-1})
\text{ such that } z_m+\wt{\xi}_{m-2ar'}\wh{s}(x_m)-v_{m+1}\in p^{b'}\ocal.
\ee
By \eqref{eq-next-step} and \eqref{eq-next-digit} we have that there is 
\be\label{eq-final-step}
v_{m+1}\in p^{ar'}(\ocal_{0,u}+\ocal_{0,u}\wh{\xi}_1+\cdots+\ocal_{0,u}\wh{\xi}_1^{r'-1})
\text{ such that } y_0-v_{m+1}\in \pfr^{r(m+1)}\ocal;
\ee
and by \eqref{eq-one-side-inclusion-of-rings} and \eqref{eq-final-step} the claim follows. QED.

{\bf Claim 4.} Suppose $a,b,r$ are as in \eqref{eq-at-the-level-of-grading} for a constant $C\ge 12[K:\bbq_p]$. Suppose $\wh{\xi}_1\in R\cap \pfr^{ae+r}\ocal\setminus \pfr^{ae+r+1}\ocal$. Then there is $\xi\in \ocal$ such that
\[
\xi-\wh{\xi}_1\in p^{\lfloor b/2\rfloor}\ocal, \text{ and }
\pi_{p^{\lfloor b/2\rfloor}}(R \cap p^{4ar'}\ocal)=\pi_{p^{\lfloor b/2\rfloor}} (p^{ar'}\ocal_{0,u}[\xi]\cap p^{4ar'}\ocal),
\]
where as before $\ocal_{0,u}=\bbz_p[\theta(\f_0^{\times})]$ and $r':=e/r$. Moreover $\bbq_p[\theta(\f_0^{\times})][\xi]$ is a totally ramified extension of $\bbq_p[\theta(\f_0^{\times})]$ and $[\bbq_p[\theta(\f_0^{\times})][\xi]:\bbq_p[\theta(\f_0^{\times})]]=r'$.

{\em Proof of Claim 4.} We know that $v_{\pfr}(\xi^{r'})=v_{\pfr}(p^{3ar'}\wh{\xi}_1^{r'})=4aer'+e>v_{\pfr}(p^{4ar'})$. And so by Claim 3, there are $c_0,\ldots,c_{r'-1}\in \ocal_{0,u}$ such that 
\be\label{eq-linear-combination-lower-degree-terms}
p^{3ar'} \wh{\xi}_1^{r'}+c_{r'-1}\wh{\xi}_1^{r'-1}+\cdots+c_1\wh{\xi}_1+c_0\in p^{b'}\ocal,
\ee
where $b':=b-2([\f_0:\f_p]-1)a$. Before we continue our analysis, we prove the following subclaim. 

{\bf Subclaim (a).}  Let $K_{0,u}:=\bbq_p[\theta(\f_0^{\times})]$. Then $K_{0,u}$ is an unramified extension of $\bbq_p$, its ring of integers is $\ocal_{0,u}:=\bbz_p[\theta(\f_0^{\times})]$, and its residue field is $\f_0$; in particular for any $c\in K_{0,u}^{\times}$ we have $v_{\pfr}(c)\in e\bbz$. 

{\em Proof of Subclaim (a).} Let $\ocal'_{0,u}$ be the ring of integers of $K_{0,u}$, and $\f_0'$ be its residue field. Then $\ocal_{0,u}:=\bbz_p[\theta(\f_0^{\times})]\subseteq \ocal'_{0,u}$. And so the $\f_0\subseteq \f_0'$. Let $\beta$ be a generator of $\f_0^{\times}$.  Since $\theta(\beta)$ is integral over $\bbz_p$, by Hensel's lemma we have $[\bbq_p[\theta(\beta)]:\bbq_p]=[\f_p[\beta]:\f_p]$; this means $[K_{0,u}:\bbq_p]=[\f_0:\f_p]$. But we know that the degree $[K_{0,u}:\bbq_p]$ of the field extension $K_{0,u}$ of $\bbq_p$ is equal to the product of its ramification index and its residue degree. So we conclude that $\f'_0=\f_0$ and $K_{0,u}$ is an unramified extension of $\bbq_p$. Since $\ocal_{0,u}$ is a complete subring of $\ocal'_{0,u}$ and $\pi_p(\ocal_{0,u})=\pi_p(\ocal'_{0,u})$, we deduce that $\ocal'_{0,u}=\ocal_{0,u}$. And the subclaim (a) follows. 

 Let $c_{r'}:=p^{3ar'}$ and $f(t):=\sum_{i=0}^{r'}c_it^i\in \ocal_{0,u}[t]$. Next we will find an upper bound for the $\pfr$-adic valuation $v_{\pfr}(f'(\wh{\xi}_1))$ of the value of the derivative $f'(t)$ of $f(t)$ at $t=\wh{\xi}_1$.
 
 {\bf Subclaim (b).} In the above setting $v_{\pfr}(f'(\wh{\xi}_1))\le 4er'a$.
 
 {\em Proof of Subclaim (b).} We have $f'(\wh{\xi}_1)=\sum_{i=1}^{r'} ic_i \wh{\xi}_1^{i-1}$. And by Subclaim (a) we have 
 \be\label{eq-p-adic-valution-terms-derivative}
 v_{\pfr}(ic_i \wh{\xi}_1^{i-1})\equiv r(i-1) \pmod e
 \ee
 for any integer $i$ in $[1,r']$. Since $0\le r(i-1)<e$ for $i\in [1,r']$, by \eqref{eq-p-adic-valution-terms-derivative} we have
 \be\label{eq-distinction-p-adic-valuation}
  v_{\pfr}(ic_i \wh{\xi}_1^{i-1})\neq  v_{\pfr}(jc_j \wh{\xi}_1^{j-1})
 \ee
 for distinct integers $i$ and $j$ in $[1,r']$. By \eqref{eq-distinction-p-adic-valuation} we deduce that
 \[
 v_{\pfr}(f'(\wh{\xi}_1))=\min_{1\le i\le r'} v_{\pfr}(ic_i \wh{\xi}_1^{i-1})\le v_{\pfr}(r'p^{3ar'}\wh{\xi}_1^{r'-1})\le 4er'a
 \]
 (to get the last assertion we are assuming $a\ge 3$); and the subclaim (b) follows.
 
 By \eqref{eq-linear-combination-lower-degree-terms} and Subclaim (b), we have that 
 \be\label{eq-comparison-f-derivative}
v_{\pfr}(f(\wh{\xi}_1))-2v_{\pfr}(f'(\wh{\xi}_1))\ge e(b'-4r'a)= e(b-2([\f_0:\f_p]-1+2r')a)\ge e(1-6[K:\bbq_p]/C)b\ge eb/2. 
 \ee  
 Hence by \cite[Chapter II, Section 2, Proposition 2]{Lang-AlgNum} and \eqref{eq-comparison-f-derivative}, there is $\xi\in \ocal$ such that 
 \be\label{eq-getting-zero-small-nbhd}
f(\xi)=0,\text{ and }  v_{\pfr}(\xi-\wh{\xi}_1)\ge v_{\pfr}(f(\wh{\xi}_1)/(f'(\wh{\xi}_1))^2)\ge eb/2.
 \ee
 Let $K_0:=K_{0,u}[\xi]$. Since $f(\xi)=0$ and $f(t)\in K_{0,u}[t]$, we have $[K_0:K_{0,u}]\le \deg f=r'$. By the second part of \eqref{eq-getting-zero-small-nbhd} we can deduce that $v_{\pfr}(\xi)=v_{\pfr}(\wh{\xi}_1)=ae+r$; and so $r\in v_{\pfr}(K_0^{\times})$. This and Subclaim (a) imply that the ramification index of the field extension $K_0$ over $K_{0,u}$ is at least $[r\bbz:e\bbz]=r'$. Since the index $[K_0:K_{0,u}]$ of the field extension $K_0$ over $K_{0,u}$ is equal to the product of the ramification index and the residue index of this field extension, by the above discussion we deduce that 
 \be\label{eq-totally-ramfied}
 [K_0:K_{0,u}]=r',\h K_0/K_{0,u} \text{ is a totally ramified extension, and } f(t) \text{ is irreducible in } K_{0,u}[t]. 
 \ee
 Since $\xi$ is integral over $\ocal_{0,u}$, $\ocal_{0,u}$ is integrally closed, and the degree of $\xi$ over $K_{0,u}$ is $r'$, we have that
 \be\label{eq-ring-generated}
 \ocal_{0,u}[\xi]=\ocal_{0,u}+\ocal_{0,u}\xi+\cdots+\ocal_{0,u}\xi^{r'-1}.
 \ee
 Hence we have 
 \begin{align*}
 \pi_{p^{\lfloor b/2\rfloor}}(R\cap p^{4ar'}\ocal)
&=\pi_{p^{\lfloor b/2\rfloor}}(p^{ar'}(\ocal_{0,u}+\ocal_{0,u}\wh{\xi}_1+\cdots+\ocal_{0,u}\wh{\xi}_1^{r'-1})\cap p^{4ar'}\ocal)
&\text{ (by Claim 3) } 
 \\
 &=\pi_{p^{\lfloor b/2\rfloor}}(p^{ar'}(\ocal_{0,u}+\ocal_{0,u}\xi+\cdots+\ocal_{0,u}\xi^{r'-1})\cap p^{4ar'}\ocal).
 & \text{ (by \eqref{eq-getting-zero-small-nbhd}) }
 \\
 &=\pi_{p^{\lfloor b/2\rfloor}}(p^{ar'}\ocal_{0,u}[\xi]\cap p^{4ar'}\ocal)
 & \text{ (by \eqref{eq-ring-generated}) }
 \end{align*}
 QED.
 
 {\bf Claim 5.} Suppose $a,b,r$ are as in \eqref{eq-at-the-level-of-grading} for a constant $C\ge 12[K:\bbq_p]$. Suppose $\xi$ is as in Claim 4, and let $K_{0,u}:=\bbq_p[\theta(\f_0^{\times})]$ and $K_0:=K_{0,u}[\xi]$. Let $\ocal_0$ be the ring of integers of $K_0$. Then 
 \[
 \pi_{p^{\lfloor b/2 \rfloor}}(R\cap p^{4ar'}\ocal)= \pi_{p^{\lfloor b/2 \rfloor}}(\ocal_0\cap  p^{4ar'}\ocal).
 \] 
 
 {\em Proof of Claim 5.} Suppose $x\in \ocal_0\cap p^{4ar'}\ocal$. Then there are $c'_i\in K_{0,u}$ such that $x=\sum_{i=0}^{r'-1}c'_i\xi^i$. By Subclaim (a) (of Claim 4), we have that $v_{\pfr}(K_{0,u}^{\times})=e\bbz$. And, by Claim 4, $v_{\pfr}(\xi)=ae+r$. Hence we have that 
 \be\label{eq-valuation-linear-combination-1}
 v_{\pfr}(c_i'\xi^{i})\equiv ri \pmod{e}.
 \ee 
 By \eqref{eq-valuation-linear-combination-1} and the fact that $0\le ri<e$ for $i\in [0,r'-1]$, we deduce that $v_{\pfr}(c_i'\xi^i)$'s are pairwise distinct; and therefore
 \be\label{eq-valuation-linear-combination-2}
 5er'a\le v_{\pfr}(x)=v_{\pfr}\Big(\sum_{i=0}^{r'-1}c_i'\xi^i\Big)=\min_{0\le i\le r'-1}\{v_{\pfr}(c_i')+i(ae+r)\}.
 \ee
 Hence for any integer $i$ in $[0,r'-1]$ we have
 \[ 
 v_{\pfr}(c_i')\in [4er'a-eia-ir,\infty)\cap e\bbz=[4er'a-eia,\infty)\cap e\bbz
 \]
 which implies $c_i'':=c_i'/p^{3ar'}\in \ocal_{0,u}$. Hence by \eqref{eq-ring-generated} we have  
 \be\label{eq-describing-ring-integers}
x=p^{3ar'}\sum_{i=0}^{r'-1}c_i''\xi^i\in p^{3ar'}\ocal_{0,u}[\xi].
 \ee
 Therefore by \eqref{eq-describing-ring-integers} and Claim 4, we have $\pi_{p^{\lfloor b/2 \rfloor}}(x)\in \pi_{p^{\lfloor b/2 \rfloor}}(p^{ar'}\ocal_{0,u}[\xi]\cap p^{4ar'}\ocal)=
 		\pi_{p^{\lfloor b/2 \rfloor}}(R\cap p^{4ar'}\ocal)$. And so
 \be\label{eq-one-side-inclusion}
 \pi_{p^{\lfloor b/2 \rfloor}}(\ocal_0\cap p^{4ar'}\ocal)\subseteq \pi_{p^{\lfloor b/2 \rfloor}}(R\cap p^{4ar'}\ocal).
 \ee
 On the other hand, by Claim 4 we have
 \be\label{eq-the-other-side}
 \pi_{p^{\lfloor b/2 \rfloor}}(R\cap p^{4ar'}\ocal)=\pi_{p^{\lfloor b/2 \rfloor}}(p^{ar'}\ocal_{0,u}[\xi]\cap p^{4ar'}\ocal)\subseteq \pi_{p^{\lfloor b/2 \rfloor}}(\ocal_0\cap p^{4ar'}\ocal);
 \ee
 and by \eqref{eq-one-side-inclusion} and \eqref{eq-the-other-side} claim follows. QED.
 
{\bf Claim 6.} Suppose $R$ is a subring of $\ocal$, and for positive integers $a$ and $b$ we have $a\le b$ and $\pi_{p^b}(R\cap p^a\ocal)=\pi_{b^b}(R\cap p^a\ocal_0)$ where $\ocal_0$ is the ring of integers of a subfield $K_0$ of $K$. Then $\pi_{p^{b-a}}(R)\subseteq \pi_{p^{b-a}}(\ocal_0)$.

{\em Proof of Claim 6.} For any $x\in R$, there is $x'\in \ocal_0$ such that 
$p^ax-x'\in p^b\ocal$. And so $x-p^{-a}x'\in p^{b-a}\ocal$, which implies $p^{-a}x'\in K_0\cap \ocal=\ocal_0$. Therefore $\pi_{p^{b-a}}(x)\in \pi_{p^{b-a}}(\ocal_0)$; and the claim follows. QED. 
 
Claim 5 and Claim 6 imply the assertion of Theorem~\ref{t:SubringOfO}. 
\end{proof}

\subsection{Subrings with the same graded structure as the ring $\ocal_0$ of integers of a subfield.} The main goal of this section is to show knowing ${\rm gr}_{p}(R;\ocal)={\rm gr}_{p}(\ocal_0;\ocal)$ and a bit more information is enough to deduce that $R=\ocal_0$. 
\begin{prop}\label{prop:detecting-ring-of-integers-via-grading-structure}
	Suppose $R$ is a subring of the ring $\ocal$ of integers of a finite extension $K$ of $\bbq_p$. Suppose $N$ is a positive integer which is at least $6[K:\bbq_p]$. Suppose $\ocal_0$ is the ring of integers of a closed subfield $K_0$ of $K$. If $\pi_{p^{6[K:\bbq_p]}}(R)=\pi_{p^{6[K:\bbq_p]}}(\ocal_0)$ and ${\rm gr}_{i,p}(R;\ocal)={\rm gr}_{i,p}(\ocal_0;\ocal)$ for any integer $i$ in $[0,N-1]$, then $\pi_{p^{N-4}}(R)=\pi_{p^{N-4}}(\ocal_0)$.
\end{prop}
\begin{proof}
{\bf Step 1.} (Getting the unramified part)  First we notice that \eqref{eq-at-the-level-of-grading} holds for $a=0$ and $b=N$. Let $\f_0$ and $\f$ be the residue fields of $\ocal_0$ and $\ocal$, respectively. Let $\pfr$ be a uniformizing element of $\ocal$. As  before, let $\theta:\f^{\times}\rightarrow \ocal^{\times}$ be the group embedding such that $\pi_{\pfr}(\theta(x))=x$ for any $x\in \f^{\times}$. So by Claim 2 in the proof of Theorem~\ref{t:SubringOfO}, we have that there is $\wh{s}:\f_0^{\times}\rightarrow R$ such that $\wh{s}(x)-\theta(x)\in p^N\ocal$ for any $x\in \f_0^{\times}$ (we extend the domain of $\wh{s}$ to $\f$ by setting $\wh{s}(0):=0$). This implies that
	$
	\pi_{p^N}(\theta(\f_0^{\times}))\subseteq \pi_{p^N}(R), 
	$
	and so 
	\be\label{eq:the-unramified-part-grading-recognition}
	\pi_{p^N}(\bbz_p[\theta(\f_0^{\times})])\subseteq \pi_{p^N}(R).
	\ee
	Notice that $K_{0,u}:=\bbq_p[\theta(\f_0^{\times})]$ is a subfield of $K_0$, $K_{0,u}$ is an unramified extension of $\bbq_p$, and $K_0$ is a totally ramified extension of $K_{0,u}$. And the ring $\ocal_{0,u}$ of integers of $K_{0,u}$ is $\bbz_p[\theta(\f_0^{\times})]$.

{\bf Step 2.} (Describing elements of $R$)	
	Suppose $\pfr_0$ be a uniformizing element of $\ocal_0$. Let $e_0:=v_{\pfr}(\pfr_0)$ and $d_0:=e/e_0$. Since $\pi_{p^{6[K:\bbq_p]}}(R)=\pi_{p^{6[K:\bbq_p]}}(\ocal_0)$, there is $\xi_0\in R$ such that 
	\be\label{eq-choice-of-xi}
	\xi_0-\pfr_0\in p^{6[K:\bbq_p]}\ocal.
	\ee
	And so we have 
	\begin{align}\label{eq-level-p}
	\notag
	\pi_p(\wh{s}(\f_0)+\wh{s}(\f_0)\xi_0+\cdots+\wh{s}(\f_0)\xi_0^{d_0-1})&
	=\f_0+\f_0\pi_p(\pfr_0)+\cdots+\f_0\pi_p(\pfr_0^{d_0-1})\\
	&
	=\pi_p(\theta(\f_0)+\theta(\f_0)\pfr_0+\cdots+\theta(\f_0)\pfr_0^{d_0-1})\\
	\notag
	&
	=\pi_p(\ocal_0)=\pi_p(R).
	\end{align}
	On the other hand, $|{\rm gr}_{i,p}(R;\ocal)|=|{\rm gr}_{i,p}(\ocal_0;\ocal)|=|\pi_p(\ocal_0)|$ for any integer $i$ in $[0,N]$. So by Corollary~\ref{cor-grade-shift} and \eqref{eq-level-p}, we deduce that
	\be\label{eq-representatives-for-levels}
	{\rm gr}_{i,p}(R;\ocal)=\pi_{p^{i+1}}(\wh{s}(\f_0)p^i+\wh{s}(\f_0)p^i\xi_0+\cdots+\wh{s}(\f_0)p^i\xi_0^{d_0-1}).
	\ee 
	Since $\wh{s}(\f_0)\subseteq R$, $\xi_0\in R$, and $p\in R$, by \eqref{eq-representatives-for-levels} we get that
	\be\label{eq-representing-elements-of-R}
	\pi_{p^{N}}(R)=\pi_{p^{N}}\left(\sum_{i=1}^{N}\sum_{j=0}^{d_0-1}\wh{s}(\f_0)p^i\xi_0^j\right).
	\ee

{\bf Step 3.} (Finding a zero of a degree $d_0$ polynomial close to $\xi_0$) Since $\xi_0^{d_0}$ is in $R$, by \eqref{eq-representing-elements-of-R} there is a monic polynomial $f_0(x):=\sum_{i=0}^{d_0}c_ix^i\in \ocal_{0,u}[x]$ of degree $d_0$ such that 
\be\label{eq-value-of-polynomial-very-small}
f_0(\xi_0)\in p^{N}\ocal.
\ee
	As $v_{\pfr}(\xi_0)=v_{\pfr}(\pfr_0)=e_0$ and $v_{\pfr}(\ocal_{0,u}\setminus \{0\})\subseteq e\bbz$, we get that for any integer $i$ in $[1,d_0]$
	\[
	v_{\pfr}(ic_i\xi_0^{i-1})\equiv (i-1)e_0 \pmod e;
	\]
	and so $v_{\pfr}(ic_i\xi_0^{i-1})$ are distinct integers. Hence 
	\be\label{eq-controlling-derivative}
	v_{\pfr}(f_0'(\xi_0))=\min_{1\le i\le d_0} v_{\pfr}(ic_i\xi_0^{i-1}) 
	\le v_{\pfr}(d_0\xi_0^{d_0-1})\le v_{\pfr}(d_0)+(d_0-1)e_0< 2e.	
	\ee	
	By \cite[Chapter II, Section 2, Proposition 2]{Lang-AlgNum}, \eqref{eq-value-of-polynomial-very-small}, and \eqref{eq-controlling-derivative}, there is a zero $\pfr_0'\in \ocal$ of $f_0(x)$ such that 
	\be\label{eq-close-to-a-zero}
	v_{\pfr}(\xi_0-\pfr_0')\ge v_{\pfr}(f_0(\xi_0))-2v_{\pfr}(f_0'(\xi_0))> (N-4)e.
	\ee

{\bf Step 4.} (Showing that $\pfr_0'$ is in $K_0$) By \eqref{eq-choice-of-xi}, \eqref{eq-value-of-polynomial-very-small}, and $f_0(x)\in \ocal_{0,u}[x]$, we deduce that 
	\be\label{eq-small-value-at-p0}
	v_{\pfr}(f_0(\pfr_0))\ge 6[K:\bbq_p].
	\ee	
	By \eqref{eq-choice-of-xi}, \eqref{eq-controlling-derivative}, and $f_0(x)\in \ocal_{0,u}[x]$, we deduce that 
	\be\label{eq-large-derivative-at-p0}
	v_{\pfr}(f_0'(\pfr_0))<2e.
	\ee
	Hence again by \cite[Chapter II, Section 2, Proposition 2]{Lang-AlgNum}, \eqref{eq-small-value-at-p0}, and \eqref{eq-large-derivative-at-p0}, there is a zero $\pfr_0''\in \ocal$ of $f_0(x)$ such that
		\be\label{eq-p0-close-to-a-zero}
	v_{\pfr}(\pfr_0-\pfr_0'')\ge v_{\pfr}(f_0(\pfr_0))-2v_{\pfr}(f_0'(\pfr_0))> (6[K:\bbq_p]-4)e.
		\ee
	By \eqref{eq-close-to-a-zero} and \eqref{eq-p0-close-to-a-zero}, we have 
	\be\label{eq-two-zeros-are-close}
	v_{\pfr}(\pfr_0'-\pfr_0'') \ge (6[k:\bbq_p]-4)e.
	\ee
	Hence, if $\pfr_0'\neq \pfr_0''$, we get 
	\[
	(6[k:\bbq_p]-4)e\le v_{\pfr}(\pfr_0'-\pfr_0'') \le v_{\pfr}(f_0'(\pfr_0))<2e,
	\]
	which is a contradiction. Hence $\pfr_0'=\pfr_0''$. Let $L$ be the splitting field of $f_0(x)$ over $K_{0,u}$, and let $|\cdot|$ be the unique extension of $|\cdot|_{\pfr}$ to $L$. Then for any root $\beta$ of $f_0(x)$ in $L$ that is not equal to $\pfr_0'$, by \eqref{eq-large-derivative-at-p0}, \eqref{eq-p0-close-to-a-zero}, and $\pfr_0'=\pfr_0''$ 	we have 
	\be\label{eq-away-from-other-zeros}
	|\beta-\pfr_0'|\ge |f'_0(\pfr_0')|> |\f|^{-2e}\ge |\f|^{-(6[K:\bbq_p]-4)e}>|\pfr_0-\pfr_0'|. 
	\ee
Hence by Krasner's lemma (for instance see \cite[Chapter II, Section 2, Proposition 3]{Lang-AlgNum}) we have
\[
\pfr_0'\in K_{0,u}[\pfr_0]=K_0, \text{ and } v_{\pfr}(\pfr_0-\pfr_0')>(6[K:\bbq_p]-4)e.
\]
{\bf Step 5.} (Arguing why $\ocal_0=\ocal_{0,u}[\pfr_0']$) By Step 4, we know that $v_{\pfr}(\pfr_0)=v_{\pfr}(\pfr_0')$ and $\pfr_0'\in K_0$. Hence $\pfr_0'$ is a uniformizing element of $K_0$. As $K_0$ is a purely ramified extension of $K_{0,u}$, we have that $\ocal_0=\ocal_{0,u}[\pfr_0']$. 

{\bf Step 6.} (Finishing proof) By \eqref{eq-representing-elements-of-R} (see Step 1), \eqref{eq-close-to-a-zero} (see Step 3), and Step 5, we have 
\[
\pi_{p^{N-4}}(R)=\pi_{p^{N-4}}\left(\sum_{i=1}^{N}\sum_{j=0}^{d_0-1}\wh{s}(\f_0)p^i\pfr_0'^j\right)=\pi_{p^{N-4}}(\ocal_0).
\]	
\end{proof}

\section{Scalar-Sum-Product phenomena.}~\label{s:LinearCombination}
In this section, using conditional (Shannon) entropy we study {\em scalar-sum-product properties} of ring of integers $\ocal$ of a finite extension $K$ of $\bbq_p$. 

\subsection{Scalar-Sum inequality for regular sets.} 
The main goal of this section is to prove Proposition~\ref{p:ScalarSumRegular}.
\begin{prop}[Scalar-Sum inequality for regular sets]\label{p:ScalarSumRegular}
Let $K$ be a finite extension of $\bbq_p$, $\ocal$ be its ring of integers, and $\f$ be its residue field. Let $\Omega\subseteq \ocal$ be such that $\pi_{\pfr}$ induces a bijection between $\Omega$ and $\f_{\pfr}^{\times}$.  

Let $A$ and $B$ be $(m_0,\ldots,m_{N-1})$-regular and $(l_0,\ldots,l_{N-1})$-regular\footnote{For the definition of a regular set, see Definition~\ref{d:RegularSubset}.}. subsets of $\pi_{\pfr^N}(\ocal)$, respectively. Then
\[
\max_{\omega\in \Omega} |A+\pi_{\pfr^N}(\omega)B|\ge \prod_{i=0}^{N-1}\max\left(1,\left(\frac{1}{m_i l_i}+\frac{1}{|\f|}\right)^{-1}\right).
\]
\end{prop}   

Let us fix a subset $\Omega\subseteq \ocal^{\times}$ such that $\pi_{\pfr}$ induces a bijection between $\Omega$ and $\f^{\times}$. As it was mentioned in Section~\ref{ss:outline}, for any element $X\in \ocal$, there are unique $D_{i,\Omega}(X)\in \Omega\cup \{0\}$ such that 
\[
X=D_{0,\Omega}(X)+\pfr D_{1,\Omega}(X) +\pfr^2 D_{2,\Omega}(X) +\cdots;  
\]
and we call $D_{i,\Omega}(X)$ the $i$-th $\pfr$-adic digit with respect to $\Omega$. We fix $\Omega$ at the beginning of each given proof and write $D_i(X)$ instead of $D_{i,\Omega}(X)$. Again as explained in Section~\ref{ss:outline}, we can and will talk about the $i$-th $\pfr$-adic digit of an element $X$ of $\pi_{\pfr^N}(\ocal)$ for any integer $i$ in $[0,N-1]$; and we have 
\[
X=D_{0,\Omega}(X)+\pfr D_{1,\Omega}(X) +\cdots  +\pfr^{N-1} D_{N-1,\Omega}(X)+\pfr^N\ocal.
\]

\begin{definition}\label{d:RegularSubset}
A subset $A$ of $\pi_{\pfr^N}(\ocal)$ is called an $(m_0,m_1, \ldots , m_{N-1})$-regular subset if for any $0\le n\le N-1$ and $\bar{x}:=\pfr^n\ocal+x$ we have that either $\bar{x}\not\in\pi_{\pfr^{n}}(A)$ or 
\[
|\pi_{\pfr^{n+1}}(A)\cap \pi_{\pfr^{n+1}}(\pfr^n\ocal+x)|=m_n.
\]
\end{definition}
The following Lemma gives us a good way of thinking about regular subsets. 
\begin{lem}\label{l:ConditionalMeasuresRegular}
Let $A$ be an $(m_0,\ldots, m_{N-1})$-regular subset of $\pi_{\pfr^N}(\ocal)$. Let $X$ be a random variable with respect to the probability counting measure on $A$. Then 
\begin{enumerate}
\item $\pi_{\pfr^k}(X)$ is a random variable with respect to the probability counting measure on $\pi_{\pfr^k}(A)$; and
\item for any $a\in A$ the conditional probability measure 
\[ P(D_k(X)| D_0(X)=D_0(a),\ldots,D_{k-1}(X)=D_{k-1}(a))\] is a probability counting measure on a set of size $m_k$.
\end{enumerate}
\end{lem}
\begin{proof}
Both of the above claims are easy consequences of the fact that $A$ is a regular set.
\end{proof}
As explained in Section~\ref{ss:outline} (Step 2), to prove Proposition~\ref{p:ScalarSumRegular}, we work with random variables $X$ and $Y$ that are distributed according to the probability counting measures on the sets $A$ and $B$, respectively. And we use basic properties of (Shannon) entropy and conditional entropy. Here we recall their definitions and basic properties.

\begin{definition}\label{d:Entropy}
Let $X$ be a random variable on a finite set $\mathcal{X}$.
\begin{enumerate}
	\item  The (Shannon) entropy $H(X)$ of $X$ is 
\[
H(X):=\sum_{x\in \mathcal{X}} -\bbp(X=x) \log \bbp(X=x),
\]
where $\bbp(X=x)$ is the probability of having $X=x$.
  \item Suppose $Y$ is another random variable on $\mathcal{X}$. Then the entropy of $X$ conditioned to $Y$ is 
\[
H(X|Y):=\sum_{y\in \mathcal{X}} \bbp(Y=y) H(X|Y=y)=-\sum_{y\in \mathcal{X}} \bbp(Y=y) \sum_{x\in \mathcal{X}} \bbp(X=x|Y=y) \log \bbp(X=x|Y=y),
\] 
where $X|Y=y$ is the random variable $X$ conditioned to the random variable $Y$ taking a certain value $y$, and $\bbp(X=x|Y=y)$ is the probability of having $X=x$ conditioned to $Y=y$. 
\end{enumerate}
\end{definition}
Here are some of the basic properties of entropy that will be used in this note. 
\begin{lem}\label{l:PropertiesEntropy}
Suppose $\mathcal{X}$ is a finite set, and $X$ and $Y$ are random variables with values in $\mathcal{X}$. Then
\begin{enumerate}
	\item $H(X,Y)=H(X)+H(Y|X)$.
	\item $H(X)\ge H(X|Y)$.
	\item $H(f(X)|X)=0$ where $f$ is a function; and so $H(Y|X,f(X))=H(Y|X)$.
  \item Let $H_2(X):=-\log \sum_{x\in \mathcal{X}}\bbp(X=x)^2$; this is called the R\'{e}nyi entropy. Let $H_0(X):=\log |X|$, where $|X|$ is the size of the support of $X$. Then
  \[
  H_2(X)\le H(X) \le H_0(X).
  \] 
  \item $H(X|f(Y))\ge H(X|Y)$ where $f$ is a function.
\end{enumerate}
\end{lem}
\begin{proof}
These are all well-known facts; for instance for parts (a)-(d) see \cite[Theorem 2.4.1, Theorem 2.5.1, Theorem 2.6.4, Lemma 2.10.1, Problem 2.1]{book-entropy}. Part (e) is a consequence of parts (b) and (c): 
\[
H(X|f(Y))\ge H(X|f(Y),Y)=H(X|Y).
\]
\end{proof}
\begin{lem}\label{le:entropy-m-digit}
Suppose $\ocal$ is the ring of integers of a finite extension $K$ of $\bbq_p$. Let $X$ and $Y$ be two random variables with values in $\pi_{\pfr^N}(\ocal)$. Suppose $\Omega$ is a subset of $\ocal$ such that $\pi_{\pfr}$ induces a bijection from $\Omega$ to the set $\f^{\times}$ of non-zero elements of the residue field $\f$. Let $\alpha$ be a random variable with respect to the probability counting measure on $\Omega$; and $\overline{\alpha}:=\pi_{\pfr^N}(\alpha)$. Then for any integer $m$ in $[0,N-1]$ we have
\[
H(D_m(X+\overline{\alpha} Y)|\alpha, D_0(X+\overline{\alpha} Y),\ldots, D_{m-1}(X+\overline{\alpha} Y)) 
\hspace{8cm}
\]
\[
\hspace{3cm}
\ge 
H(\pi_{\pfr}(D_m(X))+\pi_{\pfr}(\alpha)\pi_{\pfr}(D_m(Y))|\pi_{\pfr}(\alpha), D_0(X),\ldots,D_{m-1}(X),D_0(Y),\ldots,D_{m-1}(Y)),
\]    	
where $D_i(Z):=D_{i,\Omega}(Z)$ is the $i$-th $\pfr$-adic digit of $Z$ with respect to $\Omega\cup\{0\}$.
\end{lem}
\begin{proof}
For a fixed $\alpha$, let $\sigma_{m,\alpha}:((\Omega\cup \{0\})\times (\Omega\cup \{0\}))^{m}\rightarrow \Omega\cup\{0\}$ be the following {\em carry over} function: 
\be\label{eq:carry-over-function}
\text{for any $a_0,\ldots,a_{m-1},b_0,\ldots,b_{m-1}\in \Omega\cup\{0\}$, }
\sigma_{m,\alpha}(\{(a_i,b_i)\}_{i=0}^{m-1}):=D_m\big(\sum_{i=0}^{m-1}a_i\pfr^i +\alpha\sum_{i=0}^{m-1}b_i\pfr^i\big).
\ee
Let $X_{m-1}:=\sum_{i=0}^{m-1}D_i(X)\pfr^i$ and $Y_{m-1}:=\sum_{i=0}^{m-1}D_i(Y)\pfr^i$. Then

\begin{align*}
	X+\overline{\alpha} Y 
	&
	\equiv (X_{m-1}+D_m(X)\pfr^m)+\overline{\alpha}(Y_{m-1}+D_m(Y)\pfr^m)
	\equiv (X_{m-1}+\overline{\alpha} Y_{m-1})+ (D_m(X)+\overline{\alpha}D_m(Y))\pfr^m
	\\
	&\equiv \sum_{i=0}^{m-1}D_i(X_{m-1}+\overline{\alpha} Y_{m-1})\pfr^i
	+(\sigma_{m,\alpha}(\{(D_i(X),D_i(Y))\}_{i=0}^{m-1})+
	D_m(X)+\overline{\alpha}D_m(Y))\pfr^m 
	\pmod{\pfr^{m+1}}.
\end{align*}
Therefore $D_m(X+\overline{\alpha}Y)$ is uniquely determined by 
\[
\pi_{\pfr}(\sigma_{m,\alpha}(\{(D_i(X),D_i(Y))\}_{i=0}^{m-1})+
	D_m(X)+\overline{\alpha}D_m(Y));
\]
and vice versa. Hence we have 
\[
	H_m:=H(D_m(X+\overline{\alpha} Y)|\alpha, D_0(X+\overline{\alpha} Y),
	\ldots, D_{m-1}(X+\overline{\alpha} Y)) 
\hspace{6.5 cm}
\]
\[
	=H(\pi_{\pfr}(\sigma_{m,\alpha}(\{(D_i(X),D_i(Y))\}_{i=0}^{m-1})+
	D_m(X)+\overline{\alpha}D_m(Y))|\alpha, D_0(X+\overline{\alpha} Y),\ldots, D_{m-1}(X+\overline{\alpha} Y)) 
\]
So by Lemma~\ref{l:PropertiesEntropy}, part (e), and the fact that the first $m-1$ $\pfr$-adic digits of $X+\overline{\alpha}Y$ are determined by $\{(D_i(X),D_i(Y)\}_{i=0}^{m-1}$ and $\alpha$, we have 
\[
H_m \ge H(\pi_{\pfr}(\sigma_{m,\alpha}(\{(D_i(X),D_i(Y))\}_{i=0}^{m-1})+
	D_m(X)+\overline{\alpha}D_m(Y))|\alpha, \{(D_i(X),D_i(Y)\}_{i=0}^{m-1}). 
\]
And since for a given $\alpha$ and $\{(D_i(X),D_i(Y)\}_{i=0}^{m-1}$, $\pi_{\pfr}(\sigma_{m,\alpha}(\{(D_i(X),D_i(Y))\}_{i=0}^{m-1})+
	D_m(X)+\overline{\alpha}D_m(Y))$ is uniquely determined by $\pi_{\pfr}(D_m(X)+\overline{\alpha}D_m(Y))$ and vice versa, we have 
\[
H_m\ge H(\pi_{\pfr}(D_m(X)+\overline{\alpha}D_m(Y))|\alpha, \{(D_i(X),D_i(Y)\}_{i=0}^{m-1}).
\]
And since $\pi_{\pfr}$ induces a bijection between $\Omega\cup\{0\}$ and $\f$, the claim follows.
\end{proof}
\begin{cor}\label{cor:cond-entropy-lin-com}
	Suppose $\ocal$ is the ring of integers of a finite extension $K$ of $\bbq_p$. Let $X$ and $Y$ be two random variables with values in $\pi_{\pfr^N}(\ocal)$. Suppose $\Omega$ is a subset of $\ocal$ such that $\pi_{\pfr}$ induces a bijection from $\Omega$ to the set $\f^{\times}$ of non-zero elements of the residue field $\f$. Let $\alpha$ be a random variable with respect to the probability counting measure on $\Omega$; and $\overline{\alpha}:=\pi_{\pfr^N}(\alpha)$. Then
	\[
	H(X+\overline{\alpha}Y|\alpha)\ge \sum_{m=0}^{N-1} \overline{H}_m,
	\]
	where $\overline{H}_m:=H(\pi_{\pfr}(D_m(X))+\pi_{\pfr}(\alpha)\pi_{\pfr}(D_m(Y))|\pi_{\pfr}(\alpha), \{(D_i(X),D_i(Y)\}_{i=0}^{m-1})$.
\end{cor}
\begin{proof}
Since $X+\overline{\alpha}Y$ is uniquely determined by its first $N-1$ $\pfr$-adic and vice versa, we have 
\[
H(X+\overline{\alpha}Y|\alpha)=H(\{D_i(X+\overline{\alpha}Y)\}_{i=0}^{N-1}|\alpha).
\]	
And so by Lemma~\ref{l:PropertiesEntropy}, part (a), we have
\be\label{eq-chain-rule}
H(X+\overline{\alpha}Y|\alpha)=\sum_{m=0}^{N-1} H_m,
\ee
where $H_m:=H(D_m(X+\overline{\alpha}Y)|\alpha, \{D_i(X+\overline{\alpha}Y)\}_{i=0}^{m-1})$. Therefore by \eqref{eq-chain-rule} and Lemma~\ref{le:entropy-m-digit} we get
\[
H(X+\overline{\alpha}Y|\alpha)\ge H(\pi_{\pfr}(D_m(X))+\pi_{\pfr}(\alpha)\pi_{\pfr}(D_m(Y))|\pi_{\pfr}(\alpha),\{(D_i(X),D_i(Y)\}_{i=0}^{m-1});
\]
and the claim follows.
\end{proof}
Based on Corollary~\ref{cor:cond-entropy-lin-com}, we see the need of having a lower bound for the entropy of a linear combination of two random variables with values in the residue field $\f$. As it was pointed out in Section~\ref{ss:outline} (see Question~\ref{qu-entropy-lower-bound}), at this point we do not know the answer to this question for arbitrary random variables. The following proposition partially answers this question. 

\begin{prop}\label{p:cond-entropy-linear-comb}
	Let $\overline{A}$ and $\overline{B}$ be two non-empty subsets of a finite field $\f$. Suppose $\overline{X}$ and $\overline{Y}$ are two random variables with values in $\f$ with respect to the probability counting measures on $\overline{A}$ and $\overline{B}$, respectively. Suppose $\overline{\alpha}$ is a uniform random variable with values in $\f^{\ast}$. Then 
	\[
	H(\overline{X}+\overline{\alpha}\overline{Y}|\overline{\alpha})\ge 
	-\log \left(\frac{1}{|\overline{A}||\overline{B}|}+\frac{1}{|\f|}\right)\ge 
	\min\{H(\overline{X})+H(\overline{Y}),\log|\f|\}-\log 2. 
	\]
\end{prop}
To prove Proposition~\ref{p:cond-entropy-linear-comb}, we start with a Lemma that helps us control an average of the R\'{e}nyi entropies $H_2(\overline{X}+c\overline{Y})$ as $c$ varies in $\f^{\times}$.  
\begin{lem}\label{l:AverageL2norm}
Let $\overline{A}, \overline{B}$ be two non-empty subsets of a finite field $\f$. Let $\pcal_{\overline{A}}$ and $\pcal_{\overline{B}}$ be the probability counting measures on $\overline{A}$ and $\overline{B}$, respectively. For two functions $f,g:\f\rightarrow \bbc$, let $f\ast g(x):=\sum_{y\in \f}f(y)g(x-y)$ be the convolution of $f$ and $g$. Then 
\[
\frac{1}{|\f^{\times}|}\sum_{c\in \f^{\times}} \|\pcal_{\overline{A}}\ast c\pcal_{\overline{B}}\|_2^2\le \min\left(1,\frac{1}{|\overline{A}||\overline{B}|}+\frac{1}{|\f|}\right),
\] 
where $c\pcal_{\overline{B}}$ is the push-forward of $\pcal_{\overline{B}}$ under the multiplication by $c$.

\end{lem}
\begin{proof}
Let us recall that for any two subsets $U$ and $V$ of $\f$, the additive energy of $U$ and $V$ is 
\[
E(U,V)=|\{(x_1,y_1,x_2,y_2)\in U\times V\times U\times V|\h x_1+y_1=x_2+y_2\}|,
\]
 and we have
\[
E(X,Y)=\|\one_X\ast \one_Y\|_2^2.
\]
where $\one_X$ is the characteristic function of the set $X$. Hence we get
\begin{align*}
\frac{1}{|\f^{\times}|}\sum_{\alpha\in \f^{\times}} \|\pcal_{\abar} \ast \alpha\pcal_{\bbar}\|_2^2
& =\frac{1}{|\f^{\times}||\abar|^2|\bbar|^2}\sum_{\alpha\in \f^{\times}}\|\one_{\abar}\ast \one_{\alpha \bbar}\|_2^2
=\frac{1}{|\f^{\times}||\abar|^2|\bbar|^2}\sum_{\alpha\in \f^{\times}}E(\abar,\alpha \bbar)\\
&=\frac{1}{|\f^{\times}||\abar|^2|\bbar|^2}\sum_{\alpha\in \f^{\times}}|\{(a_1,b_1,a_2,b_2)\in \abar\times \bbar \times \abar \times \bbar|\h a_1+\alpha b_1=a_2+\alpha b_2\}| \\
&=\frac{1}{|\f^{\times}||\abar|^2|\bbar|^2}\sum_{\alpha\in \f^{\times}}|\{(a_1,b_1,a_2,b_2)\in \abar\times \bbar \times \abar \times \bbar|\h a_1=a_2,\h a_1+\alpha b_1=a_2+\alpha b_2\}|\\
&+\frac{1}{|\f^{\times}||\abar|^2|\bbar|^2}\sum_{\alpha\in \f^{\times}}
|\{(a_1,b_1,a_2,b_2)\in \abar\times \bbar \times \abar \times \bbar|\h a_1\neq a_2,\h a_1+\alpha b_1=a_2+\alpha b_2\}|\\
&=\frac{1}{|\abar||\bbar|}\\
&+\frac{1}{|\f^{\times}||\abar|^2|\bbar|^2}\sum_{\alpha\in \f^{\times}}
\{(a_1,b_1,a_2,b_2)\in \abar\times \bbar \times \abar \times \bbar|\h a_1\neq a_2, b_1\neq b_2, \alpha=\frac{a_1-a_2}{b_2-b_1}\}|\\
&=\frac{1}{|\abar||\bbar|}+\frac{|(\abar^2\setminus \Delta(\abar))\times (\bbar^2\setminus \Delta(\bbar))|}{|\f^{\times}||\abar|^2|\bbar|^2}, \h\text{where $\Delta(X):=\{(x,x)| x\in X\}$,}\\
&=\frac{1}{|\abar||\bbar|}+\frac{(|\abar|-1)(|\bbar|-1)}{|\f^{\times}||\abar||\bbar|} \le \min\left(1,\frac{1}{|\overline{A}||\overline{B}|}+\frac{1}{|\f|}\right).
\end{align*}
\end{proof}
\begin{proof}[Proof of Proposition~\ref{p:cond-entropy-linear-comb}]
{\bf Step 1.} (Entropy $\ge$ R\'{e}nyi entropy) Using Lemma~\ref{l:PropertiesEntropy}, part (d), we have
\be\label{eq-cond-to-renyi}
H(\overline{X}+\overline{\alpha}\overline{Y}|\overline{\alpha})=
\frac{1}{|\f^{\times}|}\sum_{c\in \f^{\times}} H(\overline{X}+c\overline{Y})\ge \frac{1}{|\f^{\times}|}\sum_{c\in \f^{\times}} H_2(\overline{X}+c\overline{Y})
=-\frac{1}{|\f^{\times}|}\sum_{c\in \f^{\times}}\log \|\pcal_{\overline{A}}\ast c\pcal_{\overline{B}}\|_2^2.
\ee

{\bf Step 2.} (Convexity of  $-\log$) By the convexity of $-\log$ function and Jensen's inequality, we have
\be\label{eq-convexity}
-\log \left(\frac{1}{|\f^{\times}|}\sum_{c\in \f^{\times}} \|\pcal_{\overline{A}}\ast c\pcal_{\overline{B}}\|_2^2\right)
\le -\frac{1}{|\f^{\times}|}\sum_{c\in \f^{\times}}\log \|\pcal_{\overline{A}}\ast c\pcal_{\overline{B}}\|_2^2.
\ee
{\bf Step 3.} (Finishing the proof) By Lemma~\ref{l:AverageL2norm}, \eqref{eq-cond-to-renyi}, and \eqref{eq-convexity}, we get that 
\[
H(\overline{X}+\overline{\alpha}\overline{Y}|\overline{\alpha})\ge 
-\log\left(\frac{1}{|\overline{A}||\overline{B}|}+\frac{1}{|\f|}\right).
\]
And since $H(\overline{X})=\log |\overline{A}|$ and $H(\overline{Y})=\log |\overline{B}|$, claim follows.
\end{proof}
\begin{proof}[Proof of Proposition~\ref{p:ScalarSumRegular}] For a random variable $Z$ with values in a set $\mathcal{Z}$ and a function $f:\mathcal{Z}\rightarrow \bbc$, let $\bbe_{Z}(f)$ be the expectation of the random variable $f(Z)$; for instance for a function $f:\Omega\rightarrow \bbc$, we have $\bbe_{\alpha}(f)=\frac{1}{|\f^{\times}|}\sum_{c\in\f^{\times}}f(c)$.

Let $X$ and $Y$ be uniform random variables on the sets $A$ and $B$, respectively. Suppose $\Omega$ is a subset of $\ocal$ such that $\pi_{\pfr}$ induces a bijection from $\Omega$ to $\f^{\times}$. 

{\bf Step 1.} (From cardinality to entropy) Using Lemma~\ref{l:PropertiesEntropy}, part (d), we have 
\be\label{eq:cardinality-entropy}
\max_{c\in \Omega} \log |A+cB| \ge \bbe_{\alpha}(\log |A+\alpha B|) \ge 
	\bbe_{\alpha}(H(X+\alpha Y))=H(X+\alpha Y|\alpha).
\ee 
{\bf Step 2.} (Entropy to relative entropies of digits) By Corollary~\ref{cor:cond-entropy-lin-com}, we have
\be\label{eq:relative-entropy-digits}
H(X+\overline{\alpha}Y|\alpha)\ge \sum_{m=0}^{N-1} \overline{H}_m,
\ee 
where $\overline{H}_m:=H(\pi_{\pfr}(D_m(X))+\pi_{\pfr}(\alpha)\pi_{\pfr}(D_m(Y))|\pi_{\pfr}(\alpha), \{(D_i(X),D_i(Y)\}_{i=0}^{m-1})$.

{\bf Step 3.} (Regularity and bound for relative entropies) By Lemma~\ref{l:ConditionalMeasuresRegular}, for any $a\in A$, $b\in B$, and integer $k$ in $[0,N-1]$, the conditional probability measures 
\[
\bbp(\pi_{\pfr}(D_k(X))|\{D_i(X)\}_{i=0}^{k-1}=\{D_i(a)\}_{i=0}^{k-1}) 
\text{ and }
\bbp(\pi_{\pfr}(D_k(Y))|\{D_i(Y)\}_{i=0}^{k-1}=\{D_i(b)\}_{i=0}^{k-1}) 
\]
are probability counting measures on sets of size $m_k$ and $l_k$, respectively. And so by Proposition~\ref{p:cond-entropy-linear-comb} we have 
\be\label{eq-at-k-th-level}
\overline{H}_k
\ge \max\left\{0,-\log\left(\frac{1}{m_kl_k}+\frac{1}{|\f|}\right)\right\}.
\ee
 {\bf Step 4.} (Finishing the proof) By \eqref{eq:cardinality-entropy}, \eqref{eq:relative-entropy-digits}, and \eqref{eq-at-k-th-level}, we get
\[
\max_{c\in \Omega} \log |A+cB|\ge \sum_{k=0}^{N-1} \max\left\{0,-\log\left(\frac{1}{m_kl_k}+\frac{1}{|\f|}\right)\right\}= 
\log\left(\prod_{k=0}^{N-1} \max\left\{1,\left(\frac{1}{m_kl_k}+\frac{1}{|\f|}\right)^{-1}\right\}\right);
\]
and the claim follows.
\end{proof}

\subsection{Scalar-Sum-Product expansion for regular sets.}\label{ss:ScalarSumExpansionRegular}
The following is the main result of this section.
\begin{prop}[Scalar-Sum-Product expansion for regular sets]\label{p:ScalarSumExpansionRegular}
For any positive number $\vare$, positive integer $d$, $0<\delta\ll \vare^4$, and any finite extension $K$ of $\bbq_p$ with large, depending on $\vare$, residue field $\f$ the following holds:

Let $\ocal$ be the ring of integers of $K$, and $\pfr$ be a uniformizing element of $K$. Let $\Omega\subseteq \ocal$, and suppose $\pi_{\pfr}$ induces a bijection between $\Omega\subseteq \ocal$ and $\f^{\times}$. Suppose $A\subseteq \pi_{\pfr^N}(\ocal)$ satisfies the following properties:
\begin{enumerate}
	\item $A$ is an $(m_0,\ldots,m_{N-1})$-regular subset; and $m_0,m_1$ are more than 1.
	\item $|A|\le |\f|^{N(1-\vare)}$.
	\item $|\pi_{\pfr^l}(A)|\ge |\f|^{l\vare}$ for any $N\delta\le l\le N$.
\end{enumerate}
Then
\[
\max_{\alpha\in \Omega, a\in A-A}|A+\pi_{\pfr^N}(\alpha)aA|\ge |A| |\f|^{N\delta}.
\]
\end{prop}

We prove Proposition~\ref{p:ScalarSumExpansionRegular} by contradiction. For the rest of this section, $A$ and $m_i$'s satisfy all the conditions of Proposition~\ref{p:ScalarSumExpansionRegular}. Moreover we assume to the contrary that $|A+\pi_{\pfr^N}(\alpha)aA|< |A| |\f|^{N\delta}$ (for a small enough $\delta$ to be determined later) for any $\alpha\in\Omega$ and $a\in A-A$, and let
\be\label{e:xis}
x_i:=\frac{\log m_i}{\log |\f|}
\ee
for any $0\le i\le N-1$. 

\begin{lem}\label{l:xisProperties}
Let $x_i$'s be as in (\ref{e:xis}). Then
\begin{align}
\label{e:size} x_0,x_1\neq 0 & \h \text{and }
0\le x_i\le 1 \text{ for any } 0\le i\le N-1, \\
\label{e:sum} \sum_{i=0}^{N-1} x_i &\le N(1-\vare),\\
\label{e:large}\sum_{i=0}^{l-1} x_i &\ge l\vare \text{, for any $N\delta \le l \le N$},\\
\label{e:ContraryAssumption}\sum_{i=0}^{N-1-k} \min(x_i,1-x_{i+k}) &\le N\left(\delta+\frac{\log 2}{\log |\f|}\right) \text{if  $x_k\neq 0$}.
\end{align}
\end{lem}
\begin{proof}
Since $1\le m_i\le |\f|$, $|A|=\prod_{i=0}^{N-1}m_i\le |\f|^{N(1-\vare)}$, and $|\pi_{\pfr^l}(A)|=\prod_{i=0}^{l-1} m_i\ge |\f|^{l\vare}$ for $N\delta\le l\le N$, one can see that (\ref{e:size}), (\ref{e:sum}), and (\ref{e:large}) hold.

Suppose $x_k\neq 0$. So there is $a\in A-A$ such that $a\in \pi_{\pfr^N}(\pfr^k\ocal)\setminus \pi_{\pfr^N}(\pfr^{k+1}\ocal)$. Hence $aA$ is an $(1,\ldots,1,m_0,\ldots,m_{N-1-k})$-regular subset of $\pi_{\pfr^N}(\ocal)$. Let $m_{-i}=1$ for any $i\in \bbz^+$. Therefore by Proposition~\ref{p:ScalarSumRegular} we have
\begin{align*}
\max_{\alpha\in\Omega} |A+\pi_{\pfr^N}(\alpha)aA|&\ge \prod_{i=0}^{N-1}\max\left(1,\left(\frac{1}{m_i m_{i-k}}+\frac{1}{|\f|}
\right)^{-1} \right) \\
& \ge (\prod_{i=0}^{N-1}m_i) \prod_{i=0}^{N-1}\left(\frac{1}{m_{i-k}}+\frac{m_i}{|\f|}
\right)^{-1}\\
&\ge |A| \prod_{i=0}^{N-1}\frac{\min\left(m_{i-k},|\f|/m_i \right)}{2}.
\end{align*}
Thus by the contrary assumption we have
\[
N\delta \log |\f|\ge -N \log 2+\sum_{i=0}^{N-1} \min(\log m_{i-k}, \log |\f|-\log m_i).
\]
And so
\[
N\left(\delta+\frac{\log 2}{\log |\f|}\right)\ge \sum_{i=0}^{N-1-k} \min(x_{i}, 1-x_{i+k}).
\]
\end{proof}
Now we follow Lindenstrauss-Varj\'{u}'s treatment~\cite{LV} to prove that, if $|\f|\gg_{\vare}1$ and $0<\delta\ll \vare^4$, then there are no real numbers $x_0,\ldots, x_{N-1}$ that satisfy properties mentioned in Lemma~\ref{l:xisProperties}. This is based on Mann's theorem on Schnirelmann density of subsets of non-negative integers.
\begin{definition}\label{d:Density}
The Schnirelmann density $\sigma(X)$ of a non-empty subset $X$ of non-negative integers is
\[
\sigma(X):=\inf_{n\in \bbz^+}\frac{|X\cap [1,n]|}{n}.
\]
\end{definition}
\begin{thm}[Mann's Theorem]\label{t:Mann}
Let $X,Y$ be two non-empty subsets of non-negative integers. Suppose $X$ and $Y$ contain $0$. Then either $X+Y=\bbz^{\ge 0}$ or $\sigma(X+Y)\ge \sigma(X)+\sigma(Y)$.
\end{thm}
Before we proceed with the proof of Proposition~\ref{p:ScalarSumExpansionRegular}, let us recall the definition of
the $i$-th grade ${\rm gr}_{i,\pfr}(A)$ of $A$ with respect to powers of $\pfr$. For any non-negative integer $i$, we let ${\rm gr}_{i,\pfr}(A):=\pi_{\pfr^{i+1}}(\wt{A}\cap \pfr^i\ocal)\subseteq \pfr^i\ocal/\pfr^{i+1}\ocal$ where $\wt{A}:=\pfr^{-1}_{\pfr^N}(A)$. Then for $\{x_k\}_{k=0}^{N-1}$ as in Equation (\ref{e:xis}), we have ${\rm gr}_{k,\pfr}(A-A)\neq 0$ if and only if $x_k\neq 0$.
Let 
\be\label{e:NonZero}
 J:=\{k\in [0,N)|\h x_{k}\neq 0\}.
\ee
\begin{lem}[Lindenstrauss-Varj\'{u}~\cite{LV}]~\label{l:LargeInterval}
Let $x_i$'s be real numbers that satisfy conditions (\ref{e:size}) and (\ref{e:large}) of Lemma~\ref{l:xisProperties}. Then we have
\[\textstyle
(N \lceil 1/\vare\rceil \delta,N)\cap \bbz\subseteq \sum_{3\lceil1/\vare\rceil} J.
\]
\end{lem}
\begin{proof}
By (\ref{e:size}) and (\ref{e:large}), we have that
\be\label{eq-large-intersection-initial}
|(J+1)\cap [1,l]|\ge \sum_{i=0}^{l-1} x_{i} \ge l\vare
\ee
for any $l\in [N\delta, N)$. Let $k_0$ be the largest integer such that $|(J+1)\cap [1,k_0]|< k_0\vare$. Hence, by \eqref{eq-large-intersection-initial}, we have $k_0<\delta N$. For any $k_0<l\le N$ we deduce that 
\be\label{eq-large-intersection}
|(J+1)\cap [k_0+1,l]|=|(J+1)\cap [1,l]|-|(J+1)\cap [1,k_0]|\ge (l-k_0)\vare;
\ee
in particular applying \eqref{eq-large-intersection} for $l=k_0+1$ we get that $k_0\in J$. Next, we let 
\[
X:=\{j-k_0+1|\h j\in J,\h j\ge k_0-1\}\cup \{k\in \bbz|\h k\ge N-k_0\}.
\]
 By \eqref{eq-large-intersection} we have $|(J-k_0+1)\cap [1,l-k_0]|\ge (l-k_0)\vare$ for any $k_0<l\le N$. Hence we have that the Schnirelmann density $\sigma(X)$ of $X$ is at least $d\vare$. Therefore by Mann's theorem (Theorem~\ref{t:Mann}) we have
\[\textstyle
\bbz^{\ge 0}=\sum_{\lceil 1/\vare\rceil} (X\cup \{0\}).
\]
So for any integer $\lceil 1/\vare\rceil \delta N< m < N$ there are $t\le \lceil 1/\vare\rceil$ elements of $X$ that add up to $m-\lceil 1/\vare\rceil k_0$. Since $m-\lceil 1/\vare\rceil k_0<N-k_0$, there are $j_1,\ldots,j_t\in J\cap [k_0,\infty)$ such that
\[
(j_1-k_0+1)+\cdots+(j_t-k_0+1)=m-\lceil 1/\vare\rceil k_0.
\]
Thus we have
\[\textstyle
m=j_1+\cdots+j_t+(\lceil 1/\vare\rceil-t)k_0+t\in \sum_{3\lceil 1/\vare \rceil} J,
\]
as $k_0, 0, 1\in J$.
\end{proof}
For $\{x_k\}_{k=0}^{N-1}$ as in (\ref{e:xis}), let
\be\label{e:Large}
L:=\{k\in [0,N)|\h\ x_{k}\ge 1/2\}.
\ee
In particular, $L\subseteq J$. So $L$ consists of indexes $i$, where the $i$-th grade ${\rm gr}_{i,\pfr}(A-A)$ of $A-A$ is {\em large}. Now the idea is that the contrary assumption implies when we shift $L$ by an element of $J$ we cannot get {\em lots of new} indexes; the almost invariance of $L$ under the shifts by elements of $J$ leads us to a contradiction. Let $D_L(k)$ be the number of new elements that are gained after a $k$-shift of $L$; that means 
\be\label{eq-new-indexes-after-shift}
D_L(k):=|(L+k\cap [0,N))\setminus L|.
\ee 
It is useful to notice that $D_L(k)=|(L\cap [0,N-k))\setminus (L-k)|$. Here is the main property of the sets $J$, $L$, and the function $D_L(k)$.
\begin{lem}\label{l:BT}
Suppose $k\in J$. Then
\[
D_L(k)\le 2N\left(\delta+\frac{\log 2}{\log |\f|}\right),
\]
where $D_L(k)$ is as in \eqref{eq-new-indexes-after-shift}.
\end{lem}
\begin{proof}
Suppose $i\in (L\cap [0,N-k))\setminus (L-k)$. Then
\be\label{e:min}
\min(x_{i},1-x_{i+k})\ge 1/2.
\ee
On the other hand, since $k\in J$, $x_{k}\neq 0$. Therefore by (\ref{e:ContraryAssumption}) we have
\begin{align*}
\frac{D_L(k)}{2}=&
\frac{|(L\cap [0,N-k))\setminus (L-k)|}{2}
\le \sum_{i\in (L\cap [0,N-k))\setminus (L-k)} \min (x_{i},1-x_{i+k})\\
& \le \sum_{i=0}^{N-k-1}\min(x_{i},1-x_{i+k})
 \le N\left(\delta+\frac{\log 2}{\log |\f|}\right).
\end{align*}
\end{proof}
\begin{lem}\label{l:Subadditive}\cite{LV}
For any pair of positive integers $k_1,k_2<N$ we have $D_L(k_1+k_2)\le D_L(k_1)+D_L(k_2)$.
\end{lem}
\begin{proof}
For any three sets $A,B,$ and $C$ we have $A\setminus C\subseteq (A\setminus B)\cup(B\setminus C)$. Therefore
\begin{align*}
D_L(k_1+k_2)&=|(L\cap [0,N-1-k_1-k_2])\setminus (L-k_1-k_2)|\\
& \le |(L\cap [0,N-1-k_1-k_2])\setminus  ((L-k_1)\cap [0,N-1-k_1-k_2])|\\
&+|((L-k_1)\cap [0,N-1-k_1-k_2])\setminus (L-k_1-k_2)|\\
&\le D_L(k_1)+D_L(k_2).
\end{align*}
\end{proof}
\begin{lem}\label{l:AverageD}\cite{LV}
For some universal implied constants we have
\[
\frac{1}{N}\sum_{k=0}^{N-1} D_L(k) \ge  N \vare^3/16
\]
if $0<\vare\ll 1$, $\delta\ll \vare^2$ and $1\ll_{\vare} |\f|$.
\end{lem}
\begin{proof}
By (\ref{e:ContraryAssumption}) we have
\be\label{e:Complement}
N\left(\delta+\frac{\log 2}{\log |\f|}\right)\ge \sum_{i=0}^{N-1}\min(x_i,1-x_i)\ge \sum_{i\in [0,N)\setminus L} x_{i}.
\ee

By (\ref{e:large}), for any integer $l\in [N\delta,N]$, we have
\[
l\vare\le \sum_{i=0}^{l-1} x_i\le \sum_{i\in [0,N)\setminus L} x_{i}+ |L\cap [0,l)|.
\]
Hence by (\ref{e:Complement}) we have
\[
|L\cap[0,l)|\ge l\left(\vare-\frac{N}{l}\left(\delta+\frac{\log 2}{\log |\f|}\right)\right).
\]
Suppose $\max(\delta,\frac{\log 2}{\log |\f|})<\vare^2/16$. Then we have
\be\label{e:TLowerBound}
|L\cap [0,N\vare/4)| \ge N \vare^2/8.
\ee
For any $i\in (N\vare/4,N)$ and any $k\in L\cap [0,N\vare/4)$ we have
$
i\in (L\cap [0,N\vare/4))+(i-k).
$
Hence for any integer $i\in (N\vare/4,N)$ we have
\[
\sum_{0\le j < \frac{N(1-\vare)}{4}}\one_{(L\cap [0,\frac{N\vare}{4})+j)\setminus L}(i)\ge |L\cap [0,\textstyle\frac{N\vare}{4})| \one_{[0,N)\setminus L}(i),
\]
where $\one_Y$ is the characteristic function of a set $Y$.
By adding over $i$ in the above range we get
\be\label{e:LowerBoundSum}
\sum_{j=0}^{N(1-\vare/4)} D_L(j) \ge |L\cap [0,N\vare/4)|\cdot | (N\vare/4,N)\setminus L|.
\ee
By (\ref{e:ContraryAssumption}) we have
\[
N\left(\delta+\frac{\log 2}{\log |\f|}\right)\ge \sum_{i=0}^{N-1}\min(x_i,1-x_i)\ge \sum_{i\in L}(1-x_i)\ge |L|-\sum_{i=0}^{N-1} x_i\ge |L|-(1-\vare)N
\]
Therefore, by our assumption $\max(\delta,\frac{\log 2}{\log |\f|})<\vare^2/16$, we have
\be\label{e:UpperBoundT}
|L|\le N(1-\vare+\vare^2/8).
\ee
Hence by (\ref{e:TLowerBound}), (\ref{e:LowerBoundSum}), and (\ref{e:UpperBoundT}) we have
\[
\sum_{j=0}^{N-1} D_L(j)\ge (N\vare^2/8) \cdot N\left((1-\vare/4)-(1-\vare+\vare^2/8)\right)\ge   N^2 \vare^3/16.
\]
\end{proof}
\begin{cor}\label{c:LargeD}
For some integer $j_0\in [N\vare^3/32, N)$, we have $D_L(j_0)\ge N \vare^3/32$ if $0<\vare\ll 1$, $0<\delta\ll \vare^2$ and $1\ll_{\vare} |\f|$.
\end{cor}
\begin{proof}
By Lemma~\ref{l:AverageD} we have
\[
\sum_{j\in [N\vare^3/32, N-1]} D_L(j) \ge \sum_{i=0}^{N-1} D_L(j)- (N \vare^3/32)(N)\ge N^2\vare^3/16-N^2\vare^3/32=N^2\vare^3/32.
\]
And so for some $j_0\in [N\vare^3/32, N)$ we have $D_L(j_0)\ge N\vare^3/32.$
\end{proof}
\begin{proof}[Proof of Proposition~\ref{p:ScalarSumExpansionRegular}]
Suppose $\max(\delta,\frac{\log 2}{\log|\f|})<\vare^4/512$, and for some $A$ the assertion of Proposition~\ref{p:ScalarSumExpansionRegular} does not hold. Then we consider $J$ and $L$ as above. Hence by Corollary~\ref{c:LargeD} we have
\[
D_L(j_0)\ge  N \vare^3/32
\]
for some integer $j_0\in [N\vare^3/32, N)$. On the other hand, by Lemma~\ref{l:LargeInterval}, since $j_0\ge N\vare^3/32>N \lceil 1/\vare \rceil \delta$, for some integer $t$ in $[1,3\lceil 1/\vare\rceil]$ there are $t$ many elements $b_1,\ldots,b_t$ of $B$ such that
\[
j_0=b_1+\cdots+b_t.
\]
Hence by Lemma~\ref{l:Subadditive},
\be\label{e:DLowerBound}
D_L(b_i)> N\vare^4/100
\ee
 for some $i$.

On the other hand, by Lemma~\ref{l:BT}, we have that for any $b\in J$
\[
D_L(b)\le 2N\left(\delta+\frac{\log 2}{\log |\f|}\right)< N\vare^4/128,
\]
which contradicts (\ref{e:DLowerBound}).
\end{proof}

\subsection{Proof of Theorem~\ref{t:ScalarSumExpansion}: Scalar-Sum-Product expansion.}
 As in~\cite{BG} (also see~\cite{Bor} or \cite[Section 2.3]{SG1}), we start by a regularization process. The $\pfr$-adic filtration $\{\pi_{\pfr^N}(\pfr^{i}\ocal)\}_{i=1}^N$ induces an $|\f|$-regular rooted tree structure (with $N$-levels) on $\pi_{\pfr^N}(\ocal)$. So by a similar argument as the above mentioned articles we get the following large regular subset of $A$.
\begin{lem}\label{l:RegularSubset}
Let $0<\delta< \vare< 1$  and $|\f|\gg_{\vare,\delta} 1$. Then for $0<\delta'\le\vare\delta/4$ the following holds: 
Let $A\subseteq \pi_{\pfr^N}(\ocal)$. Suppose that $A$ satisfies the following properties:
\begin{enumerate}
\item $|\pi_{\pfr^i}(A)|\ge |\f|^{i\vare}$ for any $N\delta' \le i\le N$,
\item $|A+A|\le |A| |\f|^{N\delta'}$.
\end{enumerate}
Then there is $A'\subseteq A$ such that 
\begin{enumerate}
\item $A'$ is $(m_0,\ldots,m_{N-1})$-regular.
\item $|A'|\ge |A|/(2\log |\f|)^N$.
\item $|\pi_{\pfr^i}(A')|\ge |\f|^{i\vare/2}$ for $N\delta\le i\le N$.  
\end{enumerate}   
\end{lem} 
\begin{proof}
By \cite[Section 2.3]{SG1}, there is a subset $A'\subseteq A$ such that $|A'|\ge |A|/(2\log |\f|)^N$ and $A'$ is an $(m_0,\ldots,m_{N-1})$-regular set. Let $\bar{n}:=\max\{i\in [0,N-1]|\h |\pi_{\pfr^i}(A')|<|\f|^{i\vare/2}\}$. To show that $A'$ satisfies the above three conditions, it is enough to show that, if $0<\delta'\le \vare \delta/4$ and $|\f|\gg_{\vare,\delta} 1$, then $\bar{n}< N\delta$. 

Suppose to the contrary that $\bar{n}\ge N\delta >N\delta'$. Then by the assumption $|\pi_{\pfr^{\bar{n}}}(A')|\ge |\f|^{\bar{n}\vare}$. On the other hand, there is a subset $A''\subseteq A'$ such that $|\pi_{\pfr^{\bar{n}}}(A'')|=1$ and 
\[
|A''|=\frac{|A'|}{|\pi_{\pfr^{\bar{n}}}(A')|}> \frac{|A|}{(2\log |\f|)^N |\f|^{\bar{n}\vare/2}}.
\]
Therefore we have
\[
|A||\f|^{N\delta'}\ge |A+A|\ge |A''||\pi_{\pfr^{\bar{n}}}(A)|\ge \frac{|A| |\f|^{\bar{n}\vare}}{(2\log |\f|)^N |\f|^{\bar{n}\vare/2}},
\]
which implies that 
\be\label{e:NeededInequality}
(2\log |\f|)^N\ge |\f|^{\bar{n}\vare/2-N\delta'}\ge |\f|^{N(\delta\vare/2-\delta')}\ge |\f|^{N(\delta\vare/4)}.
\ee
For $|\f|\gg_{\vare,\delta} 1$ (so that $2 \log |\f|< |\f|^{\delta \vare/8}$), (\ref{e:NeededInequality}) implies that $\vare/8\ge \vare/4$, which is a contradiction. 
\end{proof}
\begin{proof}[Proof of Theorem~\ref{t:ScalarSumExpansion}]
Let $\delta_r(\vare/2)$ ($r$ stands for regular) be such that $0<\delta_r(\vare/2)\ll (\vare/2)^4$ where the implied constant is given by Proposition~\ref{p:ScalarSumExpansionRegular}. Suppose $|\f|\gg_{\vare} 1$, where the implied constant is given by Lemma~\ref{l:RegularSubset} for $\vare/2$ and $\delta_r(\vare/2)$. Now let $\delta'\ll(\vare/2)\delta_r(\vare/2)$ be given by Lemma~\ref{l:RegularSubset}.\footnote{To avoid further confusion with the $\delta$ used in Lemma~\ref{l:RegularSubset}, we are using $\delta'$, here. This is, in fact, supposed to be the claimed $\delta$ in Theorem~\ref{t:ScalarSumExpansion}.} We claim $\delta'$ satisfies the desired conditions.   

    By the choice of $\delta'$ and Lemma~\ref{l:RegularSubset}, there is $A'\subseteq A$ such that
    \begin{enumerate}
    \item $|A'|$ is an $(m_0,\ldots,m_{N-1})$-regular subset.
    \item $|\pi_{\pfr^i}(A')|\ge |\f|^{i\vare/2}$ for $N\delta_r(\vare/2) \le i\le N$,
    \item $|A'|\ge |A|/(2\log |\f|)^N$.
    \end{enumerate}
Next we modify $A'$ a bit, if necessary, to make sure that $m_0$ and $m_1$ are at least $2$.

If $m_0=1$ and $m_1>1$, then $A'+\{a_{01},a_{02}\}$ is a $(2,m_1,\ldots,m_{N-1})$-regular subset of $A+A$.  

If $m_0=m_1=1$, then $A'+\{a_{11},a_{12}\}+\{a_{01},a_{12}\}$ is a $(2,2,m_2,\ldots,m_{N-1})$-regular subset of $A+A+A$.

If $m_0>1$ and $m_1=1$, then 
\begin{enumerate}
\item there is a subset $X_0$ of $A$ such that $|X_0|=|\pi_{\pfr}(X_0)|=|\pi_{\pfr}(A)|$,
\item there is a $(1,1,m_2,\ldots,m_{N-1})$-regular subset $A'_0$ of $A'$
\end{enumerate}
Then $A'_0+\{a_{11},a_{12}\}+X_0$ is a regular $(|\pi_{\pfr}(A)|,2,m_2,\ldots,m_{N-1})$-regular subset of $A+A+A$.

So in all the cases we get an $(m_0,\ldots,m_{N-1})$-regular subset $A'$ of $A+A+A$ such that 
\begin{enumerate}
\item $m_0,m_1>1$.
\item $|\pi_{\pfr^i}(A')|\ge |\f|^{i\vare/2}$ for $N\delta_r(\vare/2)\le i\le N$.
\item $|A'|\ge |A|/(2\log |\f|)^N$.
\end{enumerate}
If $|\langle A\rangle_6|\ge |A||\f|^{N\delta}$ (for small enough $\delta$ to be determined later), we are done. So suppose this does not hold. In particular, $|A+A+A|\le |A||\f|^{N\delta}$. Hence $|A+A+A|\le |\f|^{N(1-\vare+\delta)}$. So assuming $\delta<\vare/2$, we have that $|A'|\le |\f|^{N(1-\vare/2)}$. Hence $A'$ satisfies all the conditions of Proposition~\ref{p:ScalarSumExpansionRegular}. Therefore we have
\be\label{e:RegularSubsetExpansion}
\max_{\omega\in \Omega, x\in A'-A'}|A'+\pi_{\pfr^N}(\omega)x A'|\ge |A'||\f|^{N\delta_r(\vare/2)}.
\ee
Since at least one of $a_{01}, a_{02}$ is a unit, we have that 
\[
|\langle A\rangle_6+\pi_{\pfr^N}(\omega)\langle A\rangle_6|\ge |A'+\pi_{\pfr^N}(\omega)x A'|.
\]
Therefore we have
\[
\max_{\omega\in \Omega} |\langle A\rangle_6+\pi_{\pfr^N}(\omega) \langle A\rangle_6|\ge |A'| |\f|^{N\delta_r(\vare/2)}\ge |A|\left(\frac{|\f|^{\delta_r(\vare/2)}}{2\log |\f|}\right)^N.
\]
Suppose $|\f|\gg_{\vare} 1$ so that $|\f|^{\delta_r(\vare/2)/2}\ge 2\log |\f|$. Hence we get
\[
\max_{\omega\in \Omega} |\langle A\rangle_6+\pi_{\pfr^N}(\omega) \langle A\rangle_6|\ge |A||\f|^{N\delta_r(\vare/2)/2}\ge |A||\f|^{N\delta'}.
\]
\end{proof}

\subsection{Proof of Theorem~\ref{t:BoundedGenerationScalarSum}: a scalar-sum-product set contains a large congruence set.}
\begin{prop}\label{p:BoundedGenerationScalarSum}
For any $0<\vare_1\ll \vare_2\ll 1$, $0<\delta\ll \vare_1^5$, and positive integer $1\ll_{\vare_1} C$, and any finite extension $K$ of $\bbq_p$ with large, depending on $\vare_1$, residue field $\f$ the following holds:

Let $\ocal$ be the ring of integers of $K$, and $\pfr$ be a uniformizing element of $K$. Let $\Omega\subseteq \ocal$, and suppose $\pi_{\pfr}$ induces a bijection between $\Omega\subseteq \ocal$ and $\f^{\times}$. Suppose $A\subseteq \pi_{\pfr^N}(\ocal)$ such that
\begin{enumerate}
\item $|\pi_{\pfr^i}(A)|\ge |\f|^{i\vare_1}$ for any $N\delta\le i\le N$.
\item there are $a_{01}, a_{02}, a_{11}, a_{12}\in A$ such that $a_{i1}-a_{i2}\in \pi_{\pfr^{N}}(\pfr^i\ocal\setminus \pfr^{i+1}\ocal)$.
\end{enumerate}
Then
\[
\pi_{\pfr^N}(\pfr^{\lceil\vare_2 N\rceil}\ocal)\subseteq \gen{A}{C}+\pi_{\pfr^N}(\omega_1)\gen{A}{C}+\cdots+\pi_{\pfr^N}(\omega_C)\gen{A}{C},
\]
for some $\omega_i\in \prod_C(\Omega\cup\{1\})$.
\end{prop}
\begin{proof}[Proof of Theorem~\ref{t:BoundedGenerationScalarSum} modulo Proposition~\ref{p:BoundedGenerationScalarSum}]
Let $\delta\ll \vare_1^{5m}$. Hence by Proposition~\ref{p:BoundedGenerationScalarSum} applied to the set $\pi_{\pfr^{N_m}}(A)$, where $N_m:=\lfloor \vare_2^{m-1}N\rfloor$,  we get that for a positive integer $1\ll_{\vare_1} k$ there are $\omega_i\in \prod_k(\Omega\cup\{1\})$ such that
\begin{align}
\notag \pi_{\pfr^{N_m}}(\pfr^{\lceil\vare_2^m N\rceil}\ocal)&\subseteq \pi_{\pfr^{N_m}}(\pfr^{\vare_2 N_m}\ocal)
\\
\notag &\subseteq \gen{\pi_{\pfr^{N_m}}(A)}{k}+\pi_{\pfr^{N_m}}(\omega_1)\gen{\pi_{\pfr^{N_m}}(A)}{k}+\cdots+\pi_{\pfr^{N_m}}(\omega_k)\gen{\pi_{\pfr^{N_m}}(A)}{k}
\\
\label{e:TopThickLevel} &\subseteq \pi_{\pfr^{N_m}}(\gen{A}{k}+\pi_{\pfr^{N}}(\omega_1)\gen{A}{k}+\cdots+\pi_{\pfr^{N}}(\omega_k)\gen{A}{k}),
\end{align}
For any $1\le j\le \vare_2^{-m}$, since $N\delta\le  j\vare_2^m N \le N$, by our assumption there is $x_j\in A-A$ such that 
\[
x_j\in \pfr^{\lfloor j\vare_2^m N\rfloor}\pi_{\pfr^N}(\ocal)\setminus \pfr^{\lfloor j\vare_2^m N\rfloor+1}\pi_{\pfr^N}(\ocal). 
\]   
Hence by (\ref{e:TopThickLevel}) we have
\begin{align*}
\pi_{\pfr^N}(\pfr^{\lceil\vare_2^m N\rceil}\ocal)\subseteq &
\gen{A}{k}+\pi_{\pfr^{N}}(\omega_1)\gen{A}{k}+\cdots+\pi_{\pfr^{N}}(\omega_k)\gen{A}{k}\\
&+ x_1(\gen{A}{k}+\pi_{\pfr^{N}}(\omega_1)\gen{A}{k}+\cdots+\pi_{\pfr^{N}}(\omega_k)\gen{A}{k})\\
&+\cdots\\
&+ x_{\lceil\vare_2^{-m}\rceil}(\gen{A}{k}+\pi_{\pfr^{N}}(\omega_1)\gen{A}{k}+\cdots+\pi_{\pfr^{N}}(\omega_k)\gen{A}{k}).
\end{align*}
And, since $A$ contains a unit and $\vare_1\le \vare_2$, we have that for a positive integer $1\ll_{m,\vare_1} C$
\[
\pi_{\pfr^N}(\pfr^{\lceil\vare_2^m N\rceil}\ocal)\subseteq \gen{A}{C}+\pi_{\pfr^{N}}(\omega_1)\gen{A}{C}+\cdots+\pi_{\pfr^{N}}(\omega_k)\gen{A}{C}
\]
\end{proof}
To prove Proposition~\ref{p:BoundedGenerationScalarSum}, let us start with a direct corollary of \cite[Lemma A.1]{BG}.
\begin{lem}\label{l:LargeEntropy}
Let $K$ be a finite extension of $\bbq_p$, $\ocal$ be the ring of integers of $K$, and $\pfr$ be a uniformizing element of $K$. Suppose $B\subseteq \pi_{\pfr^N}(\ocal)$ such that for any $1\le k\le N$,
\[
\max_{\xi}|\{x\in B|\h \pi_{\pfr^k}(x)=\xi\}|<|\f|^{-(3/4)k}|B|.
\]
Then $\gen{B}{200}=\pi_{\pfr^N}(\ocal)$.
\end{lem}
\begin{proof}
It is a consequence of \cite[Lemma A.1]{BG} as it is observed in \cite[Proof of Corollary A.1]{BG}.
\end{proof}
Next following~\cite[Proof of Corollary A.1]{BG} we show how Lemma~\ref{l:LargeEntropy} helps us to deal with (extremely) large sets. 
\begin{lem}\label{l:LargeSets}
For any $0<\vare\ll 1$, $0<\delta\ll \vare$, and any finite extension $K$ of $\bbq_p$ the following holds:

Let $\ocal$ be the ring of integers of $K$, $\pfr$ be a uniformizing element of $K$, and $\f$ be the residue field. Suppose $A\subseteq \pi_{\pfr^N}(\ocal)$ such that $|A|\ge |\f|^{N(1-\delta)}$. Then 
\[
\pi_{\pfr^N}(\pfr^{\lceil \vare N\rceil}\ocal)\subseteq \gen{A}{200}.
\]
\end{lem}
\begin{proof}\footnote{This is identical to \cite[Proof of Corollary A.1]{BG}. It is included for the convenience of the reader.} Let
\[
n_0:=\max\{k|\h \max_{\xi}|\{x\in A| \pi_{\pfr^k}(x)=\xi\}|>|\f|^{-(3/4)k}|A|\}.
\]
Hence for small enough $\delta$ (to be determined later) we have 
\[
|\f|^{N-n_0}>|\f|^{-(3/4)n_0}|A|\ge |\f|^{-(3/4)n_0}|\f|^{N(1-\delta)}.
\]
Therefore we have
\be\label{e:UpperboundOnBadLevel}
n_0<4N\delta.
\ee
Let $\xi\in \pi_{\pfr^{n_0}}(\ocal)$ be such that $A':=\{x\in A|\h \pi_{\pfr^{n_0}}(x)=\xi\}$ has at least $|\f|^{-(3/4)n_0}|A|$-many elements. And let
\[
B:=\pi_{\pfr^{N-n_0}}(\{x\in \ocal|\h \pi_{\pfr^N}(x_0+\pfr^{n_0} x)\in A\}),
\]
where $\pi_{\pfr^{n_0}}(x_0)=\xi$. By Lemma~\ref{l:LargeEntropy}, we have that
\[
\gen{B}{200}=\pi_{\pfr^{N-n_0}}(\ocal).
\]
Hence 
\[
\gen{A}{200}\supseteq \pi_{\pfr^N}(\pfr^{200 n_0} \ocal).
\]
Now (\ref{e:UpperboundOnBadLevel}) gives us the claim. 
\end{proof}
\begin{proof}[Proof of Proposition~\ref{p:BoundedGenerationScalarSum}]
By Lemma~\ref{l:LargeSets}, it is enough to prove that
\be\label{e:ChangeToOrderControl}
|\gen{A}{C}+\pi_{\pfr^{N}}(\omega_1)\gen{A}{C}+\cdots+\pi_{\pfr^{N}}(\omega_C)\gen{A}{C}|\ge |\f|^{N(1-O(\vare_2))},
\ee
for a positive integer $C\gg_{\vare_1} 1$ and $\omega_i\in \prod_C(\Omega\cup\{1\})$.
One can get (\ref{e:ChangeToOrderControl}) by applying Theorem~\ref{t:ScalarSumExpansion} repeatedly and using the fact that $\vare_1\ll \vare_2$.
\end{proof}
\begin{proof}[Proof of Corollary~\ref{c:ThickSegment}]
	Since $\pi_{\pfr^{e'}}(A)=\pi_{\pfr^{e'}}(\ocal)$, we have $\pi_{\pfr}(A)=\f$. Therefore there is a subset $\Omega\subseteq A$ such that $\pi_{\pfr}$ induces a bijection between $\Omega$ and $\f^{\times}$. If $K$ is ramified over $\bbq_p$, then $e'=2$. So by the assumption, we can apply Theorem~\ref{t:BoundedGenerationScalarSum} to $\Omega\subseteq A$ and $A$, which implies the claim. Now suppose $K$ is an unramified extension of $\bbq_p$, and let $s:\f\rightarrow \pi_{\pfr^2}(A)$ be a section of $\pi_{\pfr}:\pi_{\pfr^2}(A)\rightarrow \f$. Since $K$ is an unramified extension of $\bbq_p$, $\f$ cannot be embedded into $\pi_{\pfr^2}(\ocal)$ as an additive group. Hence there are $x_1,x_2 \in \f$ such that $s(x_1)+s(x_2)-s(x_1+x_2)\neq 0$. Therefore this time we can apply Theorem~\ref{t:BoundedGenerationScalarSum} to $\Omega\subseteq A$ and $\gen{A}{2}$ and get the claim. 
\end{proof}

\section{Getting a thick $\bbz_p$-segment in a sum-product of a large set.}
In this section, first we get a {\em thick $\bbz_p$-segment} in a {\em small scale} in a sum-product set where the implied constants are independent of local field $K$, but the caveat is that we assume ${\rm gr}_{1,\pfr}(\langle A\rangle;\ocal)\neq 0$. This is based on a multi-scale analog of the Bourgain-Katz-Tao argument and another application of Mann's theorem. It is worth pointing out that this result is not needed to prove Theorem~\ref{t:ThickSegment}. Next we prove Theorem~\ref{t:ThickSegment}, where we relax the condition on the set $A$, but assume that the degree of the field extension $K/\bbq_p$ is bounded. 
\subsection{Multi-scale version of the Bourgain-Katz-Tao argument.}
In this section we give a $\pfr$-adic version of \cite[Theorem 4.3]{BKT}. One surprising result is that the implied constants are independent of the choice of local field $K$.  

\begin{lem}\label{lem-multiscale-BKT}
For any positive integer $t$, positive numbers $0<\vare_1\ll\vare_2\ll_t 1$, $0<\delta\ll_{\vare_1} 1$, any positive integer $C\gg_{\vare_1} 1$, and any finite extension $K$ of $\bbq_p$ with large, depending on $\vare_1$, residue field $\f$ the following holds: 	let $\ocal$ be the ring of integers of $K$, and $\pfr$ be a uniformizing element of $K$. Suppose $A\subseteq \pi_{\pfr^N}(\ocal)$ such that
\begin{enumerate}
\item $|\pi_{\pfr^i}(A)|\ge |\f|^{i\vare_1}$ for any $N\delta\le i\le N$.
\item $0,1\in A$ and there are $a_{1}, a_{2}\in A$ such that $a_{1}-a_{2}\in \pi_{\pfr^{N}}(\pfr\ocal\setminus \pfr^{2}\ocal)$.
\end{enumerate}
Then either 
\[
\pi_{\pfr^N}(\pfr^{\lceil\vare_2 N\rceil}\ocal)\subseteq \gen{A}{C},
\]
or 
\[
\pi_{\pfr^{\lfloor t\vare' N\rfloor}}(\gen{A}{C}\cap \pi_{\pfr^N}(\pfr^{\lceil\vare'N\rceil}\ocal))
\text{ is a ring, }
\]
for some $\vare'$ in $[\vare_2^{m(\vare_1)},\vare_2]$.
\end{lem}
\begin{proof}
Let $m:=m(\vare_1)$ be a large integer (will be determined later). By Hensel's lemma, we know that there is a subgroup  $\Omega$ of $\ocal^{\times}$ such that $\pi_{\pfr}$ induces an isomorphism between $\Omega$ and $\f^{\times}$. Let $\psi_{0,N}$ be as in the paragraph of (\ref{e:Digits}). Then by Theorem~\ref{t:BoundedGenerationScalarSum}, if $0<\delta\ll_{\vare_1,m} 1$, we have that 
\be\label{e:BoundedGenerationWithScalar}
\pi_{\pfr^N}(\pfr^{\lceil\vare_2^m N\rceil}\ocal)\subseteq \gen{A}{C_1}+\psi_{0,N}(\alpha_1)\gen{A}{C_1}+\cdots+\psi_{0,N}(\alpha_k)\gen{A}{C_1},
\ee
for some integers $k:=k(\vare_1)$ and $C_1:=C_1(\vare_1,m)$, and $\alpha_i\in \f^{\times}$.  

Now we introduce a process through which the number $k$ of the involved scalars will be reduced in the expense of enlarging $C_1$ and shrinking the size of the congruence subgroup, i.e. enlarging $\vare_2^m$. Then we will analyze the case when this process halts before getting $k=0$.  

For simplicity we say ${\rm BG}(A; \vare, k, C)$ holds if for $k$ elements  $\alpha_i\in \f^{\times}$ we have 
\be\label{e:BoundedGeneration}
\pi_{\pfr^N}(\pfr^{\lceil\vare N\rceil}\ocal)\subseteq \gen{A}{C}+\psi_{0,N}(\alpha_1)\gen{A}{C}+\cdots+\psi_{0,N}(\alpha_k)\gen{A}{C}.
\ee
{\bf Claim 1.} {\em Suppose $0<\delta_0<1$ and ${\rm BG}(A; \vare, k, C)$ holds. Then we have either {\bf (reduction)}
\be\label{ce:Reduction}
{\rm BG}(A; \vare+\delta_0, k-1, 8C),
\ee
 or {\bf ($\delta_0$-injectivity)} for any $\xbf,\xbf'\in \gen{A}{2C}^{k+1}:=\gen{A}{2C}\times \cdots \times \gen{A}{2C}$ we have that 
\be\label{ce:AlmostInjection}
l(\xbf)=l(\xbf') \Rightarrow \xbf-\xbf'\in \pi_{\pfr^N}(\pfr^{\lfloor\delta_0 N\rfloor}\ocal),
\ee
where $l(x_0,\ldots,x_k):=x_0+\psi_{0,N}(\alpha_1) x_1+\cdots+\psi_{0,N}(\alpha_k) x_k$ and $\alpha_i\in \f^{\times}$ satisfy (\ref{e:BoundedGeneration}).}

\begin{proof}[Proof of Claim 1.] Suppose $\delta_0$-injectivity fails, i.e. there are $\xbf,\xbf'\in \gen{A}{2C}\times \cdots \times \gen{A}{2C}$ such that
\begin{enumerate}
\item $\xbf-\xbf'\not\in \pi_{\pfr^N}(\pfr^{\lfloor\delta_0 N\rfloor} \ocal)^{k+1}$, and
\item $l(\xbf)=l(\xbf')$.
\end{enumerate} 
Then, for some $i_0$, $(x_{i_0}-x_{i_0}')\pi_{\pfr^N}(\ocal)\supseteq \pi_{\pfr^N}(\pfr^{\lfloor \delta_0 N\rfloor} \ocal)$. Without loss of generality let us assume that it happens for $i_0=k$ (notice that, if $i_0=0$, we can multiply both sides by $\psi_{0,N}(\alpha_1^{-1})$ to make sure that one of the remaining coefficients is one). Hence we have
\begin{align}
\notag
\pi_{\pfr^N}(\pfr^{\lceil \vare N\rceil+\lfloor \delta_0 N\rfloor} \ocal)\subseteq& 
\psi_{0,N}(\alpha_0) \gen{A}{2C} (x_k-x_k')+\cdots+\psi_{0,N}(\alpha_k) \gen{A}{2C} (x_k-x_k')
\\
\notag
\text{(since $l(\xbf)=l(\xbf')$,) \h\h}\subseteq& \psi_{0,N}(\alpha_0) \gen{A}{4C}+\cdots+\psi_{0,N}(\alpha_{k-1}) \gen{A}{4C}-\left(\sum_{i=0}^{k-1} \psi_{0,N}(\alpha_i)(x_i-x_i')\right)\gen{A}{2C} 
\\
\label{e:Reduction}
\subseteq& \psi_{0,N}(\alpha_0) \gen{A}{8C}+\cdots+\psi_{0,N}(\alpha_{k-1}) \gen{A}{8C},
\end{align}
which means that ${\rm BG}(A; \vare+\delta_0,k-1,8C)$ holds.
\end{proof}

{\bf Claim 2.} {\em Suppose ${\rm BG}(A;\vare,k,C)$ holds and $\alpha_i\in \f^{\times}$ satisfy (\ref{e:BoundedGeneration}). If $l(\xbf):=\sum_i \alpha_i x_i$ is $\delta_0$-injective on $\gen{A}{2C}^{k+1}$ for some $\delta_0> \vare$ (see (\ref{ce:AlmostInjection})), then 
\[
\pi_{\pfr^{\lfloor\delta_0 N\rfloor}}\left(\gen{A}{C}\cap \pi_{\pfr^N}(\pfr^{\lceil\vare N\rceil}\ocal)\right)
\]
is closed under addition and multiplication.}

\begin{proof}[Proof of Claim 2.] Let $x,x'\in \gen{A}{C}\cap \pi_{\pfr^N}(\pfr^{\lceil \vare N \rceil}\ocal)$. So there is $\xbf_1\in \gen{A}{C}\times \cdots \times \gen{A}{C}$ such that $l(\xbf_1)=x+x'=l(x+x',0,\cdots,0)$.

By assumption for any $\xbf_1,\xbf'_1\in \gen{A}{2C}\times \cdots \times \gen{A}{2C}$ we have that 
\be\label{e:AlmostInjective}
l(\xbf_1)=l(\xbf'_1) \Rightarrow \xbf_1-\xbf'_1\in \pi_{\pfr^N}(\pfr^{\lfloor \delta_0 N \rfloor}\ocal)^{k+1}.
\ee
Hence we have
\[
\xbf_1\equiv (x+ x',0,\ldots,0) \pmod{\pfr^{\lfloor \delta_0 N \rfloor}},
\]
which implies that $\pi_{\pfr^{\lfloor \delta_0 N \rfloor}}(A\cap \pi_{\pfr^N}(\pfr^{\lceil \vare N \rceil}\ocal))$ is closed under addition.

Similarly there is $\xbf_2\in \gen{A}{C}\times \cdots \times \gen{A}{C}$ such that $l(\xbf_2)=xx'=l(xx',0,\cdots,0)$. As $xx'\in \gen{A}{2C}$, again as above $\delta_0$-injectivity on $\gen{A}{2C}^{k+1}$ implies that $\xbf_2\equiv (xx',0,\ldots,0) \pmod{\pfr^{\lfloor \delta_0 N \rfloor}}$, which implies that $\pi_{\pfr^{\lfloor \delta_0 N \rfloor}}(A\cap \pi_{\pfr^N}(\pfr^{\lceil \vare N \rceil}\ocal))$ is closed under multiplication.
\end{proof}

Having the above Claims, we inductively define three sequences $\{\vare_i'\}, \{k_i'\}, \{C_i'\}$ of numbers:
\begin{align*}
\vare_0':= \vare_2^m,&\hspace{1cm} \vare_{i+1}':= (t+1)\vare_i';\\
k_0':=k,& \hspace{1cm} k_{i+1}':=k_i-1;\\
C_0':=C_1,& \hspace{1cm} C_{i+1}':=8C_i'.
\end{align*}
First notice that for $\vare_2\ll_t 1$ and $k(\vare_1)\le m(\vare_1)$ we have that 
\be\label{e:UpperboundConstants}
\vare_i'\le (t+1)^k \vare_2^m\le (t+1)^k (t+1)^{-m} \vare_2\le \vare_2,
\hspace{1cm}
C_i' \le 8^k C_1 \ll_{\vare_1} 1.
\ee
We know that ${\rm BG}(A; \vare_0',k_0', C_0')$ holds. Suppose $i_0$ is the smallest non-negative integer such that  
\[
{\rm BG}(A; \vare_{i_0+1}', k_{i_0+1}',C_{i_0+1}')
\]
 does not hold. If $i_0=k$, then ${\rm BG}(A; \vare_2, 0, 4^k C_1)$ holds. And we are done. Suppose $i_0<k$. So, by Claim 1, ${\rm BG}(A; \vare_{i_0}',k_{i_0}', C_{i_0}')$ holds for some $\alpha_j\in \f^{\times}$ and $l(\xbf):=\sum_j \alpha_i x_j$ is $t\vare_{i_0}'$-injective on $\gen{A}{2C_{i_0}'}^{k_{i_0}'+1}$.  Therefore, by Claim 2, 
\[
\pi_{\pfr^{\lfloor t\vare_{i_0}' N\rfloor}}\left(\gen{A}{C_{i_0}'}\cap \pi_{\pfr^N}(\pfr^{\lceil\vare_{i_0}' N\rceil}\ocal)\right)
\]
is closed under addition and multiplication; and the claim follows.
\end{proof}
In order to get a meaningful conclusion from Lemma~\ref{lem-multiscale-BKT}, we have to show that $\gen{A}{C}$ has an element with $\pfr$-valuation roughly equal to $\lceil \vare'N\rceil$ as otherwise $\pi_{\pfr^{\lfloor t\vare' N\rfloor}}(\gen{A}{C}\cap \pi_{\pfr^N}(\pfr^{\lceil\vare'N\rceil}\ocal))
$ can be a very small set. For that purpose, next we will observe that Lemma~\ref{l:LargeInterval} gives us such a control.

\begin{proof}[Proof of Theorem~\ref{thm-multiscaleBKT}]
Suppose the implied constants are so that the given inequalities in Lemma~\ref{lem-multiscale-BKT} are satisfied for the parameters $\vare_1,\vare_2,t, \delta_0,$ and $C_0$. By changing the implied constants, we can further assume that $\delta_0\le \vare_1^{m(\vare_1)+1}\le \vare_1\vare_2^{m(\vare_1)}$ and $C_0\ge 6/\vare_1$. Suppose $\delta\le \vare_1\delta_0/4$ and $C\gg C_0/(\vare_1^2\delta_0)$. Notice that for some $C'\le 4/(\vare_1\delta_0)$ we have that $|\gen{A}{C'}+\gen{A}{C'}|\le |\gen{A}{C'}||\f|^{N\vare_1\delta_0/4}$. Hence by Lemma~\ref{l:RegularSubset} and a similar argument as in the proof of Theorem~\ref{t:BoundedGenerationScalarSum}, there is an $(m_0,\ldots,m_{N-1})$-regular subset $A'$ of $\gen{A}{3C'}$ such that 
\begin{itemize}
\item $|\pi_{\pfr^i}(A')|\ge |\f|^{i\vare_1/2}$ for any $N\delta_0 \le i\le N$, and
\item $m_0,m_1>1$
\end{itemize}
 By Lemma~\ref{l:LargeInterval} for the numbers $x_i:=\log m_i/\log |\f|$ we have 
 \be\label{eq-large-image-valutaion}
 (N\lceil 2/\vare_1 \rceil \delta_0,N)\cap \bbz \subseteq \textstyle \sum_{6\lceil 1/\vare_1\rceil}v_{\pfr}(\gen{A}{6C'})\subseteq v_{\pfr}(\gen{A}{36C'\lceil 1/\vare_1\rceil});
 \ee
 and the claim follows.
\end{proof}

\subsection{Proof of Theorem~\ref{t:ThickSegment}}
Let us start with (a variation of) \cite[Theorem~4]{BKT}. We include the proof for the convenience of the reader. 
\begin{lem}\label{l:BoundedGenerationFiniteField}
For any $0<\vare\ll 1$, positive integer $C\gg_{\vare} 1$, and a finite field $\f$ the following holds: 

Suppose $B\subseteq \f$, $|B|\ge |\f|^{\vare}$, and $0,1\in B$. Then $\gen{B}{C}$ is a subfield of $\f$. 
\end{lem}
\begin{proof}
By \cite[Lemma 4.1]{BKT}, there are $\alpha_1,\ldots,\alpha_k\in \f^{\times}$ such that $k\ll_{\vare} 1$ and 
\[
\alpha_1 B+\cdots+\alpha_k B=\f.
\]
{\bf Claim 1:} {\em Suppose $0,1\in X\subseteq \f$ and $\alpha_i\in \f^{\times}$ such that 
\be\label{ce:BoundedGenerationFiniteField}
\f=\alpha_1 X+\cdots+\alpha_k X.
\ee
Then either we have {\bf (reduction)}
$
\f=\sum_{i\neq i_0} \alpha_i \gen{X}{2},
$
for some $i_0$, or {\bf (injectivity)} for any $\xbf,\xbf'\in X^k$
$
\sum_i \alpha_i x_i=\sum_i \alpha_i x_i' \Rightarrow \xbf=\xbf'.
$
} 
\begin{proof}[Proof of Claim] Suppose that the injectivity does not hold, i.e. there are $\xbf\neq\xbf'\in X^k$ such that 
\be\label{e:LinearCombination}
\sum_i \alpha_i x_i=\sum_i \alpha_i x_i'.
\ee   
Without loss of generality we can assume that $x_k\neq x_k'$. Thus
\begin{align*}
\f&= \alpha_1 (x_k-x_k')X+\cdots+\alpha_k (x_k-x_k')X\\
\text{(by (\ref{e:LinearCombination}))}\hspace{.5cm}
&\subseteq \alpha_1 (X\cdot X-X\cdot X)+\cdots+\alpha_{k-1} (X\cdot X-X\cdot X)+\left(\sum_{i=1}^{k-1}\alpha_i(x_i-x_i')\right)X\\
&\subseteq \alpha_1 \gen{X}{2}+\cdots+\alpha_{k-1} \gen{X}{2}.
\end{align*}
\end{proof} 

{\bf Claim 2:} {\em Suppose $0,1\in X\subseteq \f$ and $\alpha_i\in \f^{\times}$ such that 
\[
\f=\alpha_1 X+\cdots+\alpha_k X.
\]
Suppose for any $\xbf,\xbf'\in \gen{X}{2}^k$ we have
\[
\sum_i \alpha_i x_i=\sum_i \alpha_i x_i'\Rightarrow \xbf=\xbf'.
\]
Then $X$ is a subfield of $\f$.
}
\begin{proof}[Proof of Claim] It is enough to show $X\cdot X=X$ and $X+X=X$. For any $y,y'\in X$, there is $\xbf\in X^k$ such that 
\[
\alpha_1(y+y')=\sum_i \alpha_i x_i.
\]
Hence $y+y'=x_1\in X$. And so $X$ is closed under addition. Similarly it is closed under multiplication. 
\end{proof}
Now suppose $i_0\le k$ be the largest non-negative integer such that 
\[
\f=\alpha_1'\gen{B}{4^{i_0}}+\cdots+\alpha_{k-i_0}'\gen{B}{4^{i_0}},
\]
for some $\alpha_i'\in \f^{\times}$. If $i_0=k$, we are done. If not, then by Claim 1 for $X=\gen{\gen{B}{4^{i_0}}}{2}$ we have that for any 
$\xbf,\xbf'\in \gen{\gen{B}{4^{i_0}}}{2}^{k-i_0}$ we have
\[
\sum_i \alpha_i' x_i=\sum_i \alpha_i' x_i' \Rightarrow \xbf=\xbf'.
\]
Hence, by Claim 2, $\gen{B}{4^{i_0}}$ is a subfield of $\f$. 
\end{proof}

\begin{prop}\label{prop-bounded-generation-for-equal-grades}
	For any positive integers $0<\vare_1\ll \vare_2\le 1/2$, $0<\delta\ll_{\vare_1} 1$, positive integers  $d$ and $1\ll_{\vare_1,d} C$, the following holds: suppose $K$ is a field extension of $\bbq_p$ and $[K:\bbq_p]\le d$. Let $\ocal$ be the ring of integers of $K$, $\pfr$ be a uniformizing element of $K$, and $\f$ be the residue field of $K$. Suppose $|\f|\gg_{\vare_1,d} 1$. Suppose $A$ is a subset of $\ocal$ which contains $0$ and $1$. Let $R$ be the closure of the subring of $\ocal$ that is generated by $A$. Suppose 
	\begin{itemize}
	\item[(C1)] {\em (Equality of grades)} for any integer $i$ in $[0,N-1]$, $|{\rm gr}_{i,p}(R;\ocal)|=|{\rm gr}_{0,p}(R;\ocal)|$, where ${\rm gr}_{i,p}(R;\ocal):=\pi_{p^{i+1}}(R\cap p^i\ocal)$.
	\item[(C2)] {\em (Bound for the box dimension)} for any integer $i$ in $[Ne\delta,Ne]$, $|\pi_{\pfr^i}(A)|\ge |\f|^{i\vare_1}$, where $e$ is the ramification index of $K$ over $\bbq_p$. 
	\item[(C3)] {\em (Bound for level $\pfr$)} $|\pi_{\pfr}(A)|\ge |\f|^{\vare_1}$.	
	\end{itemize}
 Then $\pi_{p^N}(\gen{A}{C})\supseteq p^{\lceil N\vare_2\rceil} \pi_{p^N}(R)$.
	\end{prop}
\begin{proof}
{\bf Step 1.} (Describing $R$ for large $N$) By Theorem~\ref{t:SubringOfO}, there is a positive integer $N_0$ depending only on $[K:\bbq_p]$ such that, if $N\ge N_0$, then there are positive integers $	a$ and $b$, and a subfield $K_0$ of $K$ such that $
b-a\ge 6[K:\bbq_p]$, $N\ge b,$
\be\label{eq-passing-to-subfield}
\pi_{p^b}(R)\subseteq \pi_{p^b}(\ocal_0) \text{, and } \pi_{p^b}(R\cap p^a\ocal)=\pi_{p^b}(\ocal_0\cap p^a \ocal).
\ee

By Corollary~\ref{cor-grade-shift} we have that $x+p\ocal\mapsto p^{i}x+p^{i+1}\ocal$ is an injection from ${\rm gr}_{0,p}(R;\ocal)$ to ${\rm gr}_{i;p}(R;\ocal)$. Since by our assumption $|{\rm gr}_{i;p}(R;\ocal)|=|\pi_p(R)|$ for any integer $i$ in $[0,N-1]$, we deduce that $x+p\ocal\mapsto p^{i}x+p^{i+1}\ocal$ is a bijection from ${\rm gr}_{0,p}(R;\ocal)$ to ${\rm gr}_{i,p}(R;\ocal)$. By the equality of grades (condition C1) and \eqref{eq-passing-to-subfield} we deduce that 
\be\label{eq-verifying-the-condtions-grade-equality}
|{\rm gr}_{i,p}(R;\ocal)|=|{\rm gr}_{a,p}(R;\ocal)|=|{\rm gr}_{a,p}(\ocal_0;\ocal)|=|{\rm gr}_{0,p}(\ocal_0;\ocal)|
\ee
for any integer $i$ in $[0,N-1]$. In particular, $|\pi_{p^{j}}(R\cap p^i\ocal)|=|\pi_{p^{j-i}}(\ocal_0)|$ for any integers $i<j$ in $[0,N-1]$. By Corollary~\ref{cor-grade-shift} we have that $x+p^{b-a}\ocal\mapsto p^{a}x+p^{b}\ocal$ is an injection from $\pi_{p^{b-a}}(R)$ to $\pi_{p^b}(R\cap p^a\ocal)$; and as these sets have equal cardinality, we deduce that this map is a bijection. The same can be said for the ring $\ocal_0$ instead of $R$. Therefore by \eqref{eq-passing-to-subfield} and $b-a\ge 6[K:\bbq_p]$, we have 
\be\label{eq-verifying-the-conditions-initial-approximation}
\pi_{p^{6[K:\bbq_p]}}(R)=\pi_{p^{6[K:\bbq_p]}}(\ocal_0).
\ee
By \eqref{eq-verifying-the-condtions-grade-equality}, \eqref{eq-verifying-the-conditions-initial-approximation}, and Proposition~\ref{prop:detecting-ring-of-integers-via-grading-structure}, we have 
\be\label{eq-R-is-ring-of-integers}
\pi_{p^{N-4}}(R)=\pi_{p^{N-4}}(\ocal_0).
\ee

{\bf Step 2.} (Bounded generation of $\pi_{\pfr^{O(1)}}(R)$) By condition (C3), we have $|\pi_{\pfr}(A)|\ge |\f|^{\vare_1}$, and we also have that $0,1\in \pi_{\pfr}(A)$; hence by Lemma~\ref{l:BoundedGenerationFiniteField} we get that 
\be\label{eq-bounded-generation-mod-uniformizer}
\pi_{\pfr}(\gen{A}{C_1})=\pi_{\pfr}(R) \text{ and } [\f:\pi_{\pfr}(R)]\le 1/\vare_1
\ee 
where $C_1$ is an integer that only depends on $\vare_1$. By induction on $i$, we show that 
\be\label{eq-bounded-generation-level-i}
\pi_{\pfr^i}(\gen{A}{C_i})=\pi_{\pfr^i}(R)
\ee
where $C_i$ is an integer that depends only on $i$ and $\vare_1$. If $\pi_{\pfr^{i+1}}(\gen{A}{C_i})$ is a ring, then $\pi_{\pfr^{i+1}}(\gen{A}{C_i})=\pi_{\pfr^{i+1}}(R)$; and we can set $C_{i+1}:=C_i$. If not, 
$
\pi_{\pfr^{i+1}}(\gen{A}{3C_i}\cap \pfr^{i}\ocal)\neq 0.
$
By \eqref{eq-bounded-generation-mod-uniformizer}, $\pi_{\pfr^{i+1}}(\pfr^i \ocal)$ is an $\pi_{\pfr}(R)$-vector space of dimension at most $1/\vare_1$; and so we can deduce that 
	the $\pi_{\pfr}(R)$-subspace spanned by $\pi_{\pfr^{i+1}}(\gen{A}{3C_i}\cap \pfr^{i}\ocal)$ is contained in $\pi_{\pfr^{i+1}}(\gen{A}{6\lceil 1/\vare_1\rceil C_i}\cap \pfr^{i}\ocal)$. Since length of any chain of $\pi_{\pfr}(R)$-subspaces of $\pi_{\pfr^{i+1}}(\pfr^i\ocal)$ is at most  $1/\vare_1$, we deduce that $\pi_{\pfr^{i+1}}(\gen{A}{C_{i+1}})=\pi_{\pfr^{i+1}}(R)$ where $C_{i+1}:=(6\lceil 1/\vare_1\rceil)^{\lceil 1/\vare_1\rceil} C_i$. 
	 
{\bf Step 3.} (Finishing proof for large $N$) Suppose $N\ge N_0$ where $N_0$ is given in Step 1. Then by condition (C2), \eqref{eq-R-is-ring-of-integers} (Step 1), \eqref{eq-bounded-generation-level-i} (Step 2), and Corollary~\ref{c:BoundedGenerationFullModP}, we have  
\be\label{eq-N-4}
\pi_{p^{N-4}}(p^{\lceil\vare_2 (N-4)\rceil}R)=\pi_{p^{N-4}}(p^{\lceil\vare_2 (N-4)\rceil}\ocal_0)\subseteq \pi_{p^{N-4}}(\gen{A}{C})
\ee
where $C$ is any integer that is larger than a function of $\vare_1$. Next in \eqref{eq-N-4}, we need to change the level from $p^{N-4}$ to $p^N$. Suppose $s:\pi_{p^{N-4}}(p^{\lceil\vare_2 (N-4)\rceil}R)\rightarrow \gen{A}{C}$ is
a section of $\pi_{p^{N-4}}$; that means 
\be\label{eq-section}
\text{for any $x\in \pi_{p^{N-4}}(p^{\lceil\vare_2 (N-4)\rceil}R)$ we have $\pi_{p^{N-4}}(s(x))=x$.}
\ee
 Let $X$ be the image of $s$. Suppose $N> 6$; then $\lceil\vare_2 (N-4)\rceil<(N-4)-1$. Since $R$ is a finite rank free $\bbz_p$-submodule of $\ocal$ and $\lceil\vare_2 (N-4)\rceil<(N-4)-1$,  \eqref{eq-section} implies that the $\bbz$-span of $X$ is dense in $p^{\lceil\vare_2 (N-4)\rceil}R$. Hence the group generated by $\pi_{p^{N-4}\pfr}(X)$ is $\pi_{p^{N-4}\pfr}(p^{\lceil\vare_2 (N-4)\rceil}R)$. On the other hand, by \eqref{eq-section}, we have that 
 \[
 v_{\pfr}(s(x_1+x_2)-s(x_1)-s(x_2))\ge (N-4)e
 \]
 for any $x_1,x_2\in \pi_{p^{N-4}}(p^{\lceil\vare_2 (N-4)\rceil}R)$. 
 
 If $\pi_{p^{N-4}\pfr}\circ s$ is a group homomorphism, then $\pi_{p^{N-4}\pfr}(X)=\pi_{p^{N-4}\pfr}(p^{\lceil\vare_2 (N-4)\rceil}R)$; in particular, we have 
 \begin{align*}
 |\pi_{p^{N-4}}(p^{\lceil\vare_2 (N-4)\rceil}R)|&=|\pi_{p^{N-4}}(X)| =|X|
 \\
& \ge |\pi_{p^{N-4}\pfr}(X)| =|\pi_{p^{N-4}\pfr}(p^{\lceil\vare_2 (N-4)\rceil}R)|
\\
& =|\pi_{p^{N-4}}(p^{\lceil\vare_2 (N-4)\rceil}R)||\pi_{p^{N-4}\pfr}(p^{\lceil\vare_2 (N-4)\rceil}R\cap p^{N-4}\ocal)|
\\
&\ge |\pi_{p^{N-4}}(p^{\lceil\vare_2 (N-4)\rceil}R)||\pi_{\pfr}(R)|,
 \end{align*}
 which is a contradiction. Therefore $\pi_{p^{N-4}\pfr}\circ s$ is not a group homomorphism; and so there is $x'\in X-X-X$ such that $v_{\pfr}(x')=(N-4)e$.
Since $v_{\pfr}(p^{-(N-4)}x')=0$, we have that $|\pi_p(R)\pi_p(p^{-(N-4)}x')|=|\pi_p(R)|$; and so by \eqref{eq-bounded-generation-mod-uniformizer} and Lemma~\ref{lem-attached-graded-rings}, we have ${\rm gr}_{N-4,p}(R;\ocal)\subseteq \pi_{p^{N-3}}(\gen{A}{3C+C_1})$. Hence by \eqref{eq-N-4} we deduce that 
\[
\pi_{p^{N-3}}(p^{\lceil\vare_2 (N-4)\rceil}R)\subseteq \pi_{p^{N-3}}(\gen{A}{4C+C_1}).
\]  
Repeating this argument 3 more times, we get that 
\be\label{eq-N}
\pi_{p^{N}}(p^{\lceil\vare_2 (N-4)\rceil}R)\subseteq \pi_{p^{N}}(\gen{A}{C'}),
\ee
where $C'$ is any positive integer that is larger than a function of $\vare_1$.

{\bf Step 4.} (Finishing proof for small $N$) If $N<\max(N_0,7)$ where $N_0$ is given in Step 1, then by Step 2 we have $\pi_{p^N}(\gen{A}{C})=\pi_{p^N}(\ocal_0)$ for any integer $C$ that is larger than a function of $N_0$
\end{proof}
Proof of the next Lemma is an adaptation of the argument given in \cite[Section A.3]{BG}.

\begin{lem}[Bourgain]\label{lem-bounded-generation-large-mod-all-powers-till-N}
	For any positive integers $0<\vare\ll 1$, positive integers  $d$,  $1\ll_{d,\vare} C$, and $1\ll_{d,\vare} N$, the following holds: suppose $K$ is a field extension of $\bbq_p$ and $[K:\bbq_p]\le d$. Let $\ocal$ be the ring of integers of $K$, $\pfr$ be a uniformizing element of $K$, and $\f$ be the residue field of $K$.
	Suppose $|\f|\gg_{\vare,d} 1$. Suppose $A$ is a subset of $\ocal$ which contains $0$ and $1$. Let $R$ be the closure of the subring of $\ocal$ that is generated by $A$. Suppose, for any integer $i$ in $[1,Ne]$, $|\pi_{\pfr^i}(A)|\ge |\f|^{i\vare}$, where $e$ is the ramification index of $K$ over $\bbq_p$.  Then there are positive integers $m$ and $n$, $a\in \ocal$, and a closed subfield $K_0$ of $K$ with ring of integers $\ocal_0$ such that 
\begin{align}
\notag
n \ll_{d,\vare} N
&
\ll_{d,\vare} n-m,
& \text{ (Exponent conditions) }
\\
\notag
\pi_{\pfr^n}(\ocal_0 a)
&
\subseteq \pi_{\pfr^n}(\gen{A}{C}), 
\hspace{1cm}
v_{\pfr}(a)=m,
&
\text{ (Bounded generation) }
\\
\notag
|\pi_p(\ocal_0)|&\ge |\pi_p(R)|.
&
\text{ (Box dimension control) }
\end{align}
	\end{lem}
\begin{proof} The key point is that we can detect in a {\em bounded} number of steps whether the grades of the ring generated by $A$ are getting larger. 

Let $\vare_1:=\vare$ and assume $0<(\vare_1/2^{d+1})\ll\vare_2\le 1/4$ satisfy the inequality given in Proposition~\ref{prop-bounded-generation-for-equal-grades}.

{\bf Step 0.} (Setup) We will recursively define a sequence of quadruples $(A_i, R_i, n_i, C_i)$ of subsets $A_i$ of $\ocal$, subrings $R_i$ of $\ocal$, and positive integers $n_i$ and $C_i$ with the following properties:
\begin{itemize}
\item[(P1)] (Ring conditions) $R_i$ is the closure of the ring generated by $A_i$; and 
\[
|{\rm gr}_{0,p}(R_i;\ocal)|=|{\rm gr}_{1,p}(R_i;\ocal)|=\cdots=|{\rm gr}_{n_i-1,p}(R_i;\ocal)|.
\]
\item[(P2)] (Set conditions) 
		\begin{itemize}
			\item[(P2-a)] $0,1\in A_i$.
			\item[(P2-b)] Let $\delta_0$ be a small enough positive number (depending on $\vare_1$ and $d$) so that Proposition~\ref{prop-bounded-generation-for-equal-grades} holds for the parameter $\vare_1/2^d$ instead of $\vare_1$; moreover we assume that $\delta_0\le \vare_1/2^{d+1}$. For any integer $j$ in $[Ne\delta_0^{2},Ne]$, $|\pi_{\pfr^j}(A_i)|\ge |\f|^{\vare_1/2^i}$.
			\item[(P2-c)] $\pi_{\pfr}(A_i)=\pi_{\pfr}(R_i)$.
		\end{itemize}  
\item[(P3)] (Bounded generation) $\pi_{p^{n_i}}(\gen{A_i}{C_i})\supseteq \pi_{p^{n_i}}(p^{\lceil n_i \vare_2 \rceil}R_i)$ and $C_i\ll_{d,\vare_1} 1$.	
\item[(P4)] (Connection between sets) 
	\begin{itemize}
		\item[(P4-a)] $\alpha_i(\gen{A_i}{2}\cap \beta_i \ocal)\subseteq \alpha_i\beta_iA_{i+1}\subseteq \gen{A_i}{C_i}$ for some $\alpha_i,\beta_i\in \ocal$ such that $v_{\pfr}(\alpha_i),v_{\pfr}(\beta_i)\ll_{d,\vare} n_i$ and $\pi_p(p^{-v_{\pfr}(\alpha_i\beta_i)/e}\alpha_i\beta_i)=1$.
		\item[(P4-b)] $|\pi_p(R_i)|<|\pi_p(R_{i+1})|$.
	\end{itemize}
\end{itemize}
And we stop when $n_i\ge N\delta_0^2/2$.  

{\bf Step 1.} Since $|\pi_{\pfr}(A)|\ge |\f|^{\vare_1}$, by Lemma~\ref{l:BoundedGenerationFiniteField} there is a positive integer $C_0'$ that is at most a function of $\vare_1$ such that 
\be\label{eq-level-1-ring-generation}
\text{$\pi_{\pfr}(\gen{A}{C_0'})$ is a subfield of $\f$.} 
\ee
Let $A_0:=\gen{A}{C_0'}$; we notice that $A_0$ satisfies (P2) because of the assumption and \eqref{eq-level-1-ring-generation}. 

{\bf Step 2.} Suppose we have already defined $A_i$ that satisfies Property (P2). At this step, we more or less get the Property (P3) and give an indication on what $C_i$ can be. Furthermore we introduce an auxiliary set $A_{i+1}'$ and an auxiliary ring $R_{i+1}'$. This pair $(A_{i+1}',R_{i+1}')$ will help us to enlarge $\pi_p(R_i)$ in a bounded number of steps. 

 Let $R_i$ be the closure of the subring generated by $A_i$; and let $n_i$ be the largest positive integer in $[0,N]$ such that the first $n_i$ grades ${\rm gr}_{j,p}(R_i;\ocal)$ of $R_i$ have equal number of elements; that means (P1) holds. By Proposition~\ref{prop-bounded-generation-for-equal-grades} there is a positive integer $\overline{C}_i\ll_{d,\vare_1} 1$ such that $\pi_{p^{n_i}}(\gen{A_i}{\overline{C}_i})\supseteq \pi_{p^{n_i}}(p^{\lceil n_i \vare_2 \rceil}R_i)$. So for any positive integer $C_i$ in $[\overline{C}_i,\Theta_{d,\vare_1}(\overline{C}_i)]$ Property (P3) holds. 
 
 If $n_i\ge N\delta_0^2/2$, we let $C_i:=\overline{C}_i$ and we are done. If not, we have to proceed and define $A_{i+1}$ and $C_i$, and make sure that (P3) and (P4) hold. We also notice that  
 \[
 |\pi_p(R_0)|<|\pi_p(R_1)|<\dots<|\pi_p(R_i)|
 \]
 and $\pi_p(R_j)$'s are $\f_p$-subspaces of $\pi_p(\ocal)$. Hence 
 \be\label{eq-upper-bound-number-of-steps}
 i\le d.
 \ee 
  Next we proceed as in Step 3 of proof of Proposition~\ref{prop-bounded-generation-for-equal-grades}: let $m_i:=\lceil n_i\vare_2\rceil$ and $s:\pi_{p^{n_i}}(p^{m_i}R_i)\rightarrow \gen{A_i}{\overline{C}_i}$ be a section of $\pi_{p^{n_i}}$ which sends $0$ to $0$; that means for any $x\in 
\pi_{p^{n_i}}(p^{m_i}R_i)$ we have $\pi_{p^{n_i}}(s(x))=x$ and $s(0)=0$. Let 
\be\label{eq-auxiliary-set}
A_{i+1}':=\{s(p^{m_i})^{-1}s(x)|\h x\in \pi_{p^{n_i}}(p^{m_i}R_i)\}\cup s(p^{m_i})^{-1}(\gen{A}{2}\cap p^{n_i}\ocal)\subseteq s(p^{m_i})^{-1}\gen{A_i}{\overline{C}_i}, 
\ee
and $R_{i+1}'$ be the closure of the ring generated by $A_{i+1}'$. Hence $1\in A_{i+1}'$ and 
\be\label{eq-initiation-of-the-next-step}
\pi_{p^{n_i-m_i}}(A_{i+1}')=\pi_{p^{n_i-m_i}}(R_i)=\pi_{p^{n_i-m_i}}(R_{i+1}'),
\ee
where the second equality holds as $\pi_{p^{n_i-m_i}}(R_i)$ is a ring. Let $n_{i+1}'$ be the largest positive integer such that the first $n_{i+1}'$-th grades of $R_{i+1}'$ have equal sizes; that means
\[
|{\rm gr}_{0,p}(R_{i+1}';\ocal)|=\cdots=|{\rm gr}_{n_{i+1}'-1,p}(R_{i+1}';\ocal)|.
\]
By \eqref{eq-initiation-of-the-next-step} and Property (P1) for $R_i$, we have that $n_{i+1}'\ge n_i-m_i$. 

{\bf Step 3.} In this step, we define $A_{i+1}$, $C_i$, $\alpha_i$, and $\beta_i$ under the assumption that $n_{i+1}'\le n_i+1$.

Since the first $n_{i+1}'$-th grades of $R_{i+1}'$ have equal sizes and $n_{i+1}'\ge n_i-m_i$, we have that $|\pi_{p^{\lceil n_{i+1}'/2 \rceil}}(R_{i+1}')|=|\pi_{p^{n_{i+1}'}}(R_{i+1}'\cap p^{\lfloor n_{i+1}'/2 \rfloor}\ocal)|$. Hence multiplication by the element $s(p^{m_i})^{-1}s(p^{\lfloor n_{i+1}'/2 \rfloor+m_i})\in A_{i+1}'$ induces a bijection from $\pi_{p^{\lceil n_{i+1}'/2 \rceil}}(R_{i+1}')$ to $\pi_{p^{n_{i+1}'}}(R_{i+1}'\cap p^{\lfloor n_{i+1}'/2 \rfloor}\ocal)$ (see Corollary~\ref{cor-grade-shift}). Since $\lceil n_{i+1}'/2\rceil\le n_i-m_i$ and $\pi_{p^{n_i-m_i}}(A_{i+1}')=\pi_{p^{n_i-m_i}}(R_{i+1}')$ (see \eqref{eq-initiation-of-the-next-step}), we get that 
\be\label{eq-next-to-the-larger-grade}
\pi_{p^{n_{i+1}'}}(R_{i+1}')=\pi_{p^{n_{i+1}'}}(\gen{A_{i+1}'}{2}).
\ee
Next we show that $\pi_{p^{n_{i+1}'+1}}(R_{i+1}')=\pi_{p^{n_{i+1}'+1}}(\gen{A_{i+1}'}{O_{[K:\bbq_p]}(1)})$. Let $c_0:=2$ and $V_0$ be the zero $\f_p$-vector space; we will recursively define an increasing sequence $\{c_j\}_j$ of positive integers and $\f_p$-vector spaces $V_j$ such that 

\begin{itemize}
\item[($\ast$)]	
either $\pi_{p^{n_{i+1}'+1}}(\gen{A_{i+1}'}{c_j})$ is a ring or there is a subspace $V_{j+1}$ of ${\rm gr}_{n_{i+1}',p}(R_{i+1}';\ocal)$ which is a subset of $\pi_{p^{n_{i+1}'+1}}(\gen{A_{i+1}'}{c_{j+1}}\cap p^{n_{i+1}'}\ocal)$ and $\dim_{\f_p} V_{j+1}=j+1$. 
\end{itemize}
Suppose $\pi_{p^{n_{i+1}'+1}}(\gen{A_{i+1}'}{c_j})$ is not a ring; then by \eqref{eq-next-to-the-larger-grade} there is $x_j\in R_{i+1}'\cap p^{n_{i+1}'}\ocal \cap \gen{A_{i+1}'}{3c_j}$ such that $\pi_{p^{n_{i+1}'+1}}(x_j)\not\in \pi_{p^{n_{i+1}'+1}}(\gen{A_{i+1}'}{c_j})$. Hence by the graded structure of ${\rm gr}_p(R_{i+1}';\ocal)$ and \eqref{eq-next-to-the-larger-grade}, we have that $\f_p x_j+V_j\subseteq \pi_{p^{n_{i+1}'+1}}(\gen{A_{i+1}'}{1+4c_j})$. Hence $V_{j+1}:=\f_p x_j+V_j$ and $c_{j+1}:=1+4c_j$ satisfy ($\ast$).

Since ${\rm gr}_{n_{i+1}',p}(R_{i+1}';\ocal)$ is an $\f_p$-vector space of dimension at most $\dim_{\f_p}(\ocal/p\ocal)=[K:\bbq_p]$, we get that $\pi_{p^{n_{i+1}'+1}}(R_{i+1}')=\pi_{p^{n_{i+1}'+1}}(\gen{A_{i+1}'}{C_{i}'})$ where $C_{i}'\ll_d 1$. Let $s':\pi_{p^{n_{i+1}'+1}}(R_{i+1}')\rightarrow \gen{A_{i+1}'}{C_{i}'}$ be a section of $\pi_{p^{n_{i+1}'+1}}$; and 
\be\label{eq-next-set-case-1}
A_{i+1}'':=s'(p^{n_{i+1}'})^{-1} (A_{i+1}'\cap p^{n_{i+1}'}\ocal). 
\ee 
Hence 
\be\label{eq-gets-larger}
|\pi_p(A_{i+1}'')|=|{\rm gr}_{n_{i+1}',p}(R_{i+1}';\ocal)|>|{\rm gr}_{0,p}(R_{i+1}';\ocal)|=|\pi_p(A_i)|.
\ee
By \eqref{eq-gets-larger}, as in Step 1, we have that $\pi_{\pfr}(\gen{A_{i+1}''}{C_0})$ is a subfield of $\f$. Let $A_{i+1}:=\gen{A_{i+1}''}{C_0}$. By \eqref{eq-auxiliary-set}
\be\label{eq-in-sum-product-set}
s(p^{m_i})^{C_{i}'}s'(p^{n_{i+1}'})^{C_0}A_{i+1}\subseteq \gen{A_{i}}{C_0C_{i}'\overline{C}_i},
\ee
and $0,1\in A_{i+1}$. By \eqref{eq-next-set-case-1} and \eqref{eq-auxiliary-set} we get that 
\begin{align} 
\label{eq-lower-bound}
	s(p^{m_i})^{C_{i}'}s'(p^{n_{i+1}'})^{C_0}A_{i+1}& 
	\supseteq s(p^{m_i})^{C_{i}'}s'(p^{n_{i+1}'})^{C_0-1}(A_{i+1}'\cap p^{n_{i+1}'}\ocal)
	\\
\notag
& \supseteq s(p^{m_i})^{C_i'-1}s'(p^{n_{i+1}'})^{C_0-1}(\gen{A_i}{2}\cap p^{m_i+n_{i+1}'}\ocal);
\end{align}
and by the assumption $n_{i+1}'\le n_i+1$. Therefore by \eqref{eq-gets-larger}, \eqref{eq-in-sum-product-set}, and \eqref{eq-lower-bound}, we get that $A_{i+1}$, $C_i:=C_0C_i'\overline{C}_i$, $\alpha_i:=s(p^{m_i})^{C_{i}'-1}s'(p^{n_{i+1}'})^{C_0-1}$, and $\beta_i:=s(p^{m_i})s'(p^{n_{i+1}'})$ satisfy Property (P4) (Connection between sets).

Notice that  for any integer $j$ in $[Ne\delta_0^{2},Ne]$  we get that
\begin{align}\label{eq-lower-bound-box-dimension}
|\pi_{\pfr^j}(A_{i+1})|
&\ge |\pi_{\pfr^j}(s(p^{m_i})^{-1}s'(p^{n_{i+1}'})^{-1}(\gen{A_i}{2}\cap p^{m_i+n_{i+1}'}\ocal))|
&\text{ by \eqref{eq-lower-bound} }
\\
\notag 
&= |\pi_{\pfr^{j+m_ie+n_{i+1}'e}}(\gen{A_i}{2}\cap p^{m_i+n_{i+1}'}\ocal)|
& 
\\
\notag
&\ge |\pi_{\pfr^{j+m_ie+n_{i+1}'e}}(A_i)|/|\pi_{p^{m_i+n_{i+1}'}}(R_i)|
&
\\
\notag
&\ge |\pi_{\pfr^j}(A_i)|/|\f|^{(m_i+n_i+1)e}
&
\\
\notag
&\ge |\f|^{j\vare_1/2^i-N\delta_0^2 e}\ge|\f|^{(j\vare_1/2^{i+1})+Ne(\delta_0\vare_1/2^{i+1}-\delta_0^2)}
&\text{ by (P2-b) for $A_i$ and $n_i\le N\delta_0^2/2$}
\\
\notag
&\ge |\f|^{(j\vare_1/2^{i+1})+Ne\delta_0(\vare_1/2^{d+1}-\delta_0)} \ge |\f|^{j\vare_1/2^{i+1}}
&\text{ by \eqref{eq-upper-bound-number-of-steps} and $\delta_0\le \vare_1/2^{d+1}$. }
\end{align}
Hence by \eqref{eq-lower-bound-box-dimension}, $A_{i+1}$ satisfies Property (P2) (Set conditions).

{\bf Step 4.} In this step, we define $A_{i+1}$, $C_i$, $\alpha_i$, and $\beta_i$ under the assumption that $n_{i+1}'> n_i+1$.

Since $(n_i+1)/2\le n_i-m_i$, by a similar argument as in the beginning of Step 3 (see the argument for \eqref{eq-next-to-the-larger-grade}), we get that 
\be\label{eq-getting-the-larger-grade-initiation}
\pi_{p^{n_i+1}}(R_{i+1}')=\pi_{p^{n_i+1}}(\gen{A'_{i+1}}{2});
\ee  
and so by \eqref{eq-auxiliary-set} we have
\be\label{eq-generating-large-subset}
\pi_{p^{n_i+1}}(s(p^{m_i})^2 R_{i+1}')=\pi_{p^{n_i+1}}(\gen{s(p^{m_i})A'_{i+1}}{2})\subseteq \pi_{p^{n_i+1}}(\gen{A_i}{2\overline{C}_i}).
\ee
By \eqref{eq-initiation-of-the-next-step} and the assumption that $n_{i+1}'>n_i+1$, we have $|{\rm gr}_{n_i,p}(R_{i+1}';\ocal)|=|\pi_p(R_i)|$; and so using  
$|{\rm gr}_{n_i,p}(R_i;\ocal)|>|\pi_p(R_i)|$ and $R_i$ is generated by $A_i$, we deduce that $\pi_{p^{n_i+1}}(A_i)\not\subseteq \pi_{p^{n_i+1}}(R_{i+1}')$. Suppose $a_i\in A_i$ is such that 
\be\label{eq-new-element-1}
\pi_{p^{n_i+1}}(a_i)\not\in \pi_{p^{n_i+1}}(R_{i+1}').
\ee
On the other hand, by \eqref{eq-initiation-of-the-next-step}, we have $\pi_{p^{n_i-m_i}}(R_i)=\pi_{p^{n_i-m_i}}(R_{i+1}')$. Let $n_i''$ be the largest positive integer such that $\pi_{p^{n_i''}}(a_i)\in \pi_{p^{n_i''}}(R_{i+1}')$; and so 
\be\label{eq-control-the-level-of-enlarging}
n_i-m_i\le n_i''\le n_i, \text{ and there is $r_{i+1}'\in R_{i+1}'$ such that $v_{\pfr}(a_i-r'_{i+1})/e=n_i''$. }
\ee
The way that $n_i''$ was chosen and the second part of \eqref{eq-control-the-level-of-enlarging} imply that 
\be\label{eq-enlarging-grade}
\pi_{p^{n_i''+1}}(a_i-r'_{i+1})\in {\rm gr}_{n_i'',p}(\ocal;\ocal)\setminus {\rm gr}_{n_i'',p}(R_{i+1}';\ocal).
\ee
On the other hand, by \eqref{eq-generating-large-subset}, we have
\be\label{eq-new-element-2}
s(p^{m_i})^2(a_i-r'_{i+1})=s(p^{m_i})^2a_i-s(p^{m_i})^2r'_{i+1}\in \textstyle\prod_2\gen{A_i}{\overline{C}_i}A_i-\gen{A_i}{2\overline{C}_i} \subseteq \gen{A_i}{4\overline{C}_i+1}.
\ee
By Corollary~\ref{cor-grade-shift}, we get that the map induced by the multiplication by $s(p^{m_i})^2$ is an injection from ${\rm gr}_{n_i'',p}(\ocal;\ocal)$ to $\pi_{p^{n_i''+2m_i+1}}(\ocal)$. Let us denote this map by $x\mapsto [s(p^{m_i})^2]x$. Hence by \eqref{eq-enlarging-grade} we have 
\begin{align}
	\label{eq-old-and-new}
|[s(p^{m_i})^2]({\rm gr}_{n_i'',p}(R_{i+1}';\ocal))|
&=|{\rm gr}_{n_i'',p}(R_{i+1}';\ocal)|\text{, and }
\\
\notag
[s(p^{m_i})^2](\pi_{p^{n_i''+1}}(a_i-r'_{i+1})) &\not\in [s(p^{m_i})^2]({\rm gr}_{n_i'',p}(R_{i+1}';\ocal)).
\end{align}
On the other hand, similar to \eqref{eq-generating-large-subset},  by \eqref{eq-auxiliary-set} and \eqref{eq-getting-the-larger-grade-initiation} we deduce that
\begin{align}
\label{eq-generating-the-old-portion}
[s(p^{m_i})^2]({\rm gr}_{n_i'',p}(R_{i+1}';\ocal))
&
\subseteq \pi_{p^{n_i''+1+2m_i}}(s(p^{m_i})^2\gen{A'_{i+1}}{2})\\
\notag
&
\subseteq \pi_{p^{n_i''+1+2m_i}}(\gen{s(p^{m_i}A'_{i+1}}{2})\\
\notag
&
\subseteq \pi_{p^{n_i''+1+2m_i}}(\gen{A_{i}}{2\overline{C}_i}).
\end{align}
Therefore by \eqref{eq-new-element-2}, \eqref{eq-old-and-new}, and \eqref{eq-generating-the-old-portion}, we have that 
\be\label{eq-grade-get-larger-bounded-steps}
|\pi_{p^{n_i''+1+2m_i}}(\gen{A_{i}}{4\overline{C}_i+1}\cap p^{n_i''+2m_i}\ocal)|>|{\rm gr}_{n_i'',p}(R_{i+1}';\ocal)|=|\pi_p(R_{i+1}')|=|\pi_p(R_i)|;
\ee
here we have used $n_i''\le n_i< n_i'$ and \eqref{eq-initiation-of-the-next-step}. The rest of the argument is similar to Step 3. 

Let $A_{i+1}'':=s''(p^{n_i''+2m_i})^{-1}(\gen{A_{i}}{4\overline{C}_i+1}\cap p^{n_i''+2m_i}\ocal)$ where $s''(p^{n_i''+2m_i})$ is an element of $\gen{A_{i}}{2\overline{C}_i}$ such that $\pi_{p^{n_i''+1+2m_i}}(s''(p^{n_i''+2m_i}))=1$ (and there is such an element because of \eqref{eq-generating-the-old-portion}). Hence $0,1\in A_{i+1}''$, and 
\be\label{eq:larger-mod-p}
|\pi_p(A_{i+1}'')|>|\pi_p(R_i)|=|\pi_p(A_i)|.
\ee
By \eqref{eq:larger-mod-p}, as in Step 1, we have that $\pi_{\pfr}(\gen{A_{i+1}''}{C_0})$ is a subfield of $\f$. Let $A_{i+1}:=\gen{A_{i+1}''}{C_0}$. Hence
\be\label{eq-connection-with-previous-set}
s''(p^{n_i''+2m_i})^{C_0-1} (\gen{A_i}{2}\cap p^{n_i''+2m_i}\ocal) \subseteq  
s''(p^{n_i''+2m_i})^{C_0} A_{i+1}\subseteq \gen{A_{i}}{C_0(4\overline{C}_i+1)}.
\ee
Hence $A_{i+1}$, $C_i:=C_0(4\overline{C}_i+1)$, $\alpha_i:=s''(p^{n_i''+2m_i})^{C_0-1}$, and $\beta_i:=s''(p^{n_i''+2m_i})$ satisfy Property (P4) (Connection between sets).

Notice that, for any integer $j$ in $[Ne\delta_0^{2},Ne]$  we get that
\begin{align}
\notag
|\pi_{\pfr^j}(A_{i+1})|
&\ge |\pi_{\pfr^j}(s''(p^{n_i''+2m_i})^{-1} (\gen{A_i}{2}\cap p^{n_i''+2m_i}\ocal) )|
&\text{ by \eqref{eq-lower-bound} }
\\
\notag 
&= |\pi_{\pfr^{j+n_i''e+2m_ie}}(\gen{A_i}{2}\cap p^{n_i''+2m_i}\ocal)|
& 
\\
\notag
&\ge |\pi_{\pfr^{j+n_i''e+2m_ie}}(A_i)|/|\pi_{p^{n_i''+2m_i}}(R_i)|
&
\\
\notag
&\ge |\pi_{\pfr^j}(A_i)|/|\f|^{(n_i''+2m_i)e}
&
\\
\notag
&\ge |\f|^{j\vare_1/2^i-N\delta_0^2 e}\ge|\f|^{(j\vare_1/2^{i+1})+Ne(\delta_0\vare_1/2^{i+1}-\delta_0^2)}
&\text{ by (P2-b) for $A_i$ and $n_i''\le n_i\le N\delta_0^2/2$ }
\\
\label{eq-lower-bound-box-dimension-2}
&\ge |\f|^{(j\vare_1/2^{i+1})+Ne\delta_0(\vare_1/2^{d+1}-\delta_0)} \ge |\f|^{j\vare_1/2^{i+1}}
&\text{ by \eqref{eq-upper-bound-number-of-steps} and $\delta_0\le \vare_1/2^{d+1}$. }
\end{align}
Hence by \eqref{eq-lower-bound-box-dimension-2}, $A_{i+1}$ satisfies Property (P2) (Set conditions).

{\bf Step 5.} Finishing the proof. 

By \eqref{eq-upper-bound-number-of-steps}, the above process stops in $i_0\le d+1$ steps, and we get $n_{i_0}\ge N\delta_0^2/2$. By Property (P4), inductively we get that for any integer $j$ in $[0,i_0+1]$ 
\[
\gamma_jA_j\subseteq \gen{A}{\prod_{i=0}^{j}C_i},
\]
for some $\gamma_j\in \ocal$ such that $v_{\pfr}\ll_{d,\vare_1} N\delta_0^2$. Hence
\be\label{eq-sum-product-set-final}
\gamma_{i_0}A_{i_0} \subseteq \gen{A}{\prod_{i=0}^{i_0-1}C_i}.
\ee
By Property (P3) we have 
\[
\pi_{p^{n_{i_0}}}(\gen{A_{i_0}}{C_{i_0}})\supseteq \pi_{p^{n_{i_0}}}(p^{\lceil n_{i_0} \vare_2 \rceil}R_{i_0});
\]
and so by \eqref{eq-sum-product-set-final} 
\be\label{eq-ring-generation-final}
\pi_{\pfr^{n_{i_0}e+C_{i_0}v_{\pfr}(\gamma_{i_0})}}(\gamma_{i_0}^{C_{i_0}}p^{\lceil n_{i_0} \vare_2 \rceil}R_{i_0})\subseteq \pi_{\pfr^{n_{i_0}e+C_{i_0}v_{\pfr}(\gamma_{i_0})}}(\gen{\gamma_{i_0}A_{i_0}}{C_{i_0}})\subseteq \pi_{\pfr^{n_{i_0}e+C_{i_0}v_{\pfr}(\gamma_{i_0})}}(\gen{A}{\prod_{i=0}^{i_0}C_i}).
\ee
Let $n:=n_{i_0}e+C_{i_0}v_{\pfr}(\gamma_{i_0})$, $m:=\lceil n_{i_0} \vare_2 \rceil e+C_{i_0}v_{\pfr}(\gamma_{i_0})$, $a:=p^{\lceil n_{i_0} \vare_2 \rceil}\gamma_{i_0}^{C_{i_0}}$, and $C:=\prod_{i=0}^{i_0}C_i$. Then
\begin{align}
\label{eq-exponent-conditions-final}
n\ll_{d,\vare_1,\delta_0} N
&
\ll_{d,\vare_1,\delta_0} n-m,
& \text{ (Exponent conditions) }
\\
\label{eq-bounded-generation-final}
\pi_{\pfr^n}(R_{i_0}a)
&
\subseteq \pi_{\pfr^n}(\gen{A}{C}), 
\hspace{1cm}
v_{\pfr}(a)=m,\h C\ll_{d,\vare_1} 1,
&
\text{ (Bounded generation) }
\\
\label{eq-grades-equality-final}
|{\rm gr}_{0,p}(R_{i_0};\ocal)|&=\dots=|{\rm gr}_{n-m,p}(R_{i_0};\ocal)|\ge |\pi_p(R)|.
&
\text{ (Grades equality) }
\end{align}
Since $N\gg_{d,\vare_1,\delta_0} 1$, by \eqref{eq-grades-equality-final} (Grades equality) and Step 1 of proof of proposition~\ref{prop-bounded-generation-for-equal-grades}, there is a subfield $K_0$ of $K$ with ring of integers $\ocal_0$ such that
\[
\pi_{\pfr^{n-m-4}}(R_{i_0})=\pi_{\pfr^{n-m-4}}(\ocal_0);
\]
and so 
\be\label{eq-ring-of-integers-final}
\pi_{\pfr^{n-4e}}(\ocal_0 a)=\pi_{\pfr^{n-4e}}(R a)\subseteq \pi_{\pfr^{n-4e}}(\gen{A}{C}),
\ee
and $|\pi_{p}(\ocal_0)|\ge |\pi_p(R)|\ge |\f|^{e\vare_1}$. And the claim follows by \eqref{eq-exponent-conditions-final}, \eqref{eq-bounded-generation-final}, \eqref{eq-grades-equality-final}, and \eqref{eq-ring-of-integers-final}.
\end{proof}

\begin{proof}[Proof of Theorem~\ref{t:ThickSegment}]
{\bf Step 1.} (Small residue field) In~\cite[Appendix]{BG}, Bourgain has essentially proved Theorem~\ref{t:ThickSegment} for a given fixed prime $p$, with two short comings: (1) in~\cite[Appendix]{BG}, it is assumed that $p$ is a fixed prime that is {\em large} compared to $d$; (2) In~\cite[Proposition 3.3]{BG}, Bourgain only claims that a thick {\em $\bbz_p$-segment} at certain scale can be generated in bounded number of steps.

Going through the argument in \cite[Appendix]{BG}, one can see that the largeness of $p$ compared to $d$ is used to ensure that $K$ is not a widely ramified extension of $\bbq_p$. In turn, this is used in \cite[Section A.6]{BG} to describe the structure of certain subrings of $\ocal$. In fact~\cite[Section A.6]{BG} is the only place, where the largeness of $p$ is used. So Proposition~\ref{prop:detecting-ring-of-integers-via-grading-structure} and Step 1 of proof of Proposition~\ref{prop-bounded-generation-for-equal-grades} remove this obstruction. 

Going through the proof of \cite[Proposition 3.3]{BG} in \cite[Section A.5]{BG}, one can observe that one gets the stronger version as it is presented in Lemma~\ref{lem-bounded-generation-large-mod-all-powers-till-N}.

{\bf Step 2.} (Large residue field) As in \cite{BG} (see also~\cite[Section 2.3]{SG1}), there is $A'\subseteq A$ such that $\pi_{\pfr^N}(A')$ is an $(m_0,\ldots,m_{N-1})$-regular set, and 
\[
|\pi_{\pfr^N}(A')|\ge \frac{|A|}{(2\log |\f|)^N}\ge |\f|^{N\vare/2},
\]
for $|\f|\gg_{\vare} 1$. Let 
\[
\overline{n}:=\max\{k|\h |\pi_{\pfr^k}(A')|=m_0\cdots m_{k-1}<|\f|^{k\vare/4}\}.
\]
So for any $\overline{n}+1\le l\le N$, we have
\be\label{e:LargeLevels}
\prod_{i=\overline{n}}^{l-1} m_i\ge \left(\prod_{i=0}^{l-1} m_i\right)\left(\prod_{i=0}^{\overline{n}-1} m_i\right)^{-1} \ge |\f|^{l\vare/2} |\f|^{\overline{n} \vare/4}\ge |\f|^{(l-\overline{n})\vare/4}.
\ee
We also have
\be\label{e:Thickness}
|\f|^{N-\overline{n}}\ge \prod_{i=\overline{n}}^{N-1} m_i\ge \left(\prod_{i=0}^{N-1} m_i\right)\left(\prod_{i=0}^{\overline{n}-1} m_i\right)^{-1} \ge |\f|^{N\vare/2} |\f|^{\overline{n} \vare/4}.
\ee
Therefore $M:=N-\overline{n}\ge N\vare/2+\overline{n}\vare/4=N\vare/2-(N-M)\vare/4\ge N\vare/4$. So there is a subset $B\subseteq \ocal$ such that
\begin{enumerate}
\item $\pfr^{N'}B\subseteq A$ for $N'\le N(1-\vare/4)$,
\item $\pi_{\pfr^M}(B)$ is a regular set for $M\ge N\vare/4$,
\item $|\pi_{\pfr^i}(B)|\ge |\f|^{i\vare/4}$ for any $1\le i\le M$.
\end{enumerate} 
In particular, there is $\lambda\in (B-B)\cap \ocal^{\times}$. So replacing $A$ with $\lambda^{-1}A$ and $B$ with $\lambda^{-1}(B-B)$, we can and will assume that $B$ contains $0$ and $1$. Notice that proving the claim of Theorem~\ref{t:ThickSegment} for a unit multiple $\lambda^{-1}A$ of $A$ implies the claim for $A$. 

Hence by Lemma~\ref{lem-bounded-generation-large-mod-all-powers-till-N} there are positive integers $M_1$ and $M_2$, $a\in \ocal$, and a subfield $K_0$ of $K$ with ring of integers $\ocal_0$ such that
\begin{align}
\label{eq-bounded-generation-final-proof}
\pi_{\pfr^{M_2}}(\ocal_0 a)
&
\subseteq \pi_{\pfr^{M_2}}(\gen{B}{C}),
\\
\label{eq-control-exponents-final-proof}
M_2 \ll_{d,\vare} M 
&
\ll_{d,\vare} M_2-M_1,\hspace{.5cm} v_{\pfr}(a)=M_2,
\\
\label{eq-control-box-dimension}
|\pi_p(\ocal_0)|
&
\ge
|\pi_p(\ocal)|^{\vare/4}. 
\end{align}
And so by \eqref{eq-bounded-generation-final-proof} we get
\[
\pi_{\pfr^{M_2+CN'}}(\ocal_0 (\pfr^{N'C}a))\subseteq 
\pi_{\pfr^{M_2+CN'}}(\gen{\pfr^{N'}B}{C}) \subseteq 
\pi_{\pfr^{M_2+CN'}}(\gen{A}{C}).
\]
Let $N_1:=M_1+CN'$ and $N_2:=M_2+CN'$. Hence by \eqref{eq-control-exponents-final-proof}, we get that $N_1$ and $N_2$ satisfy \eqref{eq:Thickness} for suitably chosen $\delta$ and enlarging $C$ if necessary. And the claim follows. 
\end{proof}

\end{document}